\documentclass[11pt]{article} 
\RequirePackage{amsthm,amsmath,amsfonts,amssymb}
\RequirePackage[numbers]{natbib}
\RequirePackage[colorlinks,citecolor=blue,urlcolor=blue, hypertexnames=false]{hyperref}
\RequirePackage{graphicx}
\usepackage{graphicx,psfrag}
\usepackage{mathtools}
\usepackage{bm}
\usepackage{bbm}
\usepackage{enumerate} 
\usepackage{enumitem}
\usepackage[a4paper, 
            left=1.1in,right=1.1in,top=1in,bottom=1.2in,%
            footskip=.7in]{geometry}

\def\ind{\mathbbm{1}}   %indicator function

\newcommand{\bbR}{\mathbb{R}} % real number
\newcommand{\bbN}{\mathbb{N}} % number
\newcommand{\bbE}{\mathbb{E}} % expectation
\newcommand{\bbS}{\mathbb{S}} % permutation
\newcommand{\bbP}{\mathbb{P}} % prob measure
\newcommand{\bbQ}{\mathbb{Q}} % prob measure

\newcommand{\bI}{\mathbf{I}} % identity transition matrix 
\newcommand{\bP}{\mathbf{P}} % transition probability
\newcommand{\bK}{\mathbf{K}} % proposal distribution
\newcommand{\cT}{\mathcal{T}}% canonical path ensemble
\newcommand{\GAP}{\mathrm{Gap}} % spectral gap
\newcommand{\hit}{\mathrm{hit}}
\newcommand{\cev}[1]{\reflectbox{\ensuremath{\vec{\reflectbox{\ensuremath{#1}}}}}}
\def\EDGE{\mathrm{Edge}}

\def\din{d_{\mathrm{in}}}              % maximum indegree
\def\dout{d_{\mathrm{out}}}            % maximum outdegree
\def\post{\pi_n} 			 % posterior distribution  

\def\Pa{\mathrm{Pa}} %% parent set. i still feel Pa is better, though I also use S at some places. 
\def\Ch{\mathrm{Ch}}  %% child set 
\def\Cpen{C_{\mathrm{pen}}} %%% penalty constant
\def\Cerr{C_{\mathrm{err}}} %% constant for covariance matrix estimation error 

\newcommand{\cE}{\mathcal{E}}  % SET
\newcommand{\cA}{\mathcal{A}}

\newcommand{\lmin}{\lambda_{\mathrm{min}}}
\newcommand{\lmax}{\lambda_{\mathrm{max}}}
\newcommand{\smin}{\sigma_{\mathrm{min}}}

\def\vmin{\underline{\nu}}
\def\vmax{\overline{\nu}}
\newcommand{\cN}{\mathcal{N}}  %neighborhood
\newcommand{\proj}{\Phi}  % projection
\newcommand{\oproj}{\Phi^{\perp}}  % projection
\DeclareMathOperator\argmax{arg\,max}

\DeclareMathOperator\diag{diag}

\DeclarePairedDelimiter{\norm}{\lVert}{\rVert}
\DeclarePairedDelimiter{\OPnorm}{\lVert}{\rVert_{\mathrm{op}}}
\DeclarePairedDelimiter{\TV}{\lVert}{\rVert_{\mathrm{TV}}}

\def\sX{\mathsf{X}} %% random variables
\def\dag{\mathcal{G}_p}
\def\cpdag{\mathcal{C}_p}
\def\model{\mathcal{M}_p}
\def\Vert{V} 
\def\Edge{E}
\def\chol{\mathcal{D}_p}
\newcommand{\EG}[1]{[#1]} %% MEC generated by a dag. 	  
\def\CI{\mathcal{CI}}
\newcommand{\indep}{\mathrel{\text{\scalebox{1.07}{$\perp\mkern-10mu\perp$}}}}

\newcommand{\notindep}{\not\!\perp\!\!\!\!\perp}
\def\Z{X}  
\def\X{X} 

\def\ty{y_0}

\def\ses{\omega^*}
\def\bmin{\beta_{\rm{min}}}
\def\betamin{\beta_{\rm{min}}} 
\def\Cbeta{C_\beta}
\def\HD{\mathrm{Hd}}

\def\DS{\Delta^*} 

\def\NM{N}   %% normal distribution 
\def\Tmix{T_{\mathrm{mix}}}  %% mixing time 
\def\tG{\tilde{G}}

\def\cS{\mathcal{S}} %% a shorthand for \model used in the supp
\def\RWGES{RW-GES} %% name for our algorithm 
\def\cZ{\mathcal{Z}} %% set of all possible error vectors... run out of notation...
\def\bbV{\mathbb{V}} %% var selection problems
\def\adds{\cN_{\mathrm{add}}}
\def\dels{\cN_{\mathrm{del}}}
\def\swaps{\cN_{\mathrm{swap}}}
\def\revs{\cN_{\mathrm{rev}}} 
  
\def\Cb{C_\beta} %% Cb for cpdag learning. 

\def\bmin{\underline{b}}

\def\cNc{\mathcal{N}_{\mathcal{C}}}
\def\tcE{\tilde{\cE}}
\def\cO{\mathcal{O}}

\def\nb{\mathcal{N}_{\mathrm{ads}}}

\def\score{\psi}
\def\scorej{\psi_j}
\newcommand{\EGG}[1]{\mathrm{EG}(#1)}

\def\adj{\mathrm{Ad}}
\def\und{\mathrm{Un}}
\def\ges{\mathcal{O}_{\mathrm{ges}}}

\def\tG{\tilde{G}}

\def\nbg{\mathcal{N}_{\rm{ads}}}

\def\hSigma{\hat{\Sigma}} 

\theoremstyle{plain}
\newtheorem{lemma}{Lemma}
\newtheorem{corollary}{Corollary}

\newtheorem{theorem}{Theorem}

\theoremstyle{remark}
\newtheorem{definition}{Definition}
\newtheorem{example}{Example}
\newtheorem{assumption}{Assumption}[section]
\newtheorem{condition}{Condition}
\newtheorem{remark}{Remark}
\newtheorem{alg}{Algorithm}[section]

\title{Complexity analysis of Bayesian learning of high-dimensional DAG models and their equivalence classes}
\author{Quan Zhou and Hyunwoong Chang  \medskip \\   Department of Statistics, Texas A\&M University}
\date{}

\begin{document}

\maketitle  

%\title{Complexity analysis of Bayesian learning \\ of high-dimensional DAG models and their \\ equivalence classes}  
%\runtitle{Complexity analysis of Bayesian structure learning}

%\begin{aug}
%%%%%%%%%%%%%%%%%%%%%%%%%%%%%%%%%%%%%%%%%%%%%%
%%Only one address is permitted per author. %%
%%Only division, organization and e-mail is %%
%%included in the address.                  %%
%%Additional information can be included in %%
%%the Acknowledgments section if necessary. %%
%%%%%%%%%%%%%%%%%%%%%%%%%%%%%%%%%%%%%%%%%%%%%%
%\author[A]{\fnms{Quan}~\snm{Zhou}} % \ead[label=e1]{quan@stat.tamu.edu}}
%\and
%\author[A]{\fnms{Hyunwoong}~\snm{Chang}} %\ead[label=e2]{hwchang@stat.tamu.edu}}
%%%%%%%%%%%%%%%%%%%%%%%%%%%%%%%%%%%%%%%%%%%%%%
%% Addresses                                %%
%%%%%%%%%%%%%%%%%%%%%%%%%%%%%%%%%%%%%%%%%%%%%%
%\address[A]{Department of Statistics, Texas A\&M University} %\printead[presep={,\ }]{e1} }
%\end{aug}

\begin{abstract}
Structure learning via MCMC sampling is known to be very challenging because of the enormous search space and the existence of Markov equivalent DAGs. 
Theoretical results on the mixing behavior are lacking. 
In this work, we prove the rapid mixing of a random walk Metropolis-Hastings algorithm, which reveals that the complexity of Bayesian learning of sparse equivalence classes grows only polynomially in $n$ and $p$, under some high-dimensional assumptions. 
A series of high-dimensional consistency results is obtained,  including the strong selection consistency of an empirical Bayes model for structure learning. 
Our proof is based on two new results. 
First, we derive a general mixing time bound on finite state spaces, which can be applied to  local MCMC schemes for other model selection problems. 
Second, we construct high-probability search paths on  the space of equivalence classes with node degree constraints by proving a combinatorial property of DAG comparisons. 
Simulation studies on the proposed MCMC sampler are conducted to illustrate the main theoretical findings. \medskip   \\ 
\textit{Keywords:} finite Markov chains; greedy equivalence search (GES); locally informed proposals; Poincar\'{e} inequality; random walk Metropolis-Hastings; rapid mixing; strong selection consistency. 
\end{abstract}

%\begin{keyword} 
%\kwd{finite Markov chains}
%\kwd{greedy equivalence search (GES)}
%\kwd{locally informed proposals}
%\kwd{Poincar\'{e}-type inequality}
%\kwd{random walk Metropolis-Hastings}
%\kwd{rapid mixing}
%\kwd{strong selection consistency}
%\end{keyword}

\section{Introduction}\label{sec:intro}

\subsection{Gaussian DAG models and equivalence classes} \label{sec:setup}
A directed acyclic graph (DAG) encodes a set of conditional independence (CI) relations among node variables, which can be read off using  the ``d-separation'' criterion~\citep{pearl1988probabilistic}. 
Structure learning of DAG models from observational data plays a fundamental role in causal inference and has found many applications in machine learning and statistical data analysis~\citep{koller2009probabilistic}. 
In genomics, for example, DAG is a convenient device for conducting pathway analysis and inferring interactions among genes or proteins~\citep{maathuis2010predicting, gao2015learning}. 

Two  DAGs with different edge sets can encode  the same set of CI relations, in which case we say both belong to the same (Markov) equivalence class. For example, the DAGs $i \rightarrow j \rightarrow k$ and $i \leftarrow j \rightarrow k$ are Markov equivalent: both encode only one CI relation $i \indep k \mid j$ (i.e., $i, k$ are independent given $j$). But they are not Markov equivalent to  $i \rightarrow j \leftarrow k$, since the latter encodes only one CI relation $i \indep k$.   
A Gaussian DAG model represents a set of multivariate normal distributions that satisfy the CI constraints encoded by the DAG. 
Due to normality, Markov equivalence further implies distributional equivalence~\citep{geiger2002parameter}, and thus observational data alone cannot distinguish between Markov equivalent DAGs; this is a main challenge in devising efficient structure learning algorithms~\citep{chickering2002learning}.  

This paper is chiefly concerned with the following problem: given $n$ i.i.d. observations from a $p$-variate DAG-perfect normal distribution, estimate the equivalence class of the underlying DAG model.  
This is a model selection problem  where the model space is a collection of $p$-vertex  equivalence classes. 
We are most interested in high-dimensional settings where $p$ grows much faster than $n$ and  the true DAG model is sparse. 

The structure learning problem can be greatly simplified if the topological ordering of the variables is known. 
By ordering, we mean a permutation $\sigma \in \bbS^p$, where  $\bbS^p$ denotes the symmetric group on $\{1, \dots, p\}$, such that for any $i < j$,  an edge connecting $ \sigma(i)$  and $\sigma(j) $ is always directed as $\sigma(i) \rightarrow \sigma(j)$.  
Such a total ordering always exists, but may not be unique, for a DAG due to acyclicity.  
For example, for   $1 \rightarrow 3 \leftarrow 2$,  the ordering can be either $(1, 2, 3)$ or $(2, 1, 3)$. 
Any two different DAGs that share a same ordering cannot be Markov equivalent. This can be proved by contradiction: if the two DAGs are Markov equivalent, they must have the same skeleton~\citep{verma1991equivalence}, but the ordering uniquely determines the directions of all edges implying that the two DAGs must be the same. 
Henceforth, we refer to the problem as DAG selection when the ordering is known and reserve the term ``structure learning'' for learning equivalence classes when the ordering is unknown; the latter is the focus of this paper. 
 
\subsection{Algorithms for Bayesian structure learning}\label{sec:algo} 
Most Bayesian structure learning methods aim to produce a posterior distribution of the DAG model or its equivalence class, which can be further used for making inference on quantities of interest via model averaging.   
To numerically approximate the posterior distribution, Markov chain Monte Carlo (MCMC) sampling is often invoked, and existing MCMC methods differ from each other mainly in three respects: the state space, the set of local operators, and the proposal scheme.  
The local operators decide which states the sampler may move to in the next iteration (i.e., they define the ``neighborhood'' of each state). 
The proposal scheme refers to how the proposal probabilities of these neighboring states are assigned. 
Most existing algorithms use either random walk Metropolis-Hastings (MH) or Gibbs schemes, but we note that informed proposal schemes recently proposed in~\citet{zanella2020informed} and~\citet{zhou2021dimension} can be applied as well. 

There are three popular choices of the state space: states can be DAGs, equivalence classes or orderings.
The famous ``structure MCMC'' sampler is a random walk MH algorithm that searches the DAG space using addition, deletion and reversal of single edges~\citep{madigan1995bayesian, giudici2003improving}. It is straightforward to implement (one only needs to check acyclicity when proposing local moves) but may not be efficient since the sampler can spend a lot of time traversing large equivalence classes. Various methods have been proposed to improve the performance by using more complicated local operators~\citep{grzegorczyk2008improving, su2016improving} or blocked Gibbs schemes~\citep{goudie2016gibbs}. 
Directly searching the space of equivalence classes seems more efficient, but a major challenge  is to construct  a proper set of local operators~\citep{andersson1997characterization, perlman2001graphical, munteanu2001eq, chickering2002learning, pena2007approximate}; 
see~\citet{madigan1996bayesian} and~\citet{castelletti2018learning} for MCMC samplers defined on the space of equivalence classes.  
Order MCMC methods~\citep{friedman2003being, ellis2008learning, agrawal2018minimal} target a posterior   distribution on the order space $\bbS^p$. They are motivated by the observation that given ordering,  the conditional posterior distribution of DAGs can be evaluated relatively easily.  
More sophisticated MCMC schemes can be built by using partial orderings~\citep{niinimaki2011partial, kuipers2017partition}.  
It should be noted that the choice of the prior distribution typically depends on the state space, which results in essentially different posterior distributions on the three state spaces (see Section~\ref{sec:mix.dag.order}).   
We will focus on the space of equivalence classes.

In principle, by treating the logarithm of the posterior probability as a scoring criterion,   deterministic score-based search algorithms can also be used to find the structure that maximizes the score (i.e.,  the maximum a posteriori estimate). 
This approach appears less popular in the  Bayesian structure learning literature, probably because it cannot quantify the uncertainty in estimation. 
One of the most important score-based algorithms is the greedy equivalence search (GES) proposed by~\citet{meek1997graphical} and~\citet{chickering2002optimal}, a two-stage greedy search algorithm defined on the space of equivalence classes.  
\citet{nandy2018high} were the first to prove the high-dimensional consistency of GES (i.e., the search returns the true equivalence class with high probability for sufficiently large $n$) using an assumption called strong faithfulness. Though it is known that strong faithfulness is very restrictive~\citep{uhler2013geometry}, such conditions appear to be necessary for proving high-dimensional consistency results for many search methods~\citep{kalisch2007estimating}. 
We refer readers to~\citet{drton2017structure} and~\citet{scutari2019learns} for other scored-based structure learning methods. 

\subsection{Overview of the paper}\label{sec:contribution}
While many MCMC methods for structure learning have been proposed, to our knowledge, no theoretical result on the mixing time is available. 
This is probably because the structure of the state space is highly complicated.  
The primary goal of this paper is to fill the gap by deriving non-asymptotic mixing time bounds. 
We find that structure MCMC and order MCMC methods are, unfortunately, hard to analyze due to the technical difficulty in bounding sizes of  equivalence classes. The equivalence class sampler of~\citet{castelletti2018learning} uses six graph operators to move between equivalence classes~\citep{he2013reversible}, but we can explicitly construct  slow mixing examples for this sampler with fixed $p$.  
This motivates us to propose our own equivalence class sampler,  \RWGES{},  which uses a random walk proposal scheme that mimics and generalizes the local moves employed by GES.  We prove a high-dimensional rapid mixing result for the \RWGES{} sampler, which essentially says that, under some conditions, the number of iterations needed to find the true equivalence class grows only polynomially in $n$ and $p$ with high probability. 
The proof consists of three steps, which we now explain separately. 

In Section~\ref{sec:mix}, we first develop a general theory on the complexity of local MCMC algorithms for model selection. This section is self-contained and of considerable independent interest.  
We build a weighted path argument and use a Poincar\'{e}-type inequality~\citep{kahale1997semidefinite} to obtain a novel, generally applicable mixing time bound under a unimodal assumption; see Condition~\ref{cond:unimodal} and Theorem~\ref{th:path.better}.  This result can be applied to other model selection problems such as variable selection and stochastic block models. It sharpens the existing mixing time bounds for random walk MH algorithms in the literature~\citep{yang2016computational, zhuo2021mixing} and can be utilized to derive theoretical guarantees for locally informed MH algorithms~\citep{zanella2020informed}.    
Our theory also reveals a link between optimization and sampling: if for some model selection problem there is a greedy local search with consistency guarantee, it is hopeful that, with some modifications, we may convert the greedy algorithm to a local MH sampler that has provable rapid mixing property.  
In general, rapid mixing is more difficult to prove and more informative than the consistency of a greedy search, since the former characterizes the overall complexity of the algorithm and requires an analysis of the local posterior landscape in the whole state space.

The \RWGES{} sampler for structure learning is formally introduced in Section~\ref{sec:main}. To impose sparsity, we define the state space to be the set of all equivalence classes that satisfy some node degree constraints.    
We do not explicitly define the target posterior distribution in this section (which will be done in Section~\ref{sec:high}); instead, we assume the posterior has some consistency property that typically holds for sufficiently large sample sizes.  
To use Theorem~\ref{th:path.better}, we need to verify its assumption for the structure learning problem, which requires us to bound the neighborhood size  (see Lemma~\ref{lm:bound.nb}) and construct ``canonical paths'' for  the \RWGES{} sampler. 
We show by examples (see Examples~\ref{ex:why.swap} and~\ref{ex:local.mode.dag}) that  a major and unique challenge in the path construction is to verify that the equivalence classes located on the ``boundary'' of the restricted search space cannot be local modes. 
To overcome this, we introduce a ``swap'' proposal move to \RWGES{} and prove a key combinatorial property of DAGs in Lemma~\ref{lm:pigeonhole}. 
Combining it with the well-known Chickering algorithm~\citep{chickering2002optimal}, we obtain the canonical paths of  \RWGES{}.

In Section~\ref{sec:high}, we propose an empirical Bayes model for structure learning and prove that it has the desired high-dimensional consistency property assumed in Section~\ref{sec:main}. Our model generalizes the DAG selection model of~\citet{lee2019minimax}, and we show that it yields the same marginal fractional likelihood for Markov equivalent DAGs. 
The main result in this section is Theorem~\ref{th:cpdag.consist}, which gives the strong selection consistency of our structure learning model. 
For Bayesian methods, such consistency results have only been established lately for the DAG selection problem with known ordering~\citep{cao2019posterior, lee2019minimax}.  
Roughly speaking, in our consistency result, the maximum degree of searched DAGs is allowed to grow at rate $\sqrt{\log p}$; see Remark~\ref{rmk:order.d}.
The analogous high-dimensional consistency results for both variable selection and DAG selection are obtained as intermediate steps of our proof of Theorem~\ref{th:cpdag.consist}.  
 
The rapid mixing of \RWGES{} now follows from the mixing time bound given in Theorem~\ref{th:path.better} and the results of Sections~\ref{sec:main} and~\ref{sec:high}. It is formally stated in Section~\ref{sec:mcmc}. 
For comparison, we provide two slow mixing examples in the same section. The first one (see Example~\ref{ex:slow1}) shows why it is difficult to relax a key assumption used in our analysis, which is called the ``strong beta-min condition'' and is similar to the strong faithfulness assumption.  
The second (see Example~\ref{ex:slow2}) illustrates that the equivalence class sampler of~\citet{castelletti2018learning} may mix slowly when $p$ is small and $n$ is large.  
We conduct simulation studies in Section~\ref{sec:sim} to show that our theoretical results hold ``approximately'' for moderately large sample sizes and provide useful guidance on the use of \RWGES{} in practice.  
Section~\ref{sec:disc} concludes the paper with discussions on why the structure MCMC is difficult to analyze and potential extensions of \RWGES{}. All proofs are relegated to the supplementary material~\cite{zhousupplementary}.  
For readers' convenience, a notation table is given in Supplement A.

\section{Mixing time bounds for model selection problems}\label{sec:mix}

\subsection{A general setup} \label{sec:local.search} 
In this section, we use $\Theta = \Theta_p$ to denote a finite model space for a general model selection problem with $p$ variables; for example, for the structure learning problem, each $\theta \in \Theta$ can be a unique equivalence class.  
Let $\cN \colon \Theta \rightarrow 2^\Theta$ be given such that $\theta \notin \cN(\theta)$ for each $\theta \in \Theta$; $\cN$ is called a neighborhood function. We say $\theta'$ is a neighbor of $\theta$ if and only if $\theta' \in \cN(\theta)$.  
We say $\cN$ is ``symmetric'' if $\theta \in \cN(\theta')$ always implies $\theta' \in \cN(\theta)$.  
When we need to emphasize $\Theta$ is equipped with $\cN$, we denote the space by $(\Theta, \cN)$.  
Let $\pi$ denote a posterior distribution on $\Theta$ for a Bayesian procedure; assume it is known up to a normalizing constant and $\pi(\theta) > 0$ for each $\theta$.  
Given a function $h \colon  (0, \infty) \rightarrow (0, \infty)$,  define  a Markov chain $\bK^h$ on $\Theta$  by 
\begin{equation}\label{eq:general.proposal}
    \bK^h(\theta,  \theta') =   \frac{ h \left( \pi(\theta') / \pi(\theta) \right) }{   
    \sum_{\tilde{\theta} \in \cN(\theta)} h  ( \pi(\tilde{\theta}) / \pi(\theta)  )
    } \ind_{ \cN(\theta) }(\theta')
\end{equation}
where $\ind$ is the indicator function. That is, given current state  $\theta$, $\bK^h$ moves to some $\theta' \in \cN(\theta)$ with probability $\propto h(\pi(\theta') / \pi(\theta))$. 
Given $\bK^h$, define another Markov chain $\bP^h$ by 
\begin{equation}\label{eq:general.MH}
    \bP^h(\theta, \theta') = \left\{\begin{array}{cc}
    \bK^h(\theta, \theta')  \min \left\{1,  \frac{ \pi(\theta') \bK^h(\theta', \theta) }{ \pi(\theta) \bK^h(\theta, \theta')}   \right\}, \quad  & \text{ if } \theta' \neq \theta, \\
    1 -  \sum_{\tilde{\theta} \neq \theta} \bP^h(\theta, \tilde{\theta})     &  \text{ if } \theta' = \theta. 
    \end{array}
    \right.
\end{equation}
If $\bP^h$ is irreducible, then $\pi$ is the unique stationary distribution of $\bP^h$. To avoid periodicity, we will often work with the lazy version  $\bP_{\rm{lazy}}^h = (\bP^h + \mathbf{I}) / 2$, where $\mathbf{I}$ is the identity matrix. 

\begin{definition}[local Metropolis-Hastings algorithms]\label{def:two.MH.algorithms} 
We say $\bP^h$ defined by~\eqref{eq:general.MH} is a  local  MH algorithm with  local  proposal $\bK^h$.  
If $h \equiv 1$,  we say $\bP^h$ is the random walk MH algorithm. 
If $h$ is non-constant and non-decreasing, we say $\bK^h$ is a (locally) informed proposal and $\bP^h$ is a (locally) informed MH algorithm. 
\end{definition}
 
The locally informed MH algorithm was proposed by~\citet{zanella2020informed}. The main idea is to assign larger proposal probabilities to those neighboring states with larger posterior so that the chain can quickly move to high-posterior regions.  
Let $h$ in~\eqref{eq:general.proposal} be $h(u) = u^a$ for some $a \geq 0$. 
Observe that when $a = 0$, $\bK^h$ is reduced to the random walk proposal, and when $a \rightarrow \infty$ we obtain the greedy search (see definition below).   
So informed proposals are generally more aggressive than random walk but less aggressive than greedy search.  

\begin{definition}[greedy local search]\label{def:greedy} 
A greedy (local) search on $(\Theta, \cN)$ with initial state $\theta^{(0)}$ generates $\theta^{(1)}, \theta^{(2)}, \dots$, sequentially by letting  $\theta^{(i)} = \argmax_{  \theta' \in \cN(\theta^{(i-1)}) \cup \{\theta^{(i-1)} \} }  \pi( \theta')$ for each $i \geq 1$. 
The search stops and returns $\theta^{(j)}$ if $\theta^{(j)} = \theta^{(j - 1)}$. 
\end{definition}
   
The efficiency of both greedy search and MH algorithms largely depends on the choice of $\cN$. For model selection problems,  $\cN(\cdot)$ is usually much smaller than $\Theta$ so that the algorithm is computationally affordable. 
But $\cN$ should also provide enough connectivity so that the algorithm cannot get trapped at sub-optimal local modes ($\theta$ is a local mode if $\pi(\theta) > \pi(\theta')$ for any $\theta' \in \cN(\theta)$). 
We measure the convergence rate of MH algorithms using mixing time.  

\begin{definition}[mixing time]  \label{def:mix}
Let $\bP$ be an irreducible and aperiodic transition matrix defined on a finite state space $\Theta$, with stationary distribution $\pi$. Define its mixing time by 
\begin{align*}
\Tmix (\bP) = \max_{\theta \in \Theta} \, \min \{t  \geq 0 \colon \TV{ \bP^t(\theta, \cdot) - \pi(\cdot)} \leq 1/4\}, 
\end{align*}
where $\TV{ \cdot }$ denotes the total variation distance which takes value in $[0, 1]$.
\end{definition}
 
\begin{remark}\label{rmk:rapid}
We say an MCMC algorithm is rapidly mixing  if its mixing time grows at most polynomially in the complexity parameters $n$ (sample size) and $p$ (number of variables). 
For most high-dimensional model selection problems, the size of $\Theta$ grows at least super-polynomially with $p$. 
For variable selection, which is probably the best-studied problem,  \citet{yang2016computational} proved the rapid mixing of a random walk MH algorithm, and \citet{zhou2021dimension} showed that an informed MH algorithm can converge much faster and obtain a mixing time independent of $p$. 
\end{remark}

\subsection{A multi-purpose path method}\label{sec:mix.path}  
We propose a general method for bounding the mixing time of $\bP^h$ defined in~\eqref{eq:general.MH} and proving consistency properties of the posterior distribution $\pi$. 
The bounds to be derived in this section are non-asymptotic, and $p$ is treated as a fixed constant.
We begin by assuming that the triple $(\Theta, \cN, \pi)$ satisfies the following condition, where $|\cdot|$ denotes the cardinality of a set. 

\begin{condition}\label{cond:unimodal}
$|\Theta| < \infty$, $\cN$ is symmetric, and $\pi > 0$.  
There exists a function $g \colon \Theta \rightarrow \Theta$, a state $\theta^* \in \Theta$ and constants $t_1, t_2 > 0, p > 1$ such that (i) $|\cN(\theta)| \leq p^{t_1}$ for each $\theta \in \Theta$, and (ii) $g (\theta) \in \cN(\theta)$ and $\pi(g(\theta)) / \pi(\theta) \geq p^{t_2}$ for each $\theta \neq \theta^*$.  
\end{condition}

\begin{remark}\label{rmk:condition.uni}
Part (ii)  is equivalent to either of the following statements. 
\begin{enumerate}[label=(\alph*)]
    \item For any $\theta \neq \theta^*$,  $\max_{\theta' \in \cN(\theta)} \pi(\theta') \geq p^{t_2} \pi(\theta)$. 
    \item For any $\theta \neq \theta^*$, there exists $k < \infty$ and a sequence $(\theta_0 = \theta, \theta_1, \theta_2, \dots, \theta_k = \theta^*)$ such that $\theta_i \in \cN(\theta_{i-1})$ and $\pi(\theta_i)/\pi(\theta_{i-1}) \geq p^{t_2}$ for each $i = 1, \dots, k$. 
\end{enumerate}
We introduce the function $g$ because, for model selection problems, one often verifies Condition~\ref{cond:unimodal} by explicitly identifying some $g(\theta)$ for each $\theta$. There may exist many choices of $g$ so that Condition~\ref{cond:unimodal} holds. 
Without loss of generality, we always define $g(\theta^*) = \theta^*$. Then, part (ii) implies that for any $\theta$ there exists $k \leq |\Theta|$ such that $g^k(\theta) = \theta^*$, and $\theta^*$ is the only attracting fixed point of $g$. 
We will call $g$ a canonical transition function and a sequence of the form $(\theta, g(\theta), \dots, g^k(\theta) = \theta^*)$ a canonical path. 
We can think of a canonical path as a candidate ``greedy search path'' since the posterior keeps increasing along the path, but note that a greedy search does not necessarily follow a canonical path since $g(\theta)$ may not be the maximizer of $\pi$ in $\cN(\theta)$.  
\end{remark}
  
Roughly speaking, in the model selection context, $\theta^*$ can be thought of as the ``true'' data-generating model, and Condition~\ref{cond:unimodal} can be interpreted as an algorithmic consistency property since it implies that $\theta^*$ is the unique mode of   $\pi$ and the greedy search always returns $\theta^*$; see part (i) of Theorem~\ref{th:path}. 
For variable selection, \citet{yang2016computational} proved that Condition~\ref{cond:unimodal} holds with high probability under some mild high-dimensional assumptions  and then used the canonical path method of~\citet{sinclair1992improved} to bound the mixing time of the random walk MH algorithm.  
We generalize their result to our setup. 
  
\begin{theorem}\label{th:path}
Let $\Theta, \cN, \pi,  g, \theta^*, t_1, t_2, p$ be as given in Condition~\ref{cond:unimodal}.  
Let $\bP^h$ be given by~\eqref{eq:general.MH} and $\bP_{\rm{lazy}}^h = (\bP^h + \mathbf{I}) / 2$ be its  lazy  version.  
The following statements hold. 
\begin{enumerate}[label=(\roman*)]
    \item  The greedy search always returns $\theta^*$ regardless of the initial state. 
    \item  If $t_2 > t_1$, then $\pi(\theta^*) \geq 1- p^{-(t_2 - t_1)}$. 
    \item  If $t_2 > t_1 $, then 
        \begin{align*}
         \Tmix ( \bP_{\rm{lazy}}^h )\leq \frac{ 4 \ell_{\rm{max}}   }{ \left\{ 1 - p^{-(t_2 - t_1)} \right\}  \min_{\theta \neq \theta^*} \bP^h(\theta, g(\theta)) }  \log  \left( \frac{4}{ \pi_{\rm{min}}   }  \right), 
        \end{align*}
    where $\ell_{\rm{max}} = \max_{ \theta \neq \theta^*} \min\{ k \geq 1 \colon g^k(\theta) = \theta^* \}$ and $\pi_{\rm{min}} = \min_{\theta \in \Theta} \pi(\theta) $. 
    \item If $t_2 > t_1$ and $h \equiv 1$, then $\bP^h(\theta, g(\theta)) \geq p^{- t_1}$. 
\end{enumerate}
\end{theorem}
\begin{proof}
See Supplement B.4. 
\end{proof}

\begin{remark}\label{rmk:lazy}  
Part (ii) of Theorem~\ref{th:path} shows that $\pi$ concentrates on $\theta^*$, which can be further used to show the strong selection consistency of a Bayesian model selection procedure (see Section~\ref{sec:node.sel}). This is very useful  since   we only require polynomial (in $p$) bounds for the ratio $\pi(g(\theta)) / \pi(\theta)$ in Condition~\ref{cond:unimodal}, while $| \Theta |$ may be (super-)exponential in $p$. 
For a  random walk MH algorithm, by parts (iii) and (iv), the order of mixing time is given by $p^{t_1} \ell_{\max} \log \pi_{\rm{min} }^{-1}$. 
For the greedy search, note that   $\ell_{\max}$ is an upper bound for the steps needed to find $\theta^*$, and in each step the search needs to evaluate $\pi$ for at most $p^{t_1}$ states. Hence, the greedy search and random walk MH algorithm have very similar complexity. 
\end{remark}

We now show that the mixing time bound in Theorem~\ref{th:path} can be improved. The new bound given in Theorem~\ref{th:path.better} has two major advantages. First, it does not involve $\ell_{\rm{max}}$, which can be large.  
Second, it replaces $\min_{\theta \neq \theta^*} \bP^h(\theta, g(\theta)) $ in Theorem~\ref{th:path} with 
$ \min_{\theta \neq \theta^*} \bP^h(\theta, \cN^*(\theta) )$, 
where $\cN^*(\theta)$ is the set of  all ``desirable moves'' for $\bP^h$ at $\theta$ including $g(\theta)$. 
This is key to bounding the mixing times of  informed MH algorithms. 
To prove Theorem~\ref{th:path.better}, we use a novel path argument that may be of independent interest. For each $\theta \neq \theta^*$, we construct a set of paths from $\theta$ to $\theta^*$ using all desirable moves. By properly weighting these paths, we are able to bound the mixing time using a Poincar\'{e}-type inequality~\citep{kahale1997semidefinite}, which significantly generalizes the canonical path method.    
See Supplement B for details.  

\begin{theorem}\label{th:path.better}
Let $\Theta, \cN, \pi,  g, \theta^*, t_1, t_2, p$ be as given in Condition~\ref{cond:unimodal}, and 
$\bP^h$ be given by~\eqref{eq:general.MH}.  
For each $\theta \neq \theta^*$, define $\cN^*(\theta) = \{ \theta' \in \cN(\theta) \colon \pi(\theta') \geq p^{t_2} \pi(\theta)  \}$.
Let $\pi_{\rm{min}} = \min_{\theta \in \Theta} \pi(\theta)$. 
If $t_2 > t_1$, then,
\begin{align*}
    \Tmix (\bP_{\rm{lazy}}^h)\leq \frac{ 2 C(p, t_1, t_2)  \log  \left( \frac{4}{ \pi_{\rm{min}}  }  \right) }{ \min_{\theta \neq \theta^*} \bP^h(\theta, \cN^*(\theta) ) }, 
    \quad \text{ where } C(p, t_1, t_2) = \frac{1 + (1 - p^{t_1 - t_2} )^{-1}}{[1 - p^{(t_1 - t_2)/2} ]^2}. 
\end{align*}
\end{theorem}
\begin{proof}
See Supplement B.6. 
\end{proof}
 
\begin{remark}\label{rmk:flow.mix}
Theorem~\ref{th:path.better} can be used to immediately improve some existing mixing time bounds in the literature. Both~\citet{yang2016computational} and~\citet{zhuo2021mixing} proved the rapid mixing of a random-walk MH algorithm for some high-dimensional discrete-state-space problem by showing Condition~\ref{cond:unimodal} holds for some $g$ and  using the canonical path method underlying Theorem~\ref{th:path}. Theorem~\ref{th:path.better} shows that $\ell_{\rm{max}}$ can be dropped (in an asymptotic setting where $p \rightarrow \infty$ and $t_1 < t_2$ are fixed, $C(p, t_1, t_2) \rightarrow 2$). 
\end{remark}

\begin{remark}\label{rmk:flow.informed}
Another important application of Theorem~\ref{th:path.better} is the mixing time analysis of informed MH algorithms.  
Define $\cN^t(\theta) = \{ \theta' \in \cN(\theta) \colon \pi(\theta') \geq p^{t} \pi(\theta)  \}$. 
If $t_2$ is sufficiently large and the function $h$ in~\eqref{eq:general.proposal} is chosen properly,  it is often possible to show that $\min_{\theta \neq \theta^*} \bP^h (\theta, \cN^t(\theta) ) \geq c$  for some $t > t_1$ and fixed constant $c > 0$.  
Indeed, for the LIT-MH algorithm for variable selection considered in~\citet{zhou2021dimension}, one can follow their calculations to verify that this holds for $c = 1/4$, and then by Theorem~\ref{th:path.better},  the order of the mixing time is only $\log \pi_{\rm{min}}^{-1}$. 
This cannot be achieved by using Theorem~\ref{th:path}, since $\bP^h(\theta, g(\theta))$ can be as small as $O(p^{-t_1})$ (e.g., when all neighboring states have the same posterior probabilities).  
\end{remark} 

The theory developed in this section relies on Condition~\ref{cond:unimodal}, which is a property of the triple $(\Theta, \cN, \pi)$.  
If Condition~\ref{cond:unimodal} holds, one can use Theorem~\ref{th:path.better} to study the mixing times of  any local MH algorithm.  
For simplicity, for the structure learning problem to be studied in the rest of this paper,  we will only consider the random walk proposal, and our main task is to construct a triple  $(\Theta, \cN, \pi)$ that satisfies Condition~\ref{cond:unimodal}.  
We will often define $\cN$ on $\Theta$ and then use $\cN$ to refer to a neighborhood relation on a restricted space $\Theta_0 \subset \Theta$; this  means that the neighborhood of $\theta \in \Theta_0$ is given by $\cN(\theta) \cap \Theta_0$.  
Note that even if $(\Theta, \cN, \pi)$  satisfies Condition~\ref{cond:unimodal}, $(\Theta_0, \cN, \pi)$ may not, which is one challenge in the sparse structure learning problem to be considered. 

\section{The \RWGES{} sampler and its canonical paths} \label{sec:main} 
\subsection{Notation and terminology} \label{sec:dag.notation}
We set up the notation to be used for the structure learning problem. 
Let $[p] = \{1, \dots, p\}$ and $|\cdot|$ denote the cardinality of a set.  
A subset of $[p]$ is typically denoted by $S$. 
The Hamming distance between two sets $S, S'$ is denoted by  $\HD(S, S') = |S \setminus S'| + |S' \setminus S|$. 

A DAG $G$ is a pair $(\Vert, \Edge)$ where $\Vert$ is the vertex set and $\Edge \subset \Vert \times \Vert $ is the set of directed edges.  
Throughout the paper, we assume  $V = [p]$ for DAG models, representing random variables $\sX_1, \dots, \sX_p$. 
Note that $(i, i) \notin \Edge$ for any $i  \in [p]$. 
Let $|G|$ denote the number of edges in the DAG $G$; thus, $|G| = |E|$.  
We use the notation $i \rightarrow j \in G$ to mean that $(i, j) \in \Edge$ and $(j, i) \notin \Edge$. 
The notation $i \rightarrow j \notin G$ means that $(i, j) \notin \Edge$. 
For two DAGs $G = (\Vert, \Edge)$ and $G' = (\Vert, \Edge')$, we write $G' = G \cup \{ i \rightarrow j \}$ if $\Edge' = \Edge \cup (i, j)$, and $G' = G \setminus \{ i \rightarrow j \}$ if $\Edge' = \Edge \setminus (i, j)$.  
We write $G = G'$ if and only if $G$ and $G'$ have  the same  vertex set and edge set. 
Given a DAG $G$, we say node $i$ is a parent of node $j$ (and node $j$ is a child of node $i$) if $i \rightarrow j \in G$.
Let $\Pa_j(G) = \{ i \in [p] \colon \;  i \rightarrow j \in G \}$  denote the set of parents of node $j$; the in-degree of  node $j$ is  $| \Pa_j(G)|$. 
The maximum in-degree of $G$ is $\max_j |\Pa_j(G)|$. 
Similarly, let $\Ch_j(G) =  \{ i \in [p] \colon \; j \rightarrow i \in G \}$, and $|\Ch_j(G)|$ is called the out-degree of node $j$. 
The degree of a node is the sum of its in-degree and out-degree, and the maximum degree of $G$ is $\max_j |\Pa_j(G) \cup \Ch_j(G)|$. 
We may simply write $\Pa_j$ if we are not referring to a specific DAG or the underlying DAG is clear from context. 
The Hamming distance between two DAGs $G, G'$ is defined by $\HD(G, G') = \sum_{j \in [p]} \HD( \Pa_j(G), \, \Pa_j(G') ).$

An equivalence class of DAGs is typically denoted by $\cE$. 
We always interpret $\cE$ as a set of DAGs, and use $|\cE|$ to denote the number of member DAGs in $\cE$.  
The equivalence class of a DAG $G$ is also denoted by $\EG{G}$; thus, $\cE = \EG{G}$ if and only if $G \in \cE$. 
The set of CI relations encoded by a DAG $G$ or an equivalence class $\cE$  is denoted by $\CI(G)$ or $\CI(\cE)$, respectively. Note that we always have $\CI(G) = \CI(\EG{G})$.

We say a $p$-variate  distribution $\mu$ is Markovian w.r.t. a $p$-vertex DAG $G$
and $G$ is an independence map (I-map) of $\mu$ if all CI relations encoded by $G$ hold for $\mu$. 
If the converse is also true, we say $\mu$ is faithful or perfectly Markovian w.r.t. $G$, and $G$ is a perfect map of $\mu$~\citep{spirtes2000causation, studeny2006probabilistic}.   
We say  $\mu$ is DAG-perfect if there exists some DAG that is a perfect map of $\mu$.  
A DAG $G$ is an I-map of a DAG $G'$ and its equivalence class $\EG{G'}$ if $\CI(G) \subseteq \CI(G')$,  and  $G$ is a minimal I-map (of $G'$) if any sub-DAG of $G$ (different from $G$) is not an I-map of $G'$.  
Given the set $\CI(G)$, a minimal I-map of $G$ with ordering $\sigma$, which we denote by $G_\sigma$, can be uniquely defined as follows: for any $i < j$, $\sigma(i) \rightarrow \sigma(j) \in G_\sigma$ if and only if nodes $\sigma(i), \sigma(j)$ are not conditionally independent given nodes $\{\sigma(1), \dots, \sigma(j - 1) \} \setminus \{ \sigma(i) \}$~\citep{solus2017consistency}.  
An example for $p = 3$ is given below.  
If $\mu$ is a $p$-variate positive measure, a unique minimal I-map of $\CI(\mu)$ with ordering $\sigma$ can be constructed in an analogous manner~\citep[Chapter 3.4]{koller2009probabilistic}.

\begin{example}\label{ex:minimal.imap}
Let $p = 3$ and  $G$ be the DAG $1 \rightarrow 3 \leftarrow 2$. 
Let $G_\sigma$ denote the minimal I-map of $G$ with ordering $\sigma$. 
If $\sigma = (1, 2, 3)$ or $(2, 1, 3)$, then  $G_\sigma = G$ since $\sigma$ is an ordering of $G$. If $\sigma$ is any other ordering, then $G_\sigma$ is the complete DAG (i.e., a DAG without missing edge). For example, if $\sigma = (1, 3, 2)$, $G_\sigma$ has three edges $1 \rightarrow 3$, $1 \rightarrow 2$ and $3 \rightarrow 2$; in particular, the edge $1 \rightarrow 2$ is included since $1 \notindep 2 \mid 3$ in $G$. 
\end{example}

\subsection{Search spaces and posterior distributions} \label{sec:search.space} 
To apply the general theory developed in Section~\ref{sec:mix}, it suffices to construct a triple $(\Theta, \cN, \pi)$ that satisfies Condition~\ref{cond:unimodal}. We will do this for both high-dimensional DAG selection and structure learning.  
Recall that for DAG selection, our goal is to estimate an underlying DAG model from the data when we know it has some ordering $\sigma$, and for structure learning, our goal is to estimate the equivalence class of the DAG model. 
We first define the search spaces (i.e., model space) for the two problems. 
Let $\dag$ denote the space of all $p$-vertex DAGs, which grows super-exponentially in $p$. 
We consider two sparsity constraints for DAGs, one for the maximum in-degree and the other for the maximum out-degree. 
For $\din, \dout \in [p]$,  define 
\begin{equation*}
\dag(\din, \dout) = \{  G \in \dag \colon  \max_j | \Pa_j(G) |  \leq \din, \text{ and } \max_j | \Ch_j(G)  | \leq \dout  \}. 
\end{equation*}
Since all Markov equivalent DAGs have the same skeleton, the two constraints ensure that the degree of any DAG $G' \in \EG{G}$ for some $G \in \dag(\din, \dout)$ is at most $\din + \dout$.  
One may also use a single constraint for the maximum degree, but for the theoretical analysis to be carried out in this paper, it is more convenient to specify $\din, \dout$ separately.  
This setup is appealing to practitioners, since a DAG model with bounded degree is easier to visualize and interpret.   
Let $\cpdag(\din, \dout)$ denote the space of ``sparse equivalence classes''  defined by
\begin{equation*}
\cpdag (\din, \dout) =   \left\{ \EG{G}  \colon  G \in \dag(\din,  \dout)  \right\}. 
\end{equation*}
Hence, $\cpdag (\din, \dout) $ is the set of all equivalence classes that contain at least one member in $\dag(\din, \dout)$.   
We will use  $\cpdag (\din, \dout) $ as the model space for the sparse structure learning problem. 
The unrestricted space is denoted by $\cpdag = \cpdag(p, p)$. 
 
Recall that $\bbS^p$ is the space of all permutations of $[p]$. For each $\sigma \in \bbS^p$, let 
\begin{align*}
\dag^\sigma =\;&  \{ G \in \dag  \colon  \sigma \text{ is a topological ordering of } G   \}  \\
=\;& \{ G \in \dag  \colon     \sigma(j) \rightarrow \sigma(i)   \notin G \text{ for any } i < j  \}. 
\end{align*} 
Note a DAG may have multiple orderings; in particular, the empty DAG belongs to $\dag^\sigma$ for any $\sigma \in \bbS^p$.   
Let $\dag^\sigma(\din, \dout) = \dag^\sigma \cap \dag(\din, \dout)$ denote the space of sparse DAG models with ordering $\sigma$, which is the space we consider for the sparse DAG selection problem.  

For our target posterior probability distributions, we assume they can be expressed by using a Bayesian scoring criterion $\score \colon \dag \to \bbR$ such that $\score(G) = \score(G')$ for any Markov equivalent $G$ and $G'$, a property known as ``score equivalence''~\citep{chickering2002optimal}. 
For an equivalence class $\cE$, define $\score(\cE) = \score(G)$ using any $G \in \cE$. 
Let the un-normalized posterior probability of a DAG $G$ be given by $e^{\score(G)}$, and that of an equivalence class $\cE$ be given by  $e^{\score(\cE)}$ (see Section~\ref{sec:cpdag.model} for more details).  
We further assume that $\score$ is decomposable:  for each $G$, 
\begin{equation*}\label{eq:decomposable.score}
    \score(G) = \sum_{j \in [p]} \scorej(\Pa_j(G)),
\end{equation*}
where for each $j$, $\scorej \colon 2^{[p]} \rightarrow \bbR$ gives the local score at node $j$. 

\subsection{Neighborhood functions and the \RWGES{} sampler} \label{sec:rwges} 
We  define our neighborhood function on $\cpdag(\din, \dout)$  by considering operations on all member DAGs of each equivalence class. 
To this end, we first define three neighborhoods on the unrestricted space $\dag$ for each DAG $G$, which correspond to three types of edge modification: addition, deletion and swap.  
\begin{align*}
\adds(G) =\;&  \left\{ G' \in \dag  \colon  G' = G \cup \{ i \rightarrow j \} \text{ for some }  i\rightarrow j \notin G   \right\}, \\
\dels(G) =\;&  \left\{ G' \in \dag  \colon  G' = G \setminus \{ i \rightarrow j \} \text{ for some } i\rightarrow j \in G  \right\}, \\ 
\swaps(G) =\;&  \left\{ G' \in \dag \colon   G' = ( G \cup \{ k \rightarrow j \} ) \setminus \{ \ell \rightarrow j \} \text{ for some }  k \rightarrow j \notin G, \,  \ell \rightarrow j \in G \right\}. 
\end{align*}
Note that a swap move consists of adding an incoming edge and deleting one at the same node, which is a straightforward extension of the swap proposal used in variable selection problems. Define the ``add-delete-swap neighborhood'' of $G$ by 
\begin{equation}\label{eq:def.nb}
    \nb(G) =  \adds (G) \cup \dels (G) \cup \swaps (G). 
\end{equation}  
For each equivalence class $\cE \in \cpdag$,  define  
\begin{equation}  
\nbg(\cE) = \left\{ \EG{G'} \colon G' \in \nb(G)  \text{ for some } G \in \cE \right\},  \label{eq:def.nbg}
\end{equation}
and define the sets $\adds(\cE), \dels(\cE)$ and $\swaps(\cE)$ analogously; for example, $\cE' \in \adds(\cE)$  if and only if there exist $G \in \cE$ and $G' \in \cE'$ such that $G' \in \adds(G)$. (The neighborhood notation is overloaded here, but the meaning should be clear from the argument.) 
Clearly, $\nbg(\cE) = \adds(\cE) \cup \dels(\cE) \cup \swaps (\cE)$, and  $\nbg$ is symmetric on both $\dag$ and $\cpdag$.  
The following lemma gives a bound on the size of $\nbg(\cE)$, which is needed later when we verify part (i) of Condition~\ref{cond:unimodal}. 

\begin{lemma}\label{lm:bound.nb}
For any $\cE \in \cpdag(\din, \dout)$,  
$$  | \nbg(\cE) \cap \cpdag(\din, \dout) | \leq 3 p(p - 1) (\din + \dout) 2^{\din + \dout}.$$
\end{lemma} 
\begin{proof}
See Supplement D.1. 
\end{proof}

As explained in Section~\ref{sec:local.search}, we can construct a random walk MH algorithm on the restricted space $\cpdag(\din, \dout)$ using $\nbg$. 
The proposal distribution is given by $\bK(\cE, \cE') = 1 / | \nbg(\cE) |$ for each $\cE' \in \nbg(\cE)$, where $\nbg(\cE)$ is still defined on the unrestricted space (if we propose $\cE' \notin \cpdag(\din, \dout)$, we simply reject the proposal). 
It should be noted that, in practice, there is no need to calculate the size of $\nbg(\cE)$ or enumerate member DAGs in $\cE$.  States in $\nbg(\cE)$ can be proposed very efficiently by using some local graph operators, which is explained in detail in Supplement H.1.   
We call this sampler random walk GES (\RWGES{}), since it uses a neighborhood function similar to that of the GES algorithm~\citep{chickering2002optimal}, which is a two-stage greedy search on the space $\cpdag$ that uses $\adds$ in the first stage and $\dels$ in the second. 
Swap moves are not used in GES, and we will use $\cN_{\rm{ges}}(\cdot) = \adds(\cdot) \cup \dels(\cdot)$ to denote the neighborhood relation used by GES.

\subsection{Motivating examples} \label{sec:N.examples} 
Assume  the data-generating distribution is perfectly Markovian w.r.t. some DAG $G^*$ (which henceforth is called the ``true DAG'') and let $\cE^* = \EG{G^*}$ be the true equivalence class. 
In the classical asymptotic regime where $p$ is fixed and sample size $n$ tends to infinity,  \citet{chickering2002optimal} proved that for a large class of Bayesian scoring criteria, GES and the greedy search on $(\cpdag, \cN_{\rm{ges}})$ are consistent.  
According to our discussion following Condition~\ref{cond:unimodal}, if we fix $p$ and let $n \rightarrow \infty$, we can mimic the consistency proof of GES and use Theorem~\ref{th:path} to bound the mixing time of  the random walk MH algorithm on $(\cpdag, \cN_{\rm{ges}})$. 
The purpose of this subsection is to use examples to illustrate the technical challenges we encounter as we try to extend this argument to the space $\cpdag(\din, \dout)$. 

To simplify the discussion, we assume the score $\score$ (i.e., log-posterior) satisfies the following condition, known as local consistency~\citep{chickering2002optimal}  (which is only used for making heuristic arguments in this section).  
It essentially says that all CI relations encoded by $G^*$ can be correctly identified, which we expect to happen when $n = \infty$.

\begin{condition}\label{cond:local.consist} 
If distinct DAGs $G, G'$ satisfy $G' = G \cup \{i \rightarrow j\}$,  then (i) $\score(G) > \score(G') $ if $i \indep j \mid \Pa_j(G)$ in $G^*$, and  
(ii) $\score(G') > \score(G) $ if $i \notindep j \mid \Pa_j(G)$ in $G^*$.   
\end{condition}   
Under Condition~\ref{cond:local.consist},  GES is consistent~\citep{chickering2002optimal} and no equivalence class other than $\EG{G^*}$  can be a local mode on $(\cpdag,  \cN_{\rm{ges}})$ (the reason will become clear in the next subsection). 
However, once we introduce the  degree constraint (which is necessary for proving high-dimensional consistency results),  local modes can arise on the boundary of the restricted space. 
To illustrate this, we construct two examples below.  
Example~\ref{ex:why.swap} explains  why swap moves are useful and why in the consistency proof of GES %(and our analysis of \RWGES{} later)  %% false
we only consider edge removals when the current equivalence class is an I-map of $\cE^*$.   
Example~\ref{ex:local.mode.dag} shows that for the sparse DAG selection problem with degree constraints,  local modes can also arise unexpectedly. 
 
\begin{example}\label{ex:why.swap}
Let $p = 3$ and  DAGs $G^*$, $G$ be given by 
\begin{align*}
G^* \colon 1 \rightarrow 2 \rightarrow 3, \quad G \colon 2 \leftarrow 1 \rightarrow 3. 
\end{align*}  
Consider how to increase the score of $G$ by single-edge addition or deletion under Condition~\ref{cond:local.consist}. 
Since $1 \notindep 2$ and $1 \notindep 3$ in $G^*$,   both edges cannot be removed. 
However, since $2 \notindep 3 \mid 1$ in $G^*$, we can add the edge $2 \rightarrow 3$ to $G$ to increase the score. 
The complete DAG $G \cup\{2 \rightarrow 3\}$ is an I-map of $G^*$, from which we should be able to remove the edge $1 \rightarrow 3$ since $1 \indep 3 \mid 2$. 
One can apply the same argument to any other DAG in $\cE = \EG{G}$ and conclude that $\score(\cE) > \score(\cE')$ for any $\cE' \in \dels(\cE)$. 
In particular, we cannot remove the edge between nodes $1, 3$ from any $G \in \cE$, though the two nodes are not connected in $G^*$.  
 
Next, we impose the constraint $\din = 1$.  
Since $G$ has two edges, we have $\adds(\cE) = \{\tcE \}$, where $\tcE$ is the equivalence class of all complete DAGs. But any complete DAG has maximum in-degree $2$, which means that moving from $\cE$ to $\tcE$ is forbidden and $\cE$ is a local mode on $( \cpdag(\din = 1,  \dout = p),  \cN_{\rm{ges}} )$. However, a swap move allows us to directly move from $G$ to $G^*$ by removing $1 \rightarrow 3$ and adding $2 \rightarrow 3$ simultaneously; that is, $\cE$ is not a local mode on $( \cpdag(1,  p),  \nbg )$ where $\nbg$ is given by~\eqref{eq:def.nbg}. 
\end{example}
 
\begin{example}\label{ex:local.mode.dag}
Consider the DAG selection problem with $p = 5$ and $\sigma = (1, 2, 3, 4, 5)$.  Let $G^*, G$ be DAGs in $\dag^\sigma$ with edge sets 
\begin{align*}
G^* \colon \{  (1, 2), (1, 3), (2, 4), (2, 5) \}, \quad  G \colon \{  (1, 2), (1, 4), (2, 3), (2, 5) \}. 
\end{align*}
Under Condition~\ref{cond:local.consist}, we can increase the score of $G$ by adding $1 \rightarrow 3$ or $2 \rightarrow 4$, but deleting $1 \rightarrow 4$ or $2 \rightarrow 3$ will lower the score since $1 \notindep 4$ and $2 \notindep 3$ in $G^*$. 
Now let  $\din = \dout = 2$. Though $G^*, G \in \dag^\sigma(2, 2)$,   $G$ is a local mode on $(\dag^\sigma(2, 2), \nbg)$ because adding either $1 \rightarrow 3$  or $2 \rightarrow 4$ violates the out-degree constraint (note swap moves may not be helpful either).  
\end{example}

\subsection{Overview of the canonical path construction} \label{sec:path.summary} 
Let the true DAG model $G^* \in \dag(\din, \dout)$ and let $\cE^* = \EG{G^*}$.   
To verify Condition~\ref{cond:unimodal} for the triple $(\cpdag(\din, \dout), \nb, e^\score)$, we need to show that for any $\cE \neq \cE^*$, we can identify some $g(\cE) \in \nb(\cE)$ such that $\score(g(\cE)) > \score (\cE)$. 
By Remark~\ref{rmk:condition.uni}, this is equivalent to constructing a canonical path from any $\cE$ to $\cE^*$.  
We briefly discuss the main idea behind our construction in this subsection.  
%To understand our construction, 
It will be  helpful to think of the space  $\cpdag(\din, \dout)$ as the union of $\{ \dag^\sigma(\din, \dout) \colon \sigma \in \bbS^p \}$ and think of structure learning as simultaneous DAG selection for all $p!$ orderings. 

Suppose \RWGES{} starts at some $\cE$ which contains a member DAG $G \in \dag^\sigma(\din, \dout)$ for an arbitrary $\sigma \in \bbS^p$. We will first construct a canonical path on  $\dag^\sigma(\din, \dout)$, denoted by $(G_0 = G, G_1, G_2, \dots, G_k)$,  where the terminal state $G_k$ (if possible) is given by 
\begin{equation}\label{eq:def.hatG}
    \hat{G}(\sigma) = \argmax_{G \in \dag^\sigma(\din, \dout)} \score (G).
\end{equation} 
If Condition~\ref{cond:local.consist} holds and $G^*_\sigma \in \dag^\sigma(\din, \dout)$,   we claim $\hat{G}(\sigma) = G^*_\sigma$, where we recall $G^*_\sigma$ is the minimal I-map of $G^*$ with ordering $\sigma$.    
To show this, without loss of generality, assume $\sigma = (1, 2, \dots, p)$, and note that the following CI relations hold in $G^*$ for each $j \in [p]$ by the definition of minimal I-maps (see Section~\ref{sec:dag.notation}):  
\begin{equation}\label{eq:imap.CI}
j \indep [j - 1] \setminus \Pa_j(G^*_\sigma) \mid  \Pa_j(G^*_\sigma),   \text{ and } j \notindep i \mid  [j-1] \setminus \{i\} \text{ fo each } i \in \Pa_j(G^*_\sigma). 
\end{equation}
Under Condition~\ref{cond:local.consist}, the first property in~\eqref{eq:imap.CI} implies that if $G$ is a DAG such that $\Pa_j(G^*_\sigma) \subsetneq \Pa_j(G)$, we can increase the score of $G$ by removing some edge $\ell \rightarrow j$,  and the second implies that if $\Pa_j(G^*_\sigma) \not\subseteq \Pa_j(G) $, we can add some edge $k \rightarrow j$  or perform a swap. 
This shows $\hat{G}(\sigma) = G^*_\sigma$ and  suggests how we can construct the path from $G$ to $G^*_\sigma$. 
However, as discussed in the previous subsection, the main challenge is to deal with the degree constraints. 

Now suppose that \RWGES{} can move from $\cE$ to $\EG{ G^*_\sigma}$ following the path $(\cE, \cE_1, \cE_2, \dots, \cE_k)$ where $\cE_{i} = \EG{G_i}$. If $\EG{G^*_\sigma} = \cE^*$ (i.e., $G^*_\sigma = G^*$ or $G^*_\sigma$ is Markov equivalent to $G^*$), we have obtained the path from $\cE$ from $\cE^*$. 
If $\EG{G^*_\sigma} \neq \cE^*$, then one can use the famous Chickering algorithm~\citep{meek1997graphical, chickering2002optimal} to construct a path from $\EG{G^*_\sigma}$ to $\cE^*$ (see Lemma D3 in Supplement D.4).  Intuitively, since $G^*_\sigma$ is an I-map of $G^*$, the skeleton of $G^*$ must be a subset of the skeleton of $G^*_\sigma$ (see Lemma C3), and we can remove edges from some other member DAG of  $\EG{G^*_\sigma}$.  
 
Unfortunately, to rigorously prove that $\hat{G}(\sigma) = G^*_\sigma$ for all $\sigma \in \bbS^p$  in high-dimensional settings, one often needs to impose restrictive assumptions on the true data-generating mechanism, such as strong faithfulness~\citep{nandy2018high}. 
To our knowledge, there is no fully satisfactory solution to this issue, and we will make a similar assumption in our theoretical analysis in Section~\ref{sec:high} and assume $\hat{G}(\sigma) = G^*_\sigma$ in this section. 
Nevertheless, we will construct canonical paths of \RWGES{} using a flexible and finer argument, which, in some cases, can be used to show the rapid mixing of \RWGES{} under weaker assumptions; see Supplement I. 
 
\subsection{Canonical add-delete-swap paths of \RWGES{}} \label{sec:path.cpdag} 
The discussion above suggests that we can construct the canonical paths of \RWGES{} by first constructing the canonical paths for all DAG selection problems. 
To this end, fix an arbitrary $\sigma \in \bbS^p$ first, and consider the sparse DAG selection problem with state space $\dag^\sigma(\din, \dout)$, neighborhood function $\nb$ and posterior $e^\score$.   We treat $G^*_\sigma$ as the true model, and we need to construct a candidate canonical transition function for this problem, $g^\sigma \colon \dag^\sigma(\din, \dout) \rightarrow \dag^\sigma(\din, \dout)$,  such that for any $G \in \dag^\sigma(\din, \dout)$, 
\begin{align*}
    g^\sigma(G) \in \nb(G),  \text{ and }  (g^\sigma)^k(G) = G^*_\sigma \text{ for some } k < \infty. 
\end{align*}
For Condition~\ref{cond:unimodal} to hold we also need $\score(g^\sigma (G)) > \score(G)$.  To overcome the out-degree constraint issue illustrated by Example~\ref{ex:local.mode.dag}, we will construct  $g^\sigma(G)$ by first analyzing each node separately.  
Observe that if there is no out-degree constraint, the DAG selection problem  is equivalent to $p$ variable selection problems: for each $j$, we need to estimate the  set $\Pa_j$ which takes value in the space $\model^\sigma(j,  \, \din)$  defined by 
\begin{equation}\label{eq:def.model}
\model^\sigma(j,  \, \din) = \left\{  S \subseteq  \cA_p^\sigma(j) \colon |S| \leq \din    \right\}, \quad 
\cA_p^\sigma(j) = \left\{ k \in [p]\colon \sigma^{-1}(k)   < \sigma^{-1}(j) \right\}, 
\end{equation}
where $\cA_p^\sigma(j)$ is the set of variables that precede $\sX_j$ in the ordering $\sigma$.   
Motivated by the discussion following~\eqref{eq:imap.CI}, we construct a transition function on the space $\model^\sigma(j,  \, \din)$ in Definition~\ref{def:gj}, which gives the ``optimal'' add-delete-swap move for $\Pa_j$. 
Recall that we assume $\score(G) = \sum_j \scorej ( \Pa_j(G))$ for each $G$.

\begin{figure}[!t]
\begin{center}
\includegraphics[width=0.75\linewidth]{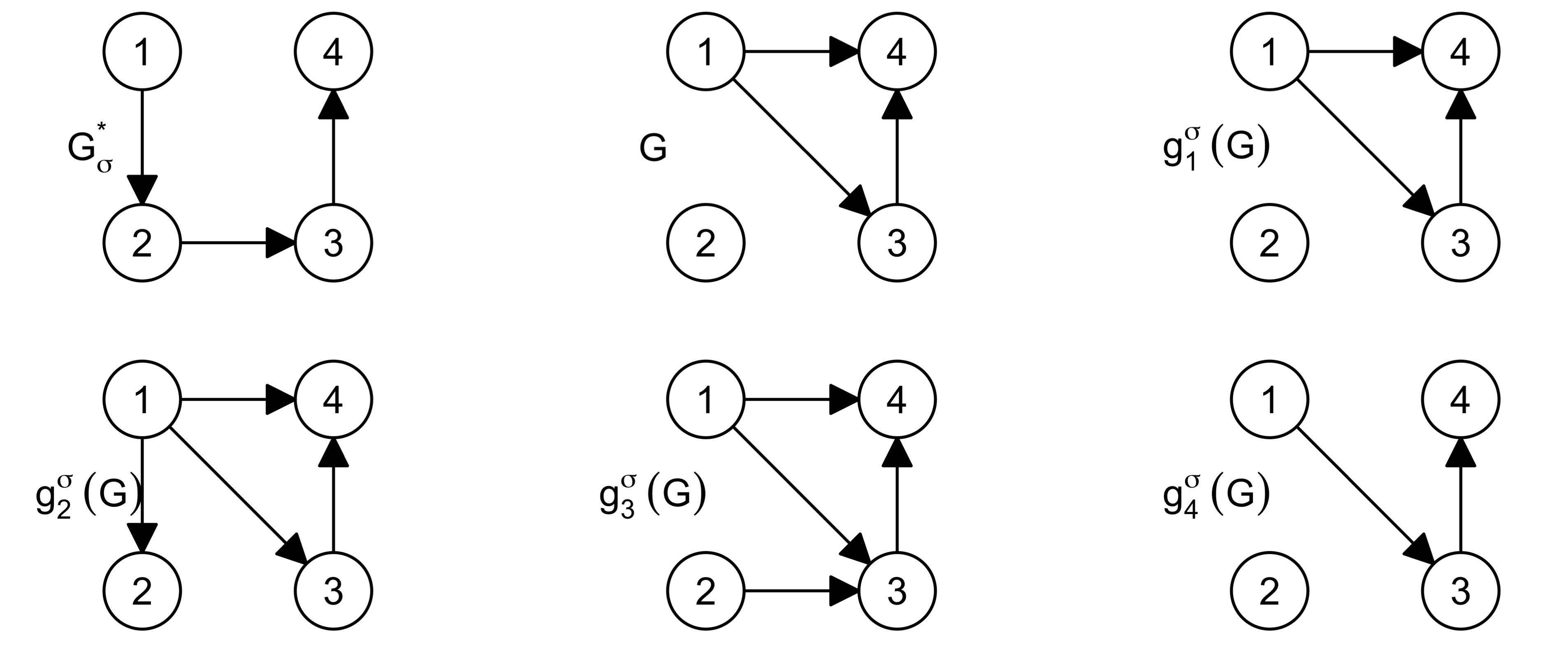} \\
\caption{An example for the operator $g_j^\sigma$. We consider four nodes with ordering $\sigma = (1, 2, 3, 4)$; assume $\din = 3$. 
 $G^*_\sigma$ has three edges, $1 \rightarrow 2$, $2 \rightarrow 3$ and $3 \rightarrow 4$. 
Consider another DAG $G$ with edges $1 \rightarrow 3$, $1 \rightarrow 4$ and $3 \rightarrow 4$. 
The DAGs $g_1^\sigma (G), g_2^\sigma (G), g_3^\sigma (G), g_4^\sigma (G)$ are shown above. 
For example, since  $\Pa_4(G^*_\sigma) = \{3\} \subset \Pa_4(G) = \{1, 3\}$, node $4$ is overfitted in $G$, and by part (ii) of Definition~\ref{def:gj},     $g_4^\sigma(G)$ is obtained by removing the edge $1 \rightarrow 4$ from $G$.  
}\label{fig1}
\end{center}
\end{figure}

\begin{definition}\label{def:gj}
Assume $G^*_\sigma \in \dag^\sigma(\din, \dout)$ and let $S^*_{\sigma, j} = \Pa_j(G^*_\sigma)$. 
For each $j$, we construct $g_j^\sigma \colon \model^\sigma(j,  \, \din) \rightarrow  \model^\sigma(j,  \, \din)$ as follows. 
Fix an arbitrary $S \in \model^\sigma(j,  \, \din)$, and let $T = S^*_{\sigma, j} \setminus S$ and $R = S \setminus S^*_{\sigma, j}$. 
\begin{enumerate}[label=(\roman*)]
    \item  If $S = S^*_{\sigma, j}$,  let $g_j^\sigma(S) = S^*_{\sigma, j}$.
    \item  If $S^*_{\sigma, j}\subset S$,  let $g_j^\sigma(S) = S \setminus \{ \tilde{\ell} \}$ where $\tilde{\ell}  = \argmax_{ \ell \in R } \scorej ( S \setminus \{ \ell \} )$. 
    \item  If $S^*_{\sigma, j} \not\subseteq S$ and $|S| < \din$,  let $g_j^\sigma(S) = S \cup \{ \tilde{k} \}$ where  $\tilde{k} = \argmax_{ k \in T } \scorej ( S \cup \{ k \} )$. 
    \item  If $S^*_{\sigma, j} \not\subseteq S$ and $|S| = \din$,  let $g_j^\sigma(S) = ( S \cup \{ \tilde{k} \} ) \setminus \{ \tilde{\ell} \}$ where $(\tilde{k}, \tilde{\ell}) = \argmax_{ (k, \ell) \in T \times R }$    $\scorej(  ( S \cup \{ k \} ) \setminus \{ \ell \} )$. 
\end{enumerate}
In case (ii), we say node $j$ is (strictly) overfitted; in cases (iii) and (iv), we say it is underfitted. 
We use $g_j^\sigma(G)$  to denote the DAG obtained by replacing the parent set of $j$ in $G$ with $g_j^\sigma( \Pa_j(G) )$; that is, 
$\Pa_j( g_j^\sigma( G ) ) = g_j^\sigma( \Pa_j(G) )$, and for any $i \neq j$, $\Pa_i( g_j^\sigma( G ) ) = \Pa_i(G)$.  
\end{definition}

\begin{remark}\label{rmk:gj.dist}
It is clear from definition that $\HD( g_j^\sigma(S), S^*_{\sigma, j}) < \HD(S, S^*_{\sigma, j})$ if $S \neq S^*_{\sigma, j}$.  
Further, $g_j^\sigma(G) \in \nb(G)$ and $\HD( g_j^\sigma(G), G^*_\sigma) < \HD(G, G^*_\sigma)$ if $\Pa_j(G) \neq \Pa_j(G^*_\sigma)$. 
In words, if node $j$ is overfitted in $G$, $g_j^\sigma(G)$ is obtained by removing an  incoming edge of node $j$. 
If node $j$ is underfitted, $g_j^\sigma(G)$ is obtained by adding an incoming edge of node $j$  (if the in-degree constraint is violated, remove another incoming edge of node $j$).
An example is provided in Figure~\ref{fig1}.  Note that this rationale  is similar to that for GES and forward-backward stepwise regression. We always first transform an underfitted model to overfitted and then remove redundant variables or edges (recall Example~\ref{ex:why.swap}).  
\end{remark}

\begin{remark}\label{rmk:gj.sel}
Consider the variable selection problem  with model space $\model^\sigma(j,  \, \din)$ and true model  $S^*_{\sigma, j}$. 
\citet{yang2016computational} proved that, under  very mild high-dimensional assumptions, $g_j^\sigma$ satisfies Condition~\ref{cond:unimodal} with high probability; that is, 
\begin{equation}\label{eq:conj.scorej}
\scorej ( g_j^\sigma( S ) ) - \scorej ( S ) \geq t \log p, \quad \quad \forall \,  S \in  \model^\sigma(j,  \, \din) \setminus \{ \Pa_j(G^*_\sigma) \}, 
\end{equation} 
for some $t > 0$ (in their conclusion $t$ is a universal constant). 
\end{remark}

Suppose that~\eqref{eq:conj.scorej} holds for each $j$. Then, to show that the triple $(\dag^\sigma(\din, \dout), \nb, e^{\score})$  satisfies part (ii) of Condition~\ref{cond:unimodal}, we only need to use  the operators $\{g_j^\sigma\colon j \in [p] \}$ to construct a path from any $G \in \dag^\sigma(\din, \dout)$ to $G^*_\sigma$.  At first glance, this seems trivial since we can use $g_j^\sigma$ repeatedly to convert any $\Pa_j(G)$ to $\Pa_j(G^*_\sigma)$. 
However,  the definition of $g_j^\sigma$ only guarantees that $g_j^\sigma(G) \in \dag^\sigma(\din, p)$, but the maximum out-degree of $g_j^\sigma(G) $ can be larger than that of $G$.  Indeed,  Example~\ref{ex:local.mode.dag} in Section~\ref{sec:N.examples} shows that,  in extreme cases,  none of the operators $g_1^\sigma, \dots, g_p^\sigma$ yields a DAG that is different from $G$  and belongs to  $\dag^\sigma(\din, \dout)$.  
Fortunately, we are able to prove that, as long as $\dout$ is chosen sufficiently large, there always exists some $j$ such that $g_j^\sigma$ yields a different DAG in $\dag^\sigma(\din, \dout)$. 
We define 
\begin{equation}\label{eq:def.dstar}
    d^*_\sigma = \max_{j \in [p]} |\Pa_j (G^*_\sigma) \cup \Ch_j(G^*_\sigma)|, \quad d^* = \max_{\sigma \in \bbS^p} d^*_\sigma,
\end{equation}
where $d^*$ will be used later in Theorem~\ref{th:ges.path}. 

\begin{lemma}\label{lm:pigeonhole}
Assume  $d^*_\sigma \leq \din$  and $\min\{ d^*_\sigma \din + 1, p \} \leq \dout$.  
For any $G \in \dag^\sigma(\din, \dout)$ such that $G \neq G^*_\sigma$,   there exists some $j \in [p]$ such that $g_j^\sigma(G) \in \dag^\sigma(\din, \dout)$ and $g_j^\sigma(G) \neq G$. 
\end{lemma}
\begin{proof}
The key idea of the proof is to use the pigeonhole principle multiple times to derive the contradiction. 
See Supplement D.2. 
\end{proof}

\begin{corollary}\label{coro:path.dag}
Let $\sigma \in \bbS^p$.  Assume that $d^*_\sigma \leq \din$  and $\min\{ d^*_\sigma \din + 1, p \} \leq \dout$.  
\begin{enumerate}[label=(\roman*)]
\item There exists a function $g^\sigma \colon \dag^\sigma(\din, \dout) \rightarrow \dag^\sigma(\din, \dout)$  such that for any $G \neq G^*_\sigma$,  
$g^\sigma(G) = g^\sigma_j(G) \neq G$ for some $j \in [p]$ and $(g^\sigma)^k(G) = G^*_\sigma$  for some $k \leq (d^*_\sigma + \din) p$. 
\item If~\eqref{eq:conj.scorej} holds for each $j \in [p]$, Condition~\ref{cond:unimodal} holds for the triple $(\dag^\sigma(\din, \dout), \nb, e^{\score})$ with $t_1 = 3$ and $t_2 = t$.   
\end{enumerate}  
\end{corollary}

\begin{proof} 
See Supplement D.3. 
\end{proof} 

We are now ready to construct a canonical transition function $g \colon \cpdag(\din, \dout) \rightarrow \cpdag(\din, \dout)$ for the structure learning problem using operators $\{g_j^\sigma\colon j \in [p] , \sigma \in \bbS^p \}$.
If $\cE$ contains a minimal I-map of $G^*$, we define $g(\cE)$ using 
Chickering algorithm~\citep{chickering2002optimal}; see Lemma D3 in the supplement.  
If not, by the definition of $\cpdag(\din, \dout)$, there exists $G \in \cE \cap \dag^\sigma(\din, \dout)$ for some $\sigma \in \bbS^p$, and we can define $g(\cE)$ using the function $g^\sigma$ constructed for the DAG selection problem. 
But note that we need to fix the DAG representation of each $\cE$ so that $g(\cE)$ can be defined uniquely. 
We give an explicit construction of $g$ in the proof of Theorem~\ref{th:ges.path},  the main result for this section.    

\begin{theorem}\label{th:ges.path}
Assume that $d^* \leq \din$ and $\min\{ d^* \din + 1, p \} \leq \dout$. Then,    $G^*_\sigma \in \dag^\sigma(\din, \dout)$ for each $\sigma \in \bbS^p$. 
Further, there exists a function $g \colon \cpdag(\din, \dout) \rightarrow \cpdag(\din, \dout)$  such that $g(\cE^*) = \cE^*$ and  the following  hold for any $\cE \in \cpdag(\din, \dout) \setminus \{\cE^*\}$.
\begin{enumerate}[label=(\roman*)]
    \item $g(\cE) = \EG{ g_j^\sigma( G ) }$ for some $j \in [p]$, $\sigma \in \bbS^p$ and  $G \in \cE \cap \dag^\sigma(\din, \dout)$ such that $g_j^\sigma(G) \neq G$. 
    \item There exist  $k \leq (d^* + \din) p$ and $k \leq \ell \leq (2d^* + \din) p$ such that $g^k(\cE) = G^*_\sigma$ for some $\sigma \in \bbS^p$ and $g^\ell(\cE) = \cE^*$.  
\end{enumerate}
\end{theorem}
\begin{proof}
See Supplement D.4. 
\end{proof}

We conclude this section with the following corollary, which shows that to establish part (ii) of Condition~\ref{cond:unimodal} for the sparse structure learning problem, it only remains to prove that~\eqref{eq:conj.scorej} holds for all $j$ and $\sigma$ simultaneously. This will be done rigorously in the next section.  
 
\begin{corollary}\label{coro:path.cpdag}
Assume $d^* \leq \din$, $\min\{ d^* \din + 1, p \} \leq \dout$ and $\score$ is score equivalent so that we can define $\score(\cE) = \score(G)$ using any $G \in \cE$. 
If~\eqref{eq:conj.scorej} holds for each $\sigma \in \bbS^p$ and each $j \in [p]$, part (ii) of Condition~\ref{cond:unimodal} holds for the triple $(\cpdag(\din, \dout), \nb, e^{\score})$ with $t_2 = t$.    
\end{corollary}

\begin{proof} 
See Supplement D.5. 
\end{proof} 

\section{High-dimensional consistency of an empirical Bayes model for structure learning} \label{sec:high} 

\subsection{Model, prior and posterior distributions}\label{sec:cpdag.model}   
%We derive our target posterior distributions through an empirical Bayes approach. 
Let $\X$ be an $n \times p$ data matrix where each row is an i.i.d. copy of a normal random vector $\sX = (\sX_1, \dots, \sX_p)$.  (The font for the random vector $\sX$ and that for the data matrix $\X$ are different.)   
Assume that, given a  DAG $G$, the distribution of $\sX$ can be described by the structural equation model (SEM), 
\begin{equation}\label{eq:sem}
\sX = B^\top \sX + \mathsf{e}, \quad \mathsf{e} \sim \NM_p(0, \Omega), 
\end{equation}
for some $(B, \Omega) \in \chol(G )$, where 
\begin{equation}\label{eq:space.chol}
\begin{aligned}
\chol(G ) =  \; & \{ (B, \Omega) \colon B \in \bbR^{p \times p} , \,  B_{i j} = 0 \text{ if }  i \rightarrow j \notin G, \text{ for any } i, j \in [p]; \\ 
 &  \Omega = \diag(\omega_1, \dots, \omega_p), \, \omega_i > 0 \text{ for any } i \in [p]\}. 
\end{aligned}
\end{equation}
That is, each $\sX_j$ follows a linear regression model where explanatory variables with nonzero regression coefficients must be parents of node $j$ in $G$. The matrix $B$ is often called the weighted adjacency matrix. 
We can equivalently express~\eqref{eq:sem} as 
\begin{equation}\label{eq:mod.chol}
\sX  \sim \NM_p(0, \Sigma(B, \Omega)), \text{ where } \Sigma(B, \Omega) = (I  - B^\top)^{-1} \Omega  (I - B)^{-1}
\end{equation}
is called the modified Cholesky decomposition ($I$ denotes the identity matrix). 
The SEM representation of the Gaussian DAG model is used frequently in the literature~\citep{drton2011global, van2013ell, aragam2019globally}.

Let $\pi_0(B, \Omega \mid G)$ denote the conditional prior distribution with support $\chol(G)$.   
It suffices to specify it for $\{ (\beta_j(G), \omega_j) \colon j = 1, \dots, p \}$, where $\beta_j(G) $ is the subvector of the $j$-th column of $B$ with entries indexed by $\Pa_j(G)$, and $\omega_j$ is the $j$-th diagonal element of $\Omega$.  
We use the empirical prior proposed by~\citet{lee2019minimax}, which is an extension of the empirical variable selection model of~\citet{martin2017empirical}. 
Our prior assumes that, given $G$, $(\beta_1(G), \omega_1), \dots, (\beta_p(G), \omega_p)$ are independently distributed according to 
\begin{align*}
\pi_0(\omega_j \mid G)  \propto \;&  \omega_j^{- \kappa /2 - 1},   \\
\beta_j(G)  \mid \Pa_j(G) = S_j, \, \omega_j   \sim \;&  \NM_{|S_j|} \left(  (\X_{S_j}^\top \X_{S_j})^{-1} \X_{S_j}^\top X_j , \, \frac{\omega_j}{\gamma} (\X_{S_j}^\top \X_{S_j})^{-1}  \right),   
\end{align*}
where $\gamma > 0, \kappa \geq 0$ are hyperparameters, $\X_j$ denotes the $j$-th column of the data matrix $\X$, and $\X_S$ is the submatrix containing columns indexed by $S$.
Next, we  compute the marginal likelihood of $G$ by integrating out $(B, \Omega)$ and using a fractional exponent $\alpha \in (0, 1)$ to offset the overuse of data caused by the empirical prior. 
The resulting fractional marginal likelihood  is given by $f_\alpha(G) = \prod_{j = 1}^p f_{\alpha, j}( \Pa_j(G) )$, where 
\begin{equation*}   
f_{\alpha, j}( S ) =   \left( 1 +  \alpha \gamma^{-1}   \right)^{- |S| / 2 }    \left\{ X_j^\top   (I - \X_S (\X_S^\top \X_S)^{-1} \X_S^\top) X_j \right\}^{ -(\alpha n + \kappa)/2}. 
\end{equation*} 
More details about this empirical prior are given in Supplement F.1. 

For sparse DAG selection  with ordering $\sigma$,  the state space is $\dag^\sigma(\din, \dout)$. For each $G$ on this space, we specify its prior probability by 
\begin{equation}\label{eq:prior.dag}
\pi_0^\sigma (G) \propto \left( c_1 p^{c_2} \right)^{- |G|}, 
\end{equation}
where $ c_1 > 0, c_2 \geq 0$ are hyperparameters.  We can then calculate the posterior distribution by $\post^\sigma(G) \propto   \pi_0^\sigma( G) f_\alpha(G)$. Using the fractional marginal likelihood $f_\alpha$, we get   
\begin{align}
  \post^\sigma(G) \propto \;&  e^{\score(G)}  \ind_{ \dag^\sigma(\din, \dout)}(G),  \text{ where }  \label{eq:dag.sigma}  \\
 \score(G) =\;&  \sum_{j=1}^p  \scorej ( \Pa_j(G) ), \text{ and }  e^{\scorej(S)} = (c_1 p^{c_2} )^{- |S| } f_{\alpha, j}(S).  \label{eq:post.modular} 
\end{align} 

For the sparse structure learning problem,  we use the  prior
\begin{equation}\label{eq:prior.equivalence} 
\pi_0(\cE, G) \propto   \left( c_1 p^{c_2} \right)^{- |G|}   \pi_0(G \mid \cE)  \ind_{ \cE }(G ), 
\end{equation}
where $\pi_0(G \mid \cE)$ satisfies  $\sum_{G \in \cE } \pi_0(G \mid \cE) = 1$. 
Denote the corresponding posterior distribution by $\post$. 
Marginalizing out $G$ from $\pi_n(\cE, G)$, we get  
\begin{align*}
\pi_n(\cE) \propto    \sum_{G \in \cE} \pi_0(G \mid \cE)  e^{\score(G)}. 
\end{align*}
In Lemma~\ref{lm:markov.equiv} below, we prove that $\psi$ yields the same value for any Markov equivalent DAGs. 
Hence, we can define $\score(\cE) = \score(G)$ using any $G \in \cE$, and $\pi_n(\cE)$ can be expressed by 
\begin{equation}\label{eq:post.cpdag}
\pi_n(\cE) \propto  e^{\score(\cE)} \ind_{\cpdag(\din, \dout)}(\cE).
\end{equation}
The indicator function in~\eqref{eq:post.cpdag} serves to remind us of the restricted search space. 
We do not consider estimating the DAG or ordering from $\post$. Indeed, for $G \in \cE$, $\post(G)$ depends on the conditional prior probability $\pi_0(G \mid \cE)$, which we leave unspecified.

\begin{lemma}\label{lm:markov.equiv}
The  function $\score$ defined by~\eqref{eq:post.modular} satisfies that $\psi(G) = \psi(G')$ whenever $G$ and $G'$ are Markov equivalent DAGs. 
\end{lemma}
\begin{proof}
See Supplement F.2. 
\end{proof} 

We will refer to $\scorej( \Pa_j ), \score(G), \score(\cE)$ as the scores of $\Pa_j, G$ and $\cE$, respectively.  
Note that a scoring criterion derived from a nodewise  normal-inverse-gamma prior  for $(B, \Omega) \mid G$ does not necessarily have the property given in Lemma~\ref{lm:markov.equiv}. For non-empirical prior distributions, see~\citet{geiger2002parameter} and~\citet{peluso2020compatible} for related results.    
 
\subsection{High-dimensional setup} \label{sec:high.setup}
Let $G^*$ denote the true DAG model and $\cE^* =\EG{G^*}$ be the true equivalence class that we want to recover from the data. 
Assume that each row of $\X$ is drawn independently from $\NM_p(0, \Sigma^*)$, a normal distribution perfectly Markovian w.r.t. $G^*$.   
We will show  $\post$ defined in~\eqref{eq:post.cpdag} concentrates on $\EG{G^*}$ by first proving that for each $\sigma$, $\post^\sigma$ defined in~\eqref{eq:dag.sigma} concentrates on the minimal I-map $G^*_\sigma$.  
Due to normality, $G^*_\sigma$ can be equivalently defined by using the modified Cholesky decomposition.   

\begin{definition}\label{def:imap}
Let $\Sigma^*$ be  positive definite  and $\NM_p(0, \Sigma^*)$ be perfectly Markovian w.r.t. some DAG $G^*$. 
For each $\sigma \in \bbS^p$,  let $\chol(\sigma) = \cup_{G \in \dag^\sigma} \chol(G)$.  By Lemma C6, we can define $(B_\sigma^*, \Omega_\sigma^*)$ to be the unique pair in $\chol(\sigma)$  such that    
$$(I - (B_\sigma^*)^\top)^{-1} \Omega_\sigma^* (I - B_\sigma^*)^{-1} = \Sigma^*.$$ 
Define $G_\sigma^*$ to be the DAG such that $i \rightarrow j \in G_\sigma^*$ if and only if $(B_\sigma^*)_{ij} \neq 0$,  which, by Lemma C5, is the minimal I-map of $G^*$ with ordering $\sigma$.  
\end{definition}

Consider a high-dimensional setting with  $p = p(n)$ tending to infinity. 
The true DAG model $G^*$, true covariance matrix $\Sigma^*$ and prior parameters $c_1, c_2, \alpha, \gamma, \din, \dout$ are all implicitly indexed by $n$.  
We say a constant is universal if it does not depend on $n$. 
To derive our consistency results, we need to make a few assumptions on the parameters and $\Sigma^*$. 

\begin{enumerate}[label=(A\arabic*), ref=(A\arabic*)]
\item There exist $\vmin = \vmin(n), \vmax = \vmax(n) > 0$  and a universal constant $\delta_0 > 0$ such that 
\begin{align*}
0 < \frac{\vmin}{(1 - \delta_0)^2} \leq \lmin(\Sigma^*) \leq \lmax(\Sigma^*) \leq  \frac{\vmax}{(1 + \delta_0)^2}, 
\end{align*} 
where $\lmin, \lmax$ denote the smallest and largest eigenvalues, respectively. \label{A:eigen} 
\item The sparsity parameter $\din$ and $n,p$ satisfy that $\din \log p = o(n)$. \label{A:np} 
\item Prior parameters satisfy that  $\kappa \leq n$, $1 \leq c_1 \sqrt{ 1 + \alpha / \gamma} \leq p$,  and 
$$c_2 \geq (\alpha + 1 )(4 \din + 6) + t$$ for some universal constant $t > 0$.  \label{A:prior}
\item Assumption on the maximum in-degree of $G^*_\sigma$. 
\begin{enumerate}[label=(A4.\arabic*), ref=(A4.\arabic*)]
\item Let $\nu_0 = 4  \vmax^2  \vmin^{-4} (\vmax - \vmin)^2$. 
For some $\sigma \in \bbS^p$, $ (\nu_0 + 1)   \max_{j \in [p]} |\Pa_j(G^*_\sigma)|   \leq  \din$. \label{A:size2} 
\item Assumption~\ref{A:size2} holds for every $\sigma \in \bbS^p$.\label{A:size} % that is, $(\nu_0 + 1) d^* \leq \din$. 
\end{enumerate}
\item Assumption on $B^*_\sigma, B^*$. 
\begin{enumerate}[label=(A5.\arabic*), ref=(A5.\arabic*)]
\item There exists a universal constant  $\Cb > 0$ such that  for some $\sigma \in \bbS^p$, 
\begin{equation}\label{eq:beta.min} 
\min \left\{  | (B_\sigma^*)_{ij} |^2 \colon  (B_\sigma^*)_{ij} \neq 0   \right\} 
    \geq    5 (  \Cb + 4 c_2  )     \frac{  \vmax^2 \log p }{\alpha \vmin^2 n}, 
\end{equation}
where $B_\sigma^*$ is given by Definition~\ref{def:imap}.  \label{A:beta2}
\item There exists a universal constant $\Cb > 0$ such that~\eqref{eq:beta.min} holds for every $\sigma \in \bbS^p$.  \label{A:beta}
\end{enumerate}
\end{enumerate}

The first three assumptions are standard and commonly used in high-dimensional statistical theory. Assumption~\ref{A:eigen} is the standard restricted eigenvalue condition~\citep{bickel2008regularized}. 
Assumption~\ref{A:np} controls the growth rates of $p$ and $\din$ (which determines the maximum model size for nodewise variable selection), and, together with Assumption~\ref{A:prior}, ensures that  we cannot overfit the data; recall from~\eqref{eq:prior.equivalence} that the hyperparameter $c_2$ controls the penalty on the model size, so it plays the same role as the tuning parameter in the penalized likelihood methods. 
Such assumptions (especially a condition similar to $\din \log p = o(n)$) are required for most high-dimensional problems including variable selection~\citep{wainwright2009sharp, yang2015minimax, yang2016computational, jeong2021posterior},  stochastic block model~\citep{gao2020general}, 
covariance matrix estimation~\citep{lam2009sparsistency, sun2013sparse, pati2014posterior}, undirected Gaussian graphical models~\citep{raskutti2008model, banerjee2015bayesian, liu2017support} and DAG selection~\citep{cao2019posterior, lee2019minimax}; see~\citet{banerjee2021bayesian} for a recent review. 
Note that the numerical constants in our assumptions are very conservative.  
For example,  Assumption~\ref{A:prior} suggests that $c_2$ should grow linearly with $\din$, but in practice, one can use  some $c_2$ much smaller than $4 \din$, which we will illustrate using a simulation study in Section~\ref{sec:third.sim.c2}.   

Assumption~\ref{A:size2} requires that the maximum in-degree of the ``true model'' for DAG selection with ordering $\sigma$ is sufficiently small compared with $\din$.  
It is similar to Assumption D of~\citet{yang2016computational} and is technically needed to show that an MH sampler using add-delete-swap moves cannot get stuck at DAG models with maximum in-degree equal to $\din$.  
But unlike their setup, we assume both lower and upper restricted eigenvalues are available, which enables us to avoid imposing an irrepresentability condition as in~\citet[Assumption D]{yang2016computational}. 
Assumption~\ref{A:size}  restricts the maximum in-degree of all minimal I-maps of $G^*$, which is allowed to   have the same order as $\din$, if $\vmax, \vmin$ defined in Assumption~\ref{A:eigen} can be bounded by universal constants. 

Assumption~\ref{A:beta2} is the well-known beta-min condition for DAG selection with ordering $\sigma$~\citep{cao2019posterior, lee2019minimax}. 
According to Definition~\ref{def:imap}, the SEM representation~\eqref{eq:sem} holds for $(B, \Omega) = (B^*_\sigma, \Omega^*_\sigma)$. Hence, Assumption~\ref{A:beta2} just means that all nonzero regression coefficients (i.e., signal sizes) of the true SEM  with ordering $\sigma$ are sufficiently large.   
Assumption~\ref{A:beta} is for structure learning and assumes the beta-min condition holds uniformly over all $\sigma \in \bbS^p$; this is often known as the strong beta-min or permutation beta-min condition~\citep{uhler2013geometry} and was used in~\citet{van2013ell} and~\citet{aragam2019globally}. 
If $p$ and $\Sigma^*$ are fixed, which implies $B^*_\sigma$ is fixed for all $\sigma \in \bbS^p$, then  Assumption~\ref{A:beta} can always be satisfied by choosing some large $n$.  
We need the strong beta-min condition (or some similar assumption)  since we want to first establish that with high probability, for every  $\sigma \in \bbS^p$, the minimal I-map $G^*_\sigma$ has the highest score among all DAGs in $\dag^\sigma(\din, \dout)$, which is needed for proving consistency results for structure learning. 
For methods based on CI tests, a similar assumption, known as ``strong faithfulness'', is commonly used~\citep{nandy2018high} (strong beta-min condition essentially replaces partial correlations in strong faithfulness with partial regression coefficients).  
\citet{uhler2013geometry} showed that the volume of normal distributions that are strongly faithful is very small. 
Though strong faithfulness and strong beta-min condition are not directly comparable, both seem to be fairly restrictive~\citep{van2013ell}. 
Unfortunately, without them, we cannot preclude the possibility that GES or local MH algorithms get trapped at local modes; see Example~\ref{ex:slow1} in Section~\ref{sec:rapid.rwges}. A discussion on how to overcome such limitations is given in Supplement I. 
We end this subsection with one more remark on Assumption~\ref{A:eigen}.    

\begin{remark}\label{rmk:beta.min}
The restricted eigenvalue condition can be used to obtain some useful bounds related to $B^*_\sigma$ and $\Omega^*_\sigma$. Write $\Omega^*_\sigma = \diag (\omega^*_{\sigma, 1}, \dots, \omega^*_{\sigma, p})$. 
The decomposition~\eqref{eq:mod.chol} implies that $\omega^*_{\sigma, k} \in (\vmin, \vmax)$ for any $\sigma \in \bbS^p$ and $k \in [p]$ since the diagonal elements of $\Sigma^*$ and $(\Sigma^*)^{-1}$ can be bounded by the extreme eigenvalues of $\Sigma^*$. 
Further, we can bound the $\ell^2$-norm of the true regression coefficients for node $j$ by $\sum_{i \in [p]} (B^*_\sigma)^2_{ij} \leq \omega^*_{\sigma, j} / \vmin - 1$, using the fact that the operator norm is no less than the $\ell^2$-norm of any column. 
\end{remark}
  
\subsection{Strong selection consistency results} \label{sec:node.sel}
For a general model selection problem, we say a Bayesian procedure has strong selection consistency if the posterior probability of the true model converges to $1$ in probability with respect to the true data-generating probability measure~\citep{johnson2012bayesian, narisetty2014bayesian, cao2019posterior}. 
By part (ii) of Theorem~\ref{th:path},  to prove the strong selection consistency, we only need to show that Condition~\ref{cond:unimodal} is satisfied for some universal $t_2 > t_1$.   

We begin with the strong selection consistency for  nodewise variable selection and DAG selection problems. 
It turns out that we only need~\eqref{eq:conj.scorej} holds for any $j \in [p]$ and $\sigma \in \bbS^p$. 
By Corollary~\ref{coro:path.cpdag}, this consistency property of $\{ g_j^\sigma \colon j \in [p], \sigma \in \bbS^p \}$  is also key to the verification of Condition~\ref{cond:unimodal} for structure learning. 
The complete proof for Theorem~\ref{th:sel0} is highly technical, and the most involved step  is to establish an analogous consistency result for a single variable selection problem using our empirical prior, which is treated in detail in Supplement E and may be of independent interest.  

\begin{theorem}\label{th:sel0}
Let  $\X \in \bbR^{n \times p}$ have i.i.d. rows drawn from  $\NM_p(0, \Sigma^*)$, which is  perfectly Markovian w.r.t. $G^*$.  
Suppose Assumptions~\ref{A:eigen},~\ref{A:np},~\ref{A:prior},~\ref{A:size} and~\ref{A:beta} hold. 
Let $t > 0$ be the universal constant given in Assumption~\ref{A:prior} and assume $\Cb  \geq 8t / 3$. 
For sufficiently large $n$, with probability at least $1 - 3p^{-1}$,  the following statements hold. 
\begin{enumerate}[label=(\roman*)]
    \item  Consistency  of the operators $\{ g_j^\sigma \colon j \in [p], \sigma \in \bbS^p \}$ given in Definition~\ref{def:gj}: 
    \begin{equation*} 
    \min \left\{      \scorej( g_j^\sigma(S)) - \scorej(S) 
    \colon \sigma\in \bbS^p,  j \in [p], \, S \in \model^\sigma(j,  \, \din) \setminus \{  S^*_{\sigma, j} \}   \right\} \geq t \log p,
    \end{equation*}
    where   $\scorej$ is given in~\eqref{eq:post.modular} and $S^*_{\sigma, j} = \Pa_j(G^*_\sigma)$.  
    \item  If $t > 2$,  we have the strong selection consistency of nodewise variable selection, 
    \begin{align*}
       \min_{\sigma \in \bbS^p} \min_{j \in [p]} \frac{  \exp (\scorej(  S^*_{\sigma, j} )) }{   \sum_{S \in \model^\sigma(j,  \, \din)}  \exp (\scorej(  S  ))   }  \geq 1 - p^{-(t - 2)},
    \end{align*}
    where $ \model^\sigma(j,  \, \din)$ is defined in~\eqref{eq:def.model}. 
    \item  If $t > 3$, we have the strong selection consistency of sparse DAG selection, 
    \begin{align*}
    \min_{\sigma \in \bbS^p}  \frac{\exp (\score (  G^*_\sigma )) }{ \sum_{G \in \dag^\sigma(\din, \dout)} \exp (\score(G)) } \geq 1 - p^{-(t - 3)}, 
    \end{align*}
    where $\score(G)$ is defined in~\eqref{eq:post.modular}. 
\end{enumerate}

\end{theorem}
\begin{proof}
See Supplement F.3. 
\end{proof}

\begin{remark}\label{rmk:t}
The universal constant $t$ can be chosen arbitrarily large.  Given any $t > 0$, in order that Theorem~\ref{th:sel0} holds, we can always choose some $c_2$ that has  same order as $\din$ and assume that the universal constant $\Cb$ in Assumption~\ref{A:beta} is sufficiently large.   
\end{remark}

As a corollary, the strong selection consistency for a single DAG selection problem with ordering $\sigma$ can be obtained by replacing Assumptions~\ref{A:size} and~\ref{A:beta} with Assumptions~\ref{A:size2} and~\ref{A:beta2}.  
This result was also proved in~\citet{lee2019minimax} under similar assumptions, but the method we use is different  (the primary goal of~\citet{lee2019minimax} was to derive minimax posterior convergence rates for the weighted adjacency matrix). 
Note that if $\sigma$ is an ordering of $G^*$, then $G^*_\sigma = G^*$. 

\begin{corollary}\label{coro:dag}
Let  $\X \in \bbR^{n \times p}$ have i.i.d. rows drawn from the distribution $\NM_p(0, \Sigma^*)$, which is  perfectly Markovian w.r.t. $G^*$.  
Suppose Assumptions~\ref{A:eigen},~\ref{A:np},~\ref{A:prior},~\ref{A:size2} and~\ref{A:beta2} hold  for some  $t > 3$ and $\Cb \geq 8t / 3$. 
Let $\sigma$ be as given in Assumptions~\ref{A:size2} and~\ref{A:beta2}. 
For sufficiently large $n$, with probability at least $1 - 3p^{-1}$, 
\begin{align*}
   \frac{ \exp (\score (  G^*_\sigma ) )}{ \sum_{G \in \dag^\sigma(\din, \dout)} \exp (\score(G)) } \geq 1 - p^{-(t - 3)}
\end{align*}
\end{corollary}
\begin{proof}
The proof is wholly analogous to that for Theorem~\ref{th:sel0}. 
\end{proof}

In order to show Condition~\ref{cond:unimodal} holds and use Theorem~\ref{th:path} to prove the strong selection consistency of sparse structure learning,   it only remains to invoke Lemma~\ref{lm:bound.nb} to bound the size of  $\nbg(\cdot)$ and then apply Corollary~\ref{coro:path.cpdag}. 
Recall the definition of $d^*_\sigma$ and $d^*$ given  in~\eqref{eq:def.dstar}. 
  
\begin{theorem}\label{th:cpdag.consist}
Let  $\X \in \bbR^{n \times p}$ have i.i.d. rows drawn from  $\NM_p(0, \Sigma^*)$, which is  perfectly Markovian w.r.t. $G^*$.  
Suppose $d^* \leq \din$, $d^* \din + 1  \leq \dout$ and $\din  + \dout \leq t_0 \log_2 p$ for some universal constant $t_0 > 0$, and   Assumptions~\ref{A:eigen},~\ref{A:np},~\ref{A:prior},~\ref{A:size} and~\ref{A:beta} hold with $\Cb  \geq 8t / 3$ and $t > t_0 + 3$. 
For sufficiently large $n$, with probability at least $1 - 3p^{-1}$,  
\begin{align*}
   \frac{  \exp( \score (  \cE^* )) }{ \sum_{\cE \in \cpdag(\din, \dout)} \exp ( \score(\cE) )} \geq 1 - p^{-(t - t_0 - 3)}, 
\end{align*}
where $\score(\cE) = \score(G)$ for any $G \in \cE$ and $\score(G)$ is defined in~\eqref{eq:post.modular}.  
Further, the greedy search on $( \cpdag(\din, \dout), \nbg, e^\score)$   returns $\cE^*$ regardless of the initial state. 
\end{theorem}
\begin{proof}
See Supplement F.4. 
\end{proof}

\begin{remark}\label{rmk:order.d}
The assumption $\din + \dout = O(\log p)$  is mild, since  the total number of edges in the DAG may have order $p$ even if $\din + \dout = O(1)$. In light of Assumption~\ref{A:size}, we may assume $d^*, \din$ have approximately the same order. Thus, roughly speaking, the assumptions of Theorem~\ref{th:cpdag.consist} imply that $d^*, \din$ cannot grow faster than $\sqrt{\log p}$. 
\end{remark} 

\subsection{Consistency results for sub-Gaussian random matrices} \label{sec:subgauss}
The normality assumption on the true distribution of $\sX$ can be relaxed. We can extend the consistency result obtained in Theorem~\ref{th:sel0}  to the case where $\X$ is a sub-Gaussian random matrix (we still consider the posterior distributions defined in Section~\ref{sec:cpdag.model}). 
Let each row of $\X$ be an i.i.d. copy of a random vector $\sX$ which has mean zero, covariance matrix $\Sigma^*$ and distribution $\mu$.  
Assume that $\mu$  is sub-Gaussian with sub-Gaussian parameter bounded by a universal constant, and $\NM_p(0, \Sigma^*)$ is perfectly Markovian w.r.t. a DAG $G^*$ (that is, $\mu$ is not necessarily  perfectly Markovian w.r.t. $G^*$). 
This includes the case where some node variables are Gaussian and some are discrete and bounded~\citep{lauritzen1992propagation}. 
Then, under  a set of similar assumptions, we can prove a consistency result analogous to Theorem~\ref{th:sel0}(i); see Theorem F1 in Supplement F.5. 
By Corollary~\ref{coro:path.cpdag}, this proves part (ii) of  Condition~\ref{cond:unimodal}, and other strong selection consistency results in Theorem~\ref{th:sel0} follow.  

The main idea of the proof of Theorem F1 is similar to the Gaussian case.  We first generalize the variable selection results of~\citet{yang2016computational} to random matrices, which is performed in Supplement E.3. 
However,  the proof techniques are very different from the Gaussian case in that we need to use random matrix theory~\citep{vershynin2010introduction} and error propagation results to show that all minimal I-maps of $G^*$ can be recovered from the empirical covariance matrix.  
The key distinction between the two scenarios is that in the sub-Gaussian case uncorrelatedness does not imply independence. Consequently, some calculations are more involved, and we need to require a slightly stronger assumption on $c_2$: in the sub-Gaussian case, we require $  \din \vmax^4 / \vmin^6 = O(c_2)$, while in the Gaussian case we only need $\din = O(c_2)$.

\section{Mixing time results for Bayesian structure learning} \label{sec:mcmc}

\subsection{Rapid mixing of the \RWGES{} sampler}\label{sec:rapid.rwges}  
Recall that \RWGES{} is simply the random walk MH algorithm defined by~\eqref{eq:general.MH} with $h \equiv 1$ and the triple $(\cpdag (\din, \dout), \nb, \post)$ where $\post$ is given by~\eqref{eq:post.cpdag}. 
 In the proof of Theorem~\ref{th:cpdag.consist}, we have verified that  Condition~\ref{cond:unimodal} holds, and thus we can apply the mixing time bounds in Section~\ref{sec:mix.path} to obtain the main result of this work, rapid mixing of  \RWGES{}.

\begin{theorem}\label{th:main.rapid}
Consider the setting of Theorem~\ref{th:cpdag.consist}, and let  $\pi_{\rm{min}} = \min_{\cE \in \cpdag(\din, \dout)} \post(\cE) $. 
Let $\bP$ denote the transition matrix of the \RWGES{} sampler  and $\bP_{\rm{lazy}}$ denote its lazy version. 
For sufficiently large $n$, with probability at least $1 - 3p^{-1}$, we have  
\begin{align*}
    \Tmix (\bP_{\rm{lazy}}) \leq C  t_0  p^{t_0 + 2} ( \log p )   \log  \left(  \frac{4}{ \pi_{\rm{min}} } \right),   
\end{align*}
for some universal constant $C$, where $t_0$ is as given in Theorem~\ref{th:cpdag.consist}.  
\end{theorem}

\begin{proof}
See Supplement G.1. 
\end{proof}
 
\begin{corollary}\label{lm:pi.min}
Suppose Assumptions~\ref{A:eigen} and~\ref{A:np} hold. We have 
 \begin{align*}
        \min_{\cE \in \cpdag(\din, \dout)}  \frac{ \post(\cE ) }{ \post (\cE^* ) }        \geq   \left( c_1 p^{c_2} \sqrt{1 + \alpha/\gamma}   \right)^{- p (\din + d^*) } \left( \frac{2\vmax}{\vmin} \right)^{ - p (\alpha n + \kappa) / 2}. 
 \end{align*}
Hence, under the setting of Theorem~\ref{th:main.rapid}, the mixing time of  the \RWGES{} sampler can be bounded by a polynomial of $n$ and $p$. 
\end{corollary}
\begin{proof}
See Supplement G.2. 
\end{proof}

\begin{remark}\label{rmk:rwges}
Corollary~\ref{lm:pi.min} implies that \RWGES{} is rapidly mixing with high probability. 
The term $\log \pi_{\rm{min}}$ in the mixing time bound is only used to handle the worst scenario where the chain starts from the state with minimum posterior probability. 
If the chain starts from some ``good'' estimate, the actual mixing rate of the chain can be much faster; see~\citet[Proposition 1]{sinclair1992improved}. 
\end{remark}

If in the beta-min condition, we only assume that the minimum edge weight of $B^*$ (the weighted adjacency matrix of the true DAG $G^*$) is sufficiently large, the rapid mixing of \RWGES{} does not hold. 
It is not difficult to construct an explicit example where \RWGES{} is slowly mixing. 
In the following example, we let $p = 3$ be fixed and show that the mixing time grows exponentially in $n$.  One can extend our example to the case $p = n$ by adding variables $\sX_4, \dots, \sX_{n}$ such that, for any $j = 4, \dots, n$, the observed vector $X_j$  is exactly orthogonal to all the other column vectors of the data matrix. 

\begin{example}\label{ex:slow1}
Assume $p = 3$ and the true SEM is given by 
\begin{align*}
    X_1 = z_1, \quad  X_2 =  b_1 X_1 + z_2,  \quad X_3 =b_2 X_2 + z_3, 
\end{align*}
where $z_1, z_2, z_3$ are vectors orthogonal to each other and $\norm{z_j}_2^2 = n$ for each $j$.  
Thus, we can let the true DAG  $G^*$  be  $1 \rightarrow 2 \rightarrow 3$.  
Suppose the prior parameters satisfy that $\din = \dout = 2$,  $c_2 = \sqrt{ n }$, $\kappa = 0$, and $c_1, \alpha, \gamma$ are fixed constants such that $c_1 \sqrt{1 + \alpha / \gamma} = 1$.    
Assume the true regression coefficients $b_1, b_2 > 0$ are given by 
\begin{align*}
    b_1^2 = b_2^2 =  \frac{ K c_2 \log p }{\alpha n } = o(1), 
\end{align*}  
where $K$ is some large universal constant.  So, $b_1, b_2$ satisfy the bound in~\eqref{eq:beta.min}. 
Consider the DAG $\tG$ given by $1 \rightarrow 2 \leftarrow 3$, which has $\EG{\tG} = \{ \tG \}$. 
The topological ordering of $\tG$ can be chosen to be $\sigma=(1, 3, 2)$, and the minimal I-map $G^*_\sigma$ is a complete DAG.  
One can show that the edge weight of $1 \rightarrow 3$ in $G^*_\sigma$ is  $b_1 b_2$. 
It is easy to verify that $b_1^2 b_2^2 = o (c_2 n^{-1} \log p )$, so the true model fails to satisfy the strong beta-min condition.  Indeed, we can prove that \RWGES{} is slowly mixing. 
See Supplement G.4. 
\end{example} 

\subsection{Rapid mixing results for sparse DAG selection}\label{sec:rapid.dag}
Suppose the ordering is given and the search is restricted to $\dag^\sigma(\din, \dout)$. 
We can construct a random walk MH sampler using neighborhood function $\nb$ defined in~\eqref{eq:def.nb} and posterior distribution $\post^\sigma$ defined in~\eqref{eq:dag.sigma}, which is just the standard add-delete-swap MH sampler. Denote its transition matrix by $\bP^\sigma$.  
If there is no out-degree constraint, by posterior modularity, one can perform sampling for the parent set of each node separately; thus, there is no need to directly draw DAG samples.  
However, when $\dout < p $, the posterior distributions of $\Pa_1, \dots, \Pa_p$ are not independent, and this add-delete-swap sampler provides a convenient solution. 
Since by Theorem~\ref{th:sel0}(i) and Corollary~\ref{coro:path.dag},  the triple $( \dag^\sigma(\din, \dout), \nb, \post^\sigma)$ satisfies Condition~\ref{cond:unimodal}, the mixing time bound for $\bP^\sigma$ immediately follows from Theorem~\ref{th:path.better}.  

\begin{theorem}\label{th:dag.rapid}
Suppose Assumptions~\ref{A:eigen},~\ref{A:np},~\ref{A:prior},~\ref{A:size2} and~\ref{A:beta2} hold  for some $\sigma \in \bbS^p$,   $t > 3$ and $\Cb \geq 8t / 3$. 
Further, assume that  $\min\{ d^*_\sigma \din + 1, p \} \leq \dout$.  
For sufficiently large $n$, with probability at least $1 - 3p^{-1}$,  we have 
\begin{align*}
    \Tmix (\bP^\sigma_{\rm{lazy}}) \leq C    \din  p^2 \log  \left(  \frac{4}{ \pi^\sigma_{\rm{min}} } \right),   
\end{align*}
for some universal constant $C$, where $ \pi^\sigma_{\rm{min}} = \min_{G \in \dag^\sigma(\din, \dout)}  \post^\sigma(G ) $.
\end{theorem}
\begin{proof}
See Supplement G.3. 
\end{proof}

\begin{remark}
The assumptions are much weaker than those used in Theorem~\ref{th:main.rapid}.
In particular, we can allow a much larger model size for each nodewise variable selection problem. This is mainly because  for any $G \in \dag^\sigma(\din, \dout)$, we have $| \nb (G)  | = O(\din p^2)$. 
But for an equivalence class $\cE \in \cpdag(\din, \dout)$, the size of $\nbg(\cE)$ may grow exponentially in $\din + \dout$. 
\end{remark}

\subsection{Slow mixing examples for a CPDAG sampler} \label{sec:cpdag.slow}
The neighborhood $\nbg(\cE)$ used in \RWGES{} can be very large for some $\cE$, which seems to be undesirable. However, other choices of the neighborhood relation on $\cpdag$ (which may seem very reasonable) can cause the search algorithm to be trapped in sub-optimal local modes.  

A popular approach to constructing sampling algorithms on $\cpdag$ is to use the CPDAG (completed partially directed acyclic graph) representations of equivalence classes. Any equivalence class $\cE$ can be uniquely represented by a CPDAG, a partially directed acyclic graph that satisfies two conditions: (i) it has the same skeleton as any $G \in \cE$; (ii) an edge is directed if and only if the edge is directed in the same orientation in every $G \in \cE$. 
A CPDAG is also called an essential graph~\citep{andersson1997characterization}.  %so we denote the CPDAG representing $\cE$ by $\EGG{\cE}$. 
One can define local proposal moves on $\cpdag$ by modifying CPDAGs. 
However, one can easily end up with a CPDAG sampler that is slowly mixing even when $p$ is fixed and $n$ goes to infinity.  

\begin{example}\label{ex:slow2}
Let $p = 3$ and  the true data-generating DAG $G^*$ be $1 \rightarrow 3 \leftarrow 2$. Since $G^*$ is the only member in $\cE^* = \EG{G^*}$, the CPDAG of $\cE^*$  is the same as $G^*$. 
Let $\tcE$ be the equivalence class that contains all complete DAGs. It is easy to verify that the CPDAG of $\tcE$ is a complete undirected graph. 
If we define the neighborhood of $\tcE$ as all the CPDAGs that can be obtained by adding or removing a directed or undirected edge from $\tcE$, then the only CPDAGs we can move to from $\tcE$ are $1 - 2 - 3$, $1 - 3 - 2$ and $ 2 - 1 - 3$. 
However, given sufficiently large sample size, all these three CPDAGs should have much smaller score than $\tcE$. For example, the CPDAG $1 - 3 - 2$ encodes the CI relation $1 \indep 2 \mid 3$, which does not exist in $G^*$, and thus connecting nodes $1$ and $2$ should increase the score. 
See Supplement G.5 for an explicit construction of this example and another $5$-node example, where we further prove that the CPDAG sampler proposed by~\citet{castelletti2018learning} is slowly mixing. 
\end{example}

\section{Simulation studies on the \RWGES{} sampler} \label{sec:sim}

\subsection{A rapid mixing example}\label{sec:first.sim.rapid}
In this section, we present three simulation studies which illustrate the theoretical results we have proved. 
We first construct a rapid mixing example for $p = 100$ and $n = 800$. 
In order to approximately satisfy the strong beta-min condition, we randomly generate the true DAG $G^*$ such that its maximum node degree is $2$ and its largest connected sub-DAG only has $10$ nodes, and then for each edge $(i, j)$ in $G^*$, we sample $B^*_{i j}$ from the uniform distribution on $(0.5, 1.5) \cup (-1.5, -0.5)$. The DAG $G^*$ we obtain has $66$ edges, among which $24$ are directed in the CPDAG representation of $\EG{G^*}$; see Supplement H.2 for the visualization.
Each row of the data matrix $\X$ is drawn independently from $\NM_p(0, \Sigma^*)$ where $\Sigma^* = (I - (B^*)^\top)^{-1}  (I - B^*)^{-1}$. 
We use $\alpha = 0.99, \gamma = 0.01, \kappa = 0, c_1 = 1, c_2 = 2$ and 
run 20 \RWGES{} chains, all initialized at the null model, for $5\times 10^4$ iterations. All 20 runs are able to find the true equivalence class in about $10^5$ iterations, which indicates a fast mixing rate; see the left panel of Figure~\ref{fig:rapid_mixing}. 
This example illustrates that though the strong beta-min condition is restrictive, if the true DAG is sufficiently sparse and has a ``simple'' structure,  \RWGES{} can be rapidly mixing  for a moderately large sample size 
(in Supplement H.2, we use this idea to explicitly construct toy examples with $p \gg n$ that satisfy all assumptions of Theorem~\ref{th:main.rapid}). 
For comparison, we repeat the analysis by only using the first $200$ observations, and we find that $11$ chains fail to sample $\EG{G^*}$. The right panel of Figure~\ref{fig:rapid_mixing}  suggests that these 11 chains get stuck at different local modes. Notice that $\EG{G^*}$ still seems to have the largest posterior probability in this case, which implies that the beta-min condition at least holds for the true ordering (i.e., if the true ordering is known, we can recover the true DAG). 
However, since $n$ is small, the strong beta-min condition is significantly violated, which makes the posterior distribution on the space of equivalence classes highly multimodal.  

\begin{figure}[!htp]
    \centering
    \includegraphics[width=0.45\linewidth]{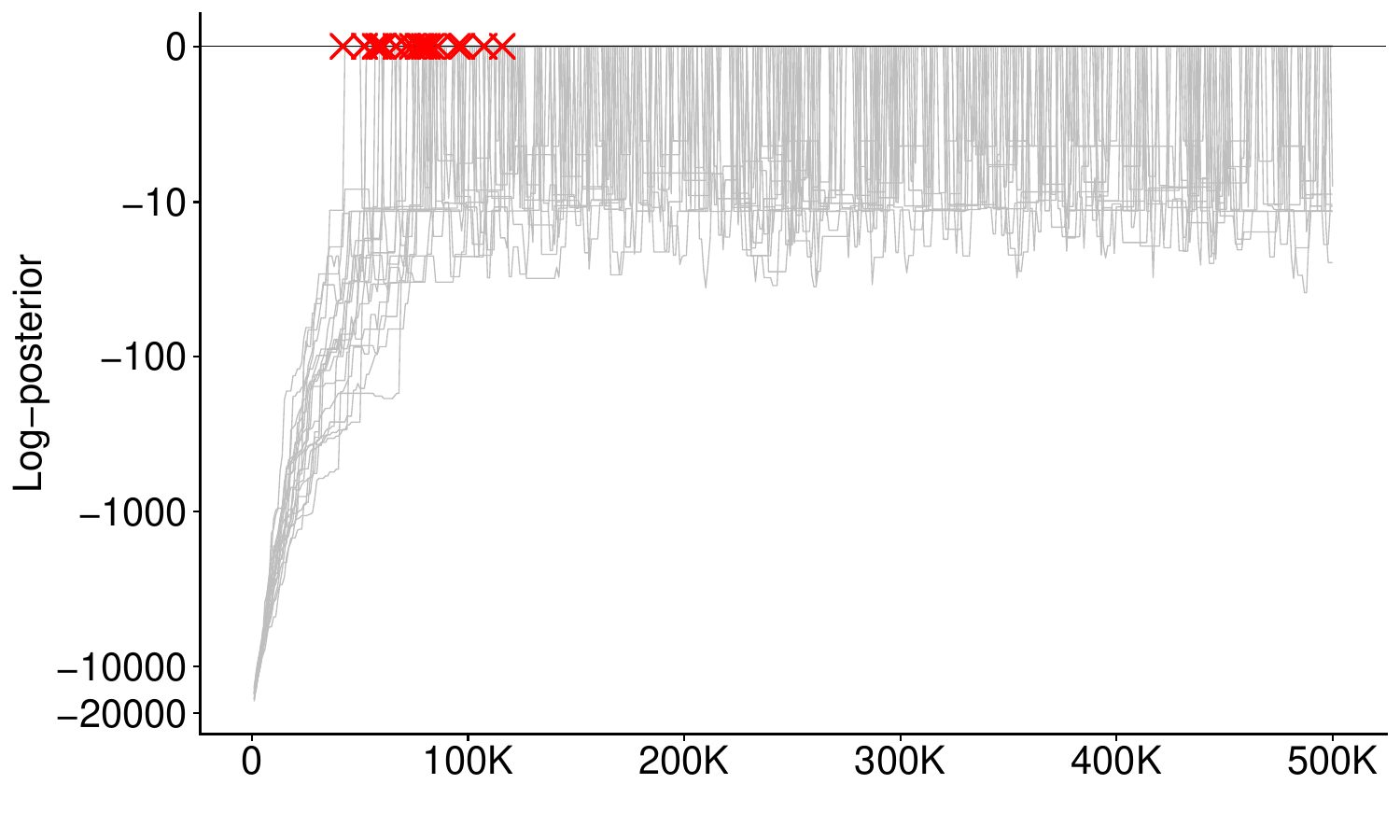} \hspace{0.5cm}
    \includegraphics[width=0.45\linewidth]{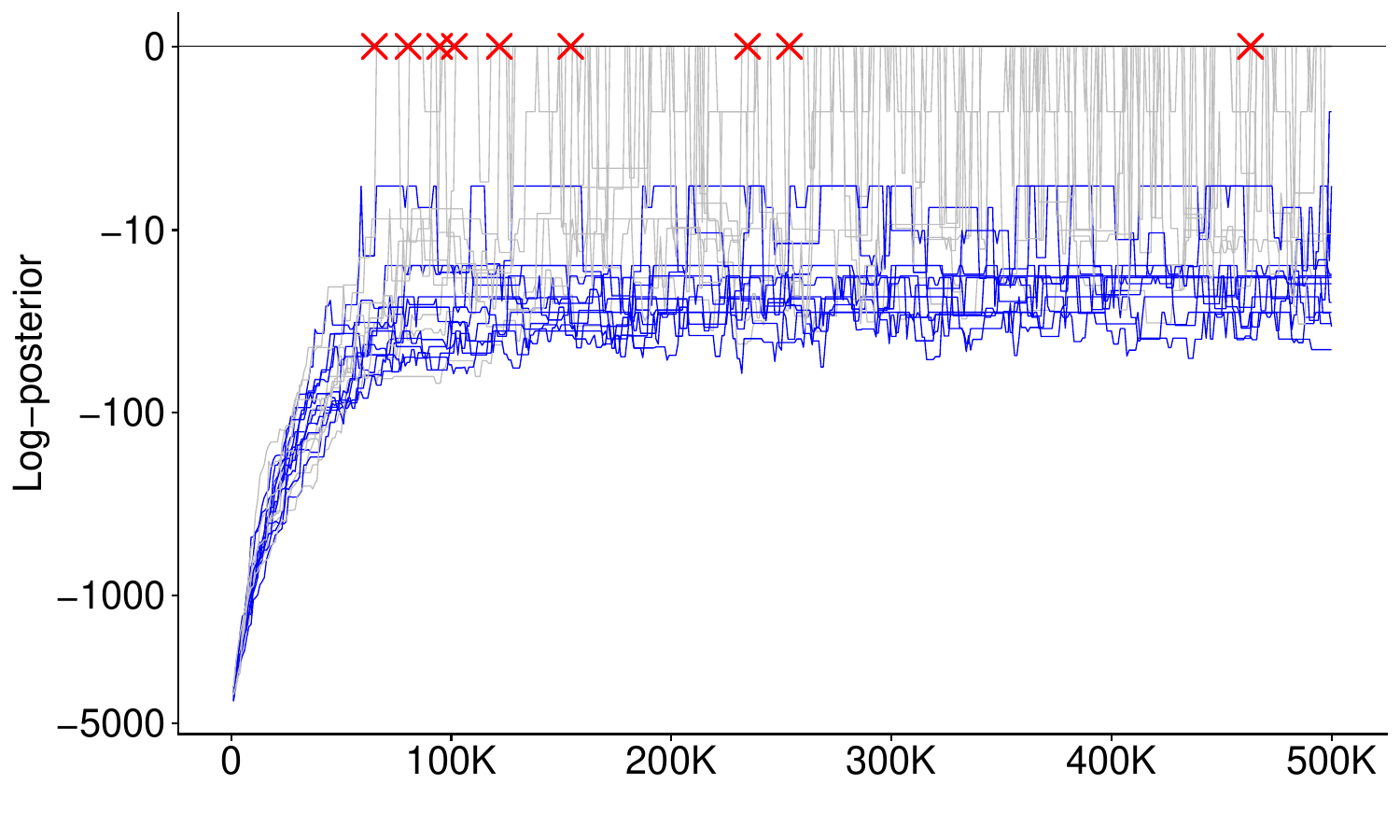}
    \caption{Trajectories of 20 independent \RWGES{} runs for a simulated data set with $p = 100$. Left: $n = 800$; right: $n = 200$. The posterior probabilities are un-normalized and the log-posterior of $\EG{G^*}$ is set to zero. Red crosses mark the times that \RWGES{} first collects $\EG{G^*}$. Runs that never sample $\EG{G^*}$ are shown in blue.}
    \label{fig:rapid_mixing}
\end{figure}

\subsection{Performance in a high-dimensional scenario}\label{sec:second.sim.high.dim} 
The complexity of the structure learning problem largely depends on $p$ and the sparsity level of the true DAG $G^*$. 
We can roughly measure the sparsity using the maximum degree of $G^*$, denoted by $\mathrm{deg}(G^*)$. 
 Assumption~\ref{A:np} and Remark~\ref{rmk:order.d} suggest that 
we consider $\mathrm{deg}(G^*)  \log p = O(n)$ and  $ \mathrm{deg}(G^*) = O(\sqrt{\log p})$.  
In the second simulation study, we examine these asymptotic orders by using $7$ simulation settings where $n$ grows linearly and $p$ grows exponentially. Given $p$, we generate $G^*$ by first sampling a random ordering and then including each edge with probability $D / (p - 1)$, where the parameter $D$ gives the expected number of neighbors of each node. We let $D$ grow at rate $\sqrt{n}$ (so we actually let $\mathrm{deg}(G^*)  \log p$ grow slightly faster than $n$). 
We generate $\X$ using the normal SEM associated with $G^*$ and choose the hyperpamareters in the same way as in Section~\ref{sec:first.sim.rapid}. 
The number of \RWGES{} iterations is set to grow polynomially with $p$ but slightly slower than $p^2$. We always initiate the sampler at the null model and  discard the first $80\%$ iterations as burn-in.  
For each setting, we generate 20 replicates ($G^*$ and $\X$ are re-sampled each time), and the results are shown in Table~\ref{table:high_dim} (see Supplement H.2 for the definition of true/false positive rates). 
Observe that the true positive rate (for both skeleton and CPDAG estimation) becomes stable as $p$ grows, while the false positive rate even decreases. 
Though for most real-world problems, the strong beta-min condition is unlikely to be satisfied and $\EG{G^*}$ may not be correctly identified, this study shows that the theoretical insights on the MCMC complexity is useful. In particular, the performance of \RWGES{} seems stable under the asymptotic regime $ \mathrm{deg}(G^*) \log p = O(n)$.  

\begin{table}[!htp]
    \centering
    \begin{tabular}{|cccc|ccc|}
    \hline 
    $p$ & $n$ & $D$ & $N_{\rm{mcmc}}/1000$ & TPR (skeleton) & TPR & FPR \\
    \hline 
7 & 60 & 1.549 & 3  & 0.854 (0.03) & 0.721 (0.06) & 0.047  (0.02)\\
14 & 90 & 1.897 & $10$ & 0.89 (0.02) & 0.668 (0.06) & 0.03 (0.006)\\
28 & 120 & 2.191 & $30$ & 0.91  (0.01) & 0.73 (0.03) & 0.017 (0.003)\\
56 & 150 & 2.449 & $100$ & 0.871 (0.01) & 0.629 (0.03) & 0.015 (0.002)\\
112 & 180 & 2.683 & $300$ & 0.866 (0.01) & 0.634 (0.02) & 0.0091  ($0.0005$)\\
224 & 210 & 2.898 & $1000$& 0.86 (0.008) & 0.634 (0.01) & 0.0049  ($0.0002$)\\
448 & 240 & 3.098 & $3000$ & 0.869 (0.004) & 0.648 (0.008) & 0.0027 ($0.00007$)\\    
\hline 
    \end{tabular}
    \caption{Performance of \RWGES{} in $7$ settings.   For the $k$-th setting, $p =7 \cdot 2^{k - 1}, n = 30(k+1), D = 0.2 \sqrt{n}$ and  the number of \RWGES{} iterations $N_{\rm{mcmc}} \approx 300\cdot (4p / 7)^{1.66}$.  
    TPR (skeleton): true positive rate with edge directions ignored; TPR: true positive rate (edge directions determined by the CPDAG); FPR: false positive rate. Results are averaged over $20$ replicates, and the number in parentheses is the standard error. 
  }
    \label{table:high_dim}
\end{table}

\subsection{On the choice of $c_2$} \label{sec:third.sim.c2}
The third simulation study aims to investigate the optimal choice of $c_2$, which is the most important prior hyperparameter of our model since it determines the order of the penalty on the graph size. We fix $p = 20$ and $n = 100$ and generate 50 true DAGs and data sets using the method described in Section~\ref{sec:second.sim.high.dim} with $D = 4$ (recall this gives the expected degree of a single node). 
When implementing \RWGES{}, we impose the maximum degree constraint, denoted by $d$ (i.e., the sampler only searches equivalence classes with maximum degree bounded by $d$);  see Supplement H.1 for details. 
 \RWGES{} is run for $40,000$ iterations for each simulated data set.
We first fix $c_2 = 1.3$ and try $d = 4, 5, \dots, 9$. The results are shown in the left column of Table~\ref{table:c2}. True positive rates increase with $d$, since some nodes in the true DAG may have large degrees and their incoming edges cannot all be detected if $d$ is small. However, the false positive rate also increases because the search space quickly grows with $d$. 
Next, we repeat the experiment by setting $c_2 = 1.1 + 0.1 d$, which, according to our tests, appears to yield close-to-optimal performance in this simulation setting. As can be seen from the right column of Table~\ref{table:c2}, the false positive rate remains roughly a constant and the true positive rates are comparable or even better than those for $c_2 = 1.3$. 
Recall that to prove posterior consistency, we assume $c_2$ is greater than $4(\alpha + 1)  \din$ plus some constant  in Assumption~\ref{A:prior}. This simulation study shows that, though the coefficient in Assumption~\ref{A:prior} is quite pessimistic, the linear growth rate (w.r.t. the maximum degree constraint) is a useful rule of thumb for tuning $c_2$ in practice.  

\begin{table}[!htp]
    \centering
    \begin{tabular}{|c|ccc|ccc|}
    \hline 
     & \multicolumn{3}{|c|}{$c_2 = 1.3$} & \multicolumn{3}{|c|}{$c_2 = 1.1 + 0.1 d$} \\
    \hline 
    $d$  & TPR (skeleton) & TPR & FPR &   TPR (skeleton) & TPR & FPR \\
    \hline 
 4  & 0.542 (0.02) & 0.307 (0.02) & 0.0824 (0.004) & 0.539 (0.01) & 0.316 (0.02) & 0.0788 (0.004)\\
5  & 0.61 (0.01) & 0.339 (0.02) & 0.101 (0.004) & 0.601 (0.01) & 0.325 (0.02) & 0.0977 (0.004)\\
6  & 0.665 (0.01) & 0.383 (0.02) & 0.115 (0.005) & 0.657 (0.01) & 0.393 (0.02) & 0.101 (0.005)\\
7  & 0.706 (0.01) & 0.412 (0.02) & 0.123 (0.006) & 0.699 (0.01) & 0.419 (0.02) & 0.096 (0.005)\\
8  & 0.72 (0.01) & 0.413 (0.02) & 0.132 (0.006) & 0.694 (0.01) & 0.421 (0.02) & 0.0993 (0.006)\\
9  & 0.718 (0.01) & 0.401 (0.02) & 0.138 (0.007) & 0.695 (0.01) & 0.437 (0.02) & 0.0932 (0.005)\\   
  \hline 
  \end{tabular}
   \caption{Simulation study with $p = 20$, $n = 100$ and expected node degree $D = 4$. Results are averaged over 50 data sets. 
   }
    \label{table:c2}
\end{table}

\section{Discussion}\label{sec:disc}

\subsection{Mixing of structure MCMC and order MCMC methods}\label{sec:mix.dag.order}
In this work, we have only analyzed the mixing times of  MCMC algorithms defined on the space of equivalence classes, but the same strategy can be pursued to study samplers defined on the DAG space and order space. 
Observe that the canonical paths we constructed in Section~\ref{sec:path.cpdag} for the \RWGES{} sampler can also be thought of as paths on the DAG space. Given an equivalence class $\cE$, we   first pick arbitrarily some $G \in \cE$. If $G$ has ordering $\sigma$, we move from $G$ to the minimal I-map $G^*_\sigma$ by only add-delete-swap modifications of the DAG. 
To move from $G^*_\sigma$ to $G^*$, we have to change the ordering. For \RWGES{}, the neighborhood function defined in~\eqref{eq:def.nbg} allows us to ``switch'' from $G^*_\sigma$ to a Markov equivalent DAG $\tilde{G}$, 
which is still an I-map of $G^*$ but no longer minimal, and then we can remove edges from $\tilde{G}$ (the existence of such $\tilde{G}$ is guaranteed by Chickering algorithm). 
Repeating this procedure, we obtain a path from $G^*_\sigma$ to $G^*$.  

Consider the classical structure MCMC sampler, a random walk MH algorithms defined on $\dag$ that use single-edge addition, deletion and reversal to propose local moves~\citep{madigan1995bayesian, castelo2003inclusion}.  Since any two Markov equivalent DAGs $G, G'$ are connected by a sequence of covered edge reversals (see Supplement C.1),  structure MCMC is  able to traverse equivalence classes and move from $G^*_\sigma$ to any other Markov equivalent DAG. Therefore, the canonical paths of \RWGES{} are also paths of structure MCMC (introduce swap moves if a restricted space is considered). The same argument can be applied to order MCMC samplers, since a covered edge reversal can be seen as an adjacent transposition on $\bbS^p$~\citep{solus2017consistency}.    
Unfortunately, the size of an equivalence class can easily be very large, and it is unclear  whether structure MCMC can always quickly leave any equivalence class even if the maximum degree is bounded.  
In Supplement G.6, we construct an interesting example where $G_0$ is Markov equivalent to $G^* \cup \{2 \rightarrow 1\}$ but it is quite difficult for structure MCMC to remove the edge between nodes $1$ and $2$ from $G_0$. Indeed, we show that on average it takes structure MCMC $O(p^4)$ iterations to move from $G_0$ to $G^*$, while it only takes \RWGES{} $O(p^2)$ iterations to move from $\EG{G_0}$ to $\EG{G^*}$. 
Nevertheless, we conjecture that  structure MCMC is still rapidly mixing under the assumptions we used in Section~\ref{sec:high} (recall ``rapid mixing'' only requires the mixing time to be polynomial in $n$ and $p$), though the proof would probably require a skillful analysis of how the size of an equivalence class changes with single-edge modifications of its member DAGs.  
  
One caveat is that the target posterior distributions on DAG and order spaces are typically  different from our target $\pi_n$ defined in~\eqref{eq:post.cpdag}. For example, for DAG MCMC methods, it is convenient to use the prior $\pi_0^{\rm{dag}}(G) \propto (c_1 p^{c_2})^{-|G|}$ for $G \in \dag(\din, \dout)$, which yields the posterior $\pi_n^{\rm{dag}}(G) \propto e^{\score(G)} \ind_{\dag(\din, \dout)}(G)$.   
Comparing them with~\eqref{eq:prior.equivalence} and~\eqref{eq:post.cpdag}, we see that $\pi_0^{\rm{dag}}(\cE)   = \sum_{G \in \cE} \pi_0^{\rm{dag}}(G)   \propto  |\cE| \pi_0(\cE)$ and  $\pi_n^{\rm{dag}}(\cE)   \propto  |\cE| \pi_n(\cE)$.   Note that we do not use $\pi_0^{\rm{dag}}$ for equivalence class samplers~\citep{castelletti2018learning} since calculating the size of $\cE$ can be extremely time-consuming. On the order space, the situation is more subtle since one DAG can be compatible with multiple orderings~\citep{eaton2012bayesian, ellis2008learning}.  
If rapid mixing of structure MCMC can be established, we expect that the same argument can be used to show the strong selection consistency of $\pi_n^{\rm{dag}}$.

\subsection{Advantages and extensions of \RWGES{}}\label{sec:rwges.ges} 
One  advantage of \RWGES{} over GES  is that \RWGES{} considers a restricted search space and is equipped with the swap proposal. This is particularly important to theoretical analysis.  
Non-sparse models can easily overfit the data (e.g. if node $j$ has more than $n$ parents, then $X_j$ can be perfectly explained leading to an infinite score), which is why the sparsity constraint is necessary for proving high-dimensional consistency results.  
For GES, even if the maximum degree of $G^*$ is bounded, there is still a possibility  that GES visits non-sparse equivalence classes along its search path and then its behavior becomes completely unpredictable. In the proof of~\citet{nandy2018high} on the high-dimensional consistency of GES, the authors directly assumed that the output of the first stage is not too large; see Assumption (A5) therein.     

The main methodological difference between the two algorithms is that  GES is essentially an optimization algorithm, while \RWGES{} is used for sampling. The general theory on the relation between optimization and sampling suggests that each has its own unique advantages~\citep{talwar2019computational}. In particular, when the sample size is not large, the posterior tends to be multimodal and MCMC sampling (if it converges) can yield better estimates via model averaging~\citep{hoeting1999bayesian}. One can also use the output of GES as the initial state for \RWGES{}, which, to some extent, may achieve the benefits of both methods. 
In our theoretical analysis, we choose to focus on \RWGES{} just for its simplicity. One can generalize it in many ways to improve its performance in practice, for example, by using an informed proposal scheme or combining it with tempering techniques (i.e., running multiple \RWGES{} samplers at different temperatures). 
One simple modification that may significantly improve the sampler's performance is to first estimate a large conditional independence graph~\citep{meinshausen2006high, raskutti2008model} and then use it to tune the proposal probabilities. This can be seen as a randomized extension of the method of~\citet{nandy2018high}.  
The canonical paths we construct in Section~\ref{sec:main} can always be applied as long as the sampler proposes states from $\nb(\cdot)$ (or a superset of it). 
But one important takeaway from our theory is that using a neighborhood smaller than $\nb(\cdot)$ may lead to slow mixing even when the sample size is sufficiently large.   
A detailed investigation into more sophisticated local MCMC schemes using $\nb(\cdot)$ is left to future research.

%%%%%%%%%%%%%%%%%%%%%%%%%%%%%%%%%%%%%%%%%%%%%%
%% Support information (funding), if any,   %%
%% should be provided in the                %%
%% Acknowledgements section.                %%
%%%%%%%%%%%%%%%%%%%%%%%%%%%%%%%%%%%%%%%%%%%%%%
%\section*{Acknowledgements}
%This work was supported by the T3 grant of Texas A\&M University. 
%%%% thank T3. 

\section*{Acknowledgements}
 The authors would like to thank all anonymous reviewers  whose comments have helped improve the quality of the paper. 
%\end{acks}
 
%%%%%%%%%%%%%%%%%%%%%%%%%%%%%%%%%%%%%%%%%%%%%%
%% Supplementary Material, if any, should   %%
%% be provided in {supplement} environment  %%
%% with title and short description.        %%
%%%%%%%%%%%%%%%%%%%%%%%%%%%%%%%%%%%%%%%%%%%%%%
\section*{Supplementary material}
%\stitle{Supplementary material for ``Complexity analysis of Bayesian learning of high-dimensional DAG models and  equivalence classes.''}
 Part A: a notation table. 
Part B: more results for mixing times of finite Markov chains and proofs for Section~\ref{sec:mix}. 
Part C:  preliminaries for graphical models. 
Part D: proofs for Section~\ref{sec:main}. 
Part E: auxiliary results for high-dimensional empirical variable selection. 
Part F: proofs for Section~\ref{sec:high}.  
Part G: proofs and  examples for Sections~\ref{sec:mcmc} and~\ref{sec:disc}.  
Part H: further details about \RWGES{} implementation and simulation studies. 
Part I: discussion on the case where the strong beta-min or faithfulness condition fails.

\renewcommand{\thesection}{\Alph{section}}
\setcounter{section}{0}
\renewcommand{\thetheorem}{\thesection\arabic{theorem}}
\setcounter{theorem}{0}
\renewcommand{\thelemma}{\thesection\arabic{lemma}}
\setcounter{lemma}{0}
\renewcommand{\thecorollary}{\thesection\arabic{corollary}}
\setcounter{corollary}{0}
\renewcommand{\thedefinition}{\thesection\arabic{definition}}
\setcounter{definition}{0}
\renewcommand{\theexample}{\thesection\arabic{example}}
\setcounter{example}{0}
\renewcommand{\thefigure}{\thesection\arabic{figure}}
\setcounter{figure}{0}
\renewcommand{\thetable}{\thesection\arabic{table}}
\setcounter{table}{0}
\renewcommand{\theremark}{\thesection\arabic{remark}}
\setcounter{remark}{0}
\renewcommand{\thealg}{\thesection\arabic{alg}}
\setcounter{alg}{0}

\def\setallcounters{
\setcounter{theorem}{0}
\setcounter{lemma}{0}
\setcounter{corollary}{0}
\setcounter{definition}{0}
\setcounter{example}{0}
\setcounter{figure}{0}
\setcounter{table}{0}
\setcounter{remark}{0}
\setcounter{alg}{0}
}

\numberwithin{equation}{section}

\newpage

%\begin{center}
%\LARGE{Supplementary Material}
%\end{center}

\section{Notation used in the main text}\label{sec:notation}
\setallcounters
In the table below,  we list the notation that is used frequently in Sections~\ref{sec:main} to~\ref{sec:disc}. 
 
\begin{center}
\begin{tabular}{|l|l|}
\hline
\textbf{Notation}   & \textbf{Description} \\ \hline
$[p]$ & $\{1,2, \dots, p \} $ \\ 
$\bbS^p$ &  set of all permutations of $[p]$ \\  
$|S|$ &  cardinality of a set $S$  \\ 
$\NM_p(\mu, \Sigma)$ & $p$-variate normal distribution with covariance matrix $\Sigma$ \\
$\sX$ &   a random vector with components $\sX_1, \dots, \sX_p$ \\ 
$\X, X_j, \X_S$ &  data matrix, column vector, submatrix with columns index by $S$ \\
$|G|$ &  number of edges in the DAG $G$ \\ 
$\HD(S, S'), \HD(G, G')$ & Hamming distance between two sets or DAGs \\
$\Pa_j(G), \Ch_j(G)$ & set of parents/children of node $j$ in the DAG $G$  \\ 
$\EG{G}$ &  the equivalence class that contains the DAG $G$   \\
%$\EGG{\cE}$ & the CPDAG representing the equivalence class $\cE$ \\ 
$\CI(G), \CI(\cE)$ & set of CI relations encoded by a DAG $G$ or an equivalence class $\cE$ \\ 
$\dag, \dag^\sigma $ &  set of $p$-vertex DAGs, set of $p$-vertex DAGs with ordering $\sigma$   \\
$\dag(\din, \dout), \dag^\sigma(\din, \dout)$ & sets of DAGs that satisfy the in-degree and out-degree constraints \\
$\cpdag, \cpdag(\din, \dout)$ &   sets of  equivalence classes \\ 
$\cA_p^\sigma(j)$ & set of nodes that precede $\sX_j$ in the ordering $\sigma$ \\ 
$\model^\sigma(j,  \, \din)$ & set of possible values of $\Pa_j(G)$ for $G \in \dag^\sigma(\din, p)$;  see~\eqref{eq:def.model} \\
$g_j^\sigma(S), g_j^\sigma(G)$ & canonical transition functions on $\model^\sigma(j, \din)$ and $\dag^\sigma(\din, \dout)$ \\
$ \nbg(\cE)$ & add-delete-swap neighborhood of an equivalence class $\cE$; see~\eqref{eq:def.nbg} \\
$\adds, \dels, \swaps $ & addition/deletion/swap neighborhood of $G$ or $\cE$ \\ 
$\Sigma(B, \Omega)$ & $\Sigma$ with a modified Cholesky decomposition given by $(B, \Omega)$ \\ 
$\chol(G) $ & set of pairs $(B, \Omega)$ such that $\Sigma(B, \Omega)$ is Markovian w.r.t. $G$; see~\eqref{eq:space.chol} \\
$\chol(\sigma)$ & set of pairs $(B, \Omega)$ compatible with ordering $\sigma$; see~\eqref{eq:def.chol.sigma} \\  
$\pi_0, \post$ & prior and posterior distributions or  density functions$^\dagger$ \\  
$\din, \dout$ & maximum  in-degree/out-degree \\ 
$c_1, c_2, \kappa, \gamma$ & hyperparameters for $\pi_0(\cE)$ and $\pi_0(B, \Omega \mid G)$ \\
$\alpha$ & exponent for the fractional likelihood function \\ 
$\scorej, \score$ & posterior scores; see~\eqref{eq:post.modular} \\ 
$\Sigma^*, G^*, \cE^*$ & true covariance matrix, DAG model and equivalence class \\
$\bbP^*$ & probability measure corresponding to the true model   \\ 
$G^*_\sigma$ & minimal I-map of $G^*$ with ordering $\sigma$ \\
$(B^*_\sigma, \Omega^*_\sigma)$ & modified Cholesky decomposition of $\Sigma^*$ in $\chol(\sigma)$ \\ 
$S^*_{\sigma, j}$ & parent set of node $j$ in $G^*_\sigma$ \\ 
$d^*_\sigma, d^*$ & maximum degree of the minimal I-map(s); see~\eqref{eq:def.dstar}  \\
%$r^*$ & maximum size of $\EG{G^*_\sigma}$ over all $\sigma$; see~\eqref{eq:rstar} \\ 
%$\bK, \bK_\sigma, \Ks $ & proposal distributions of MH algorithms \\
%$ \bP, \bP_\sigma, \Ps$ &  transition matrices of MH algorithms \\
$\bK, \bP$ & proposal and transition matrices of MH algorithms \\
\hline 
\end{tabular}
\end{center}
\noindent $^\dagger$: when $\pi_0$ or $\post$ denotes a density function, its dominating measure depends on the context.

\newpage 
\section{Path methods and mixing times of Markov chains}\label{sec:mc} 
\setallcounters
 
\subsection{On the equivalence between mixing time and hitting time} \label{sec:supp.mix}

Loosely speaking, the mixing time $\Tmix(\bP)$ gives the worst estimate for how many iterations it takes for a Markov chain to ``enter  stationarity''.
Theorem~\ref{th:path} shows that if Condition~\ref{cond:unimodal} holds with $t_2 > t_1$, $\pi$ concentrates on a single state $\theta^*$. 
In this case, it turns out that entering stationarity essentially means to hit $\theta^*$. 
Formally, we can prove that $\Tmix$ is equivalent to the expected hitting time of $\theta^*$, up to constant factors, using the result of~\citet{peres2015mixing}; see Theorem~\ref{th:hit.mix} below. 
For an intuitive explanation, observe that if $\pi(\theta^*) \approx 1$, then $\bP^t(\theta, \theta^*)$ needs to be sufficiently large so that $\TV{\bP^t (\theta, \cdot) - \pi(\cdot )} $ is small, which suggests that hitting $\theta^*$ is necessary for the chain to ``enter stationarity.''  
On the other hand, the chain regenerates each time it hits $\theta^*$, and thus between two successive visits to $\theta^*$, the chain has completed an independent cycle. So the length of each cycle gives an estimate for the mixing time. 
 
Let $\Theta$ be finite and $\bP$ be the transition matrix of an irreducible Markov chain $(Y_t)_{t \in \bbN}$ such that $\bP$ is reversible with respect to $\pi$.  
Let $\bbQ_\theta$ denote the probability measure for $(Y_t)_{t \in \bbN}$ with initial value $Y_0 = \theta$, and let $\bbE_\theta$ be the corresponding expectation. 
For $t \in \bbN$, let $\bP^t(\theta, \cdot) = \bbQ_\theta(  Y_t \in \cdot  )$ denote the $t$-step transition matrix. 
For any set $A \subseteq \Theta$, define the hitting time of $A$ by $\hit(A) = \min \{t \in \bbN \colon  Y_t \in  A \} $.  
Let $\Tmix^{\rm{L}}$ be the mixing time of the lazy chain with transition matrix $(\bP + \bI) / 2$. We have the following results. 

\begin{theorem}[\citet{peres2015mixing}]\label{th:peres}
For some $a < 1/2$, define 
$$T_{\rm{H}}^a = \max \left\{ \bbE_\theta [  \hit (A)  ]  \colon  \theta \in \Theta, A \subseteq \Theta, \pi(A) \geq a \right\}.$$
Then, $\Tmix^{\rm{L}}$ and $T_{\rm{H}}^a$ are equivalent up to constant factors. 
\end{theorem} 
\begin{remark}
This result was first proved by~\citet{aldous1982some} for continuous-time Markov chains. 
 \citet{griffiths2014tight} showed that the equivalence between $\Tmix^{\rm{L}}$ and $T_{\rm{H}}^a$ also holds for   $a= 1/2$. 
``Up to constant factors'' means that there exist constants $c_a, C_a > 0$ (which do not depend on $\bP$) such that $c_a T_{\rm{H}}^a \leq  \Tmix^{\rm{L}}  \leq C_a T_{\rm{H}}^a$. 
\end{remark}

\begin{theorem}\label{th:hit.mix}
If  there exists some state $\theta^*$ such that $\pi(\theta^*) > 1/2$,  then $ \Tmix^{\rm{L}}$ is equivalent, up to constant factors, to $T^*  = \max_{\theta \in \Theta} \bbE_\theta[ \hit( \{\theta^*\} ) ]$. 
\end{theorem}
\begin{proof}
Choose any $a \in (1 - \pi(\theta^*), 1/2)$.  
For any $A$ with $\pi(A) \geq a$, we have $\theta^* \in A$ and thus $\hit (A) \leq \hit ( \{\theta^*\} )$. 
Hence, we have $T_{\rm{H}}^a = \max \left\{ \bbE_\theta [  \hit(A)  ]  \colon  \theta \in \Theta, A = \{\theta^*\} \right\}$. 
The result then follows from Theorem~\ref{th:peres}. 
\end{proof}

Theorem~\ref{th:peres} also suggests that rapid mixing is impossible if the chain can get stuck at some state with small stationary probability for exponentially many steps. This can be proved by an elementary calculation. 

\begin{theorem}\label{th:slow}
Consider an asymptotic setting where $\Theta, \bP, \pi$ are implicitly indexed by $n$. 
For each $n$, assume there exists $\theta_0 \in \Theta$ such that $\pi(\theta_0) \leq 1/2$ and $\bP(\theta_0, \theta_0) \geq 1 - e^{-c n}$, where $c > 0$ is a universal constant. 
Then $\Tmix(\bP)$ cannot be bounded from above by any polynomial in $n$. 
\end{theorem}

\begin{proof} 
Let $A_n$ = $\Theta \setminus \{\theta_0\}$.
By the property of total variation distance, $ \TV{ \bP^t(\theta_0, \cdot) - \pi(\cdot)} \geq | \bP^t(\theta_0, A_n) - \pi(A_n)|$. 
It then follows from Definition~\ref{def:mix} that 
\begin{align*}
\Tmix (\bP) \geq \;& \min\left\{t \in \bbN \colon  | \bP^t(\theta_0, A_n) - \pi(A_n)|  \leq  1/4 \right\}  \\
\geq  \;& \min\left\{t \in \bbN \colon   \bP^t(\theta_0, A_n) \geq  1/4 \right\}, 
\end{align*} 
since $\pi(A_n) \geq 1/2$. 
Observe that $\bP^t(\theta_0, \theta_0) \geq (1 - e^{- c n})^t \geq 1 - t e^{-c n}$ for any $t \geq 1$.
Hence, $\bP^t(\theta_0, A_n) \leq te^{-cn}$, which yields the result. 
\end{proof}

\subsection{Path methods for bounding mixing times}\label{supp:subsec.path.mix}
Let $\Theta$ be finite and $\cN$ be a symmetric neighborhood function. 
We set up some notation and definitions for describing edges and paths on the neighborhood graph $(\Theta, \cN)$. 
Let $$\EDGE(\cN)  = \{ (\theta, \eta) \in \Theta^2 \colon  \theta \in \cN(\eta) \}$$ 
denote the set of all directed edges in  $(\Theta, \cN)$; in particular, $(\theta, \eta)$ and $(\eta, \theta)$ are treated as different edges for any $\theta \neq \eta$. 
Below is the definition of ``paths'' on $(\Theta, \cN)$. Note that we allow a path to contain repeated vertices (but not repeated edges), which is often known as a ``trail'' in the graph theory. 
 
\begin{definition}\label{def:path}
We say a finite sequence $\gamma = (\theta_0, \theta_1,  \dots,  \theta_{k-1}, \theta_k)$ is an $\cN$-path (or simply path) from $\theta$ to $\theta'$ with length $k$ if (i) $\theta_0 = \theta$, $\theta_k = \theta'$, (ii) for each $i = 1, \dots, k$, the edge $(\theta_{i -1 }, \theta_i) \in \EDGE(\cN)$, and (iii) $\gamma$ has no repeated edges.  
We will also denote such a path by $\gamma = (e_1, \dots, e_k)$ where $e_i = (\theta_{i-1}, \theta_i)$, and we write   $e \in \gamma$  to mean that the path $\gamma$ traverses the  edge $e$. 
A path is assumed to contain at least one edge. 
\end{definition} 

Throughout Supplement~\ref{sec:mc},  the letter $e$ is reserved for denoting edges and $\gamma$ for denoting paths, while elements of $\Theta$ are typically denoted by $\theta, \eta, z, w$.\footnote{Unfortunately, we have to abuse some notation: $\gamma$ is later used to denote a hyperparameter in our Bayesian structure learning model. 
But it should be clear that all notation used in Supplement~\ref{sec:mc} is not related to structure learning or DAG models, as we are considering a general state space $\Theta$ here.}  
Given a path $\gamma = (\theta_0, \dots, \theta_k)$, its reversal is denoted by $\cev{\gamma} = (\theta_k, \theta_{k-1}, \dots, \theta_0)$. 
Given $\gamma_1 = (\theta_0, \dots, \theta_k)$ and $\gamma_2 = (\eta_0, \dots, \eta_l)$ such that $\theta_k = \eta_0$, the concatenation of the two paths is denoted by $\gamma_1 \gamma_2 = (\theta_0, \dots, \theta_k, \eta_1, \dots, \eta_l)$.  
Let $\Gamma(\cN)$ denote the set of all $\cN$-paths.  For any $\theta \neq \eta$, let $\Gamma_{\theta \eta}(\cN)$ denote the set of all paths in $\Gamma(\cN)$ that start at $\theta$ and end at $\eta$; sometimes we will also use the notation $\Gamma(\theta, \eta; \cN)$ as an alternative to $\Gamma_{\theta \eta}(\cN)$. 
We say $(\Theta, \cN)$ is connected if $\Gamma_{\theta \eta}(\cN)$ is non-empty for  any $\theta \neq \eta$.

We will prove the mixing time bounds in Theorems~\ref{th:path} and~\ref{th:path.better} for a larger class of Markov chains. All we need is the following assumption on $\bP$. 

\begin{assumption}\label{ass:general.P}
Let $\bP$ be a Markov chain on $\Theta$ such that (i) $\bP$ is reversible with respect to some distribution $\pi > 0$,  (ii) $\{\eta \neq \theta \colon \bP(\theta, \eta ) > 0 \} = \cN(\theta)$ for any $\theta \in \Theta$, and (iii) 
all eigenvalues of $\bP$ are non-negative. 
\end{assumption}

Note that if $(\Theta, \cN)$ is connected, then $\bP$ in Assumption~\ref{ass:general.P} is also irreducible. 
Our proofs of Theorems~\ref{th:path} and~\ref{th:path.better} rely on the following Poincar\'{e}-type inequality. 

\begin{theorem} \label{th:flow}
Let $(\Theta, \cN)$ be connected and $\bP$ be given by Assumption~\ref{ass:general.P}.  
For any $e = (\theta, \eta) \in \EDGE(\cN)$, define %$\rho(e) = \pi(\theta) \bP(\theta, \eta).$   
$\rho(e) = \pi(\theta) \bP(\theta, \eta) = \pi(\eta) \bP(\eta, \theta).$  
Let $\phi \colon \Gamma(\cN) \rightarrow [0, \infty)$ be   such that  
$$\sum_{\gamma \in \Gamma_{\theta \eta}(\cN)}\phi( \gamma ) = \pi(\theta) \pi(\eta), \quad \text{ for any }  \theta \neq \eta.$$ 
Then,  for any function $\ell \colon \EDGE(\cN) \rightarrow (0, \infty)$,  we have 
\begin{align*}
\GAP(\bP)^{-1} \leq  \max_{e \in \EDGE(\cN) }  \left\{  \frac{1}{ \rho(e) \ell(e)}  \sum_{\theta \neq \eta} \sum_{\gamma \in \Gamma_{ \theta \eta }(\cN) \colon e \in \gamma}      \phi(\gamma) |\gamma|_\ell  \right\},  
\end{align*}
where  $ |\gamma|_\ell = \sum_{e \in \gamma} \ell(  e)$ and  $\GAP(\bP)$ denotes the spectral gap of $\bP$, 
\end{theorem}
\begin{proof}
See~\citet[Theorem 3.2.9]{saloff1997lectures} and~\citet{kahale1997semidefinite}. 
\end{proof}

\begin{remark} 
The function $\ell$ can be seen a generalized  ``length'' function which assigns a weight to each $e \in \EDGE(\cN)$~\citep{kahale1997semidefinite}. If we let $\ell(e) = 1$ for each $e$, then $|\gamma|_\ell$ is just the length of $\gamma$ and Theorem~\ref{th:flow}  reduces to the ``multicommodity flow'' method of~\citet{sinclair1992improved} (the function $\phi$ is called a ``flow'').  
To bound the mixing time, we can apply~\citet[Proposition 1]{sinclair1992improved} to get  
\begin{equation}\label{eq:sinclair2}
 \Tmix(\bP)  \leq \frac{ -\log[\min_{\theta \in \Theta} \pi(\theta)] + \log 4}{\GAP(\bP)}, 
\end{equation} 
Note that the assumption that $\bP$ has non-negative spectrum is not needed for Theorem~\ref{th:flow} but is necessary for~\eqref{eq:sinclair2}. 
%where $\Tmix(\bP)$ is the mixing time given in Definition~\ref{def:mix}. 
%This will be used in our proofs of Theorems~\ref{th:path} and~\ref{th:path.better}. 
\end{remark}

Roughly speaking, to obtain good bounds on $\GAP(\bP)$ using Theorem~\ref{th:flow}, we want to construct a flow $\phi$ by identifying at least one ``high-probability'' path between any $\theta \neq \eta$, 
where ``high-probability'' means that $\rho(e)$ is not too small for every edge $e$ of the path. 
To prove Theorem~\ref{th:path}, we only need to identify one such path between any $\theta \neq \eta$, which can be naturally constructed by using the function $g$ in Condition~\ref{cond:unimodal}; see Sections~\ref{sec:supp.canonical.path} and~\ref{sec:mc.proof}. 
To prove Theorem~\ref{th:path.better}, we need a much finer construction of the flow $\phi$ which takes into account all ``high-probability'' moves at each state; see Sections~\ref{sec:supp.flow} and~\ref{sec:supp.proof.path.better}.

\subsection{A general method for constructing canonical paths}\label{sec:supp.canonical.path}
Before we prove Theorem~\ref{th:path},  we first develop some general results for  constructing ``canonical path ensembles'' using ``canonical transition functions.'' 

\begin{definition}\label{def:cse}
A canonical path ensemble on $(\Theta, \cN)$ is a set of $\cN$-paths, one (and only one) for each ordered pair of two distinct states in $\Theta$. 
\end{definition} 

\begin{definition}\label{def:g}
We say $g \colon \Theta \rightarrow \Theta$ is a canonical transition function on $(\Theta, \cN)$ with a unique fixed point $\theta^*$ if  (i) $g(\theta^*) = \theta^*$; (ii) for any $\theta \neq \theta^*$, $g(\theta) \in \cN(\theta)$ and there exists some finite $k$ such that $g^k(\theta) = \theta^*$. 
\end{definition}

\begin{lemma}\label{lm:path}
Suppose $(\Theta, \cN)$ is connected, and fix some $\theta^* \in \Theta$. There exists a canonical transition function $g$ on $(\Theta, \cN)$ with fixed point $\theta^*$.   
Further, $g$ induces a canonical path ensemble on $(\Theta, \cN)$ such that each canonical path is an $\cN_g$-path, where $\cN_g(\theta) =\{ \theta' \in \Theta \colon g(\theta') = \theta, \text{ or } g(\theta) = \theta' \}$. 
\end{lemma}

\begin{proof}
First, we show that such a function $g$ exists. Since $(\Theta, \cN)$ is connected, for any $\theta \neq \theta^*$, there exists a shortest  $\cN$-path from $\theta$ to $\theta^*$, which we denote by $(\theta_0 = \theta, \theta_1, \dots, \theta_k = \theta^*)$. Define $\tilde{g}(\theta)$ to be the state $\theta_1$ on this path.  
Clearly, $\tilde{g}$ is a canonical transition function. 

Next, we explicitly construct a canonical path ensemble $\cT$ using an arbitrary canonical transition function $g$. The path from $\theta$ to $\eta$ in $\cT$ will be denoted by $\gamma_{\cT}(\theta, \eta)$, which is unique.  
Let 
\begin{equation*}\label{def:eq.natural.numbers}
    \bbN = \{0, 1, \dots \}, \quad \bbN^+ = \{1, 2,   \dots\},  \quad 
     k(\theta) = \min\{ i \in \bbN \colon g^i(\theta) = \theta^* \} < \infty. 
\end{equation*}
For $\theta \neq \theta^*$, define 
$\gamma_\cT(\theta, \theta^*) =  (\theta, g(\theta), \dots, g^{k(\theta)}(\theta) = \theta^*)$ (note that it cannot contain any duplicate state since otherwise $k(\theta)$ does not exist). 
Since $\cN$ is symmetric, we have $\theta \in \cN(g (\theta))$ for each $\theta \neq \theta^*$, and thus we can define   $\gamma_\cT(\theta^*, \theta) = \cev{\gamma}_\cT(\theta, \eta)$. 
The construction of $\gamma_\cT( \theta, \eta)$ for $\eta \neq \theta^*$ is  divided into three cases. 
 \begin{enumerate}[label={Case} \arabic*.]
     \item $\eta= g^j(\theta)$ for some $j \in \bbN^+$. 
     \item  $\theta = g^i(\eta)$ for some $i \in \bbN^+$. 
     \item  Neither Case 1 nor Case 2 holds. 
 \end{enumerate}
For Case 1, since $\gamma_\cT(\theta, \theta^*) = (\theta, g(\theta), \dots,g^{j-1}(\theta), \eta, g^{j+1}(\theta), \dots, g^k(\theta) = \theta^*)$, we can simply define $\gamma_\cT(\theta, \eta) = (\theta, g(\theta), \dots, g^{j-1}(\theta), \eta)$,  which is a sub-path of $\gamma_\cT(\theta, \theta^*)$. 
Case 2 can be handled similarly.  
For  Case 3, we define  
\begin{equation*}
\gamma_\cT( \theta, \eta) = \gamma_\cT(\theta, \theta^*) \gamma_\cT(\theta^*, \eta). 
%(\theta, g(\theta), \dots,g^j(\theta) = \theta^* = g^k(\theta), \dots, g(\theta'), \theta'),
\end{equation*}
To prove $\gamma_\cT( \theta, \eta )$ has no duplicate edges, it suffices to show that paths $\gamma_\cT( \theta, \theta^*)$ and $\gamma_\cT( \theta^*, \eta)$ do not share any states except $\theta^*$. 
We prove it by contradiction.  Suppose $w \neq \theta^*$ exists in both paths.  Then $w = g^s(\theta) = g^t(\eta)$ for some $s, t \in \bbN^+$. Without loss of generality,  assume $s  > t$. But this implies  that $\eta = g^{s - t}(\theta)$, which yields the contradiction. 
\end{proof}

\subsection{Proof  of Theorem~\ref{th:path}}\label{sec:mc.proof}

\begin{proof}[Proof of Theorem~\ref{th:path}(i)]  
We say $\theta$ is a local maximum if $\pi(\theta) \geq \pi(\theta')$ for any $\theta' \in \cN(\theta)$. 
Part (i) follows upon observing that any $\theta \neq \theta^*$ cannot be a local maximum, and thus $\theta^*$ is the only local and also global maximum. 
\end{proof}

\begin{proof}[Proof of Theorem~\ref{th:path}(ii)]  
For each $\theta \in \Theta$ and $k \in \bbN$, let 
$$g^{-k}(\theta) = \{ \theta' \in \Theta \colon g^k(\theta') = \theta, \; g^{k-1}(\theta') \neq \theta \}.$$ 
Since $\cN$ is symmetric, we have $|g^{-1}(\theta)| \leq |\cN(\theta)|$. 
Further, note that $\Theta = \bigcup_{k \geq 0} g^{-k}(\theta^*)$, and for each $k \geq 1$,   we can write 
\begin{align*}
    g^{-k}(\theta^*) = \{ g^{-1}(\theta') \colon \theta' \in g^{-(k-1)}(\theta^*) \}. 
\end{align*}
A recursive calculation using Condition~\ref{cond:unimodal} then shows that  $\pi(\theta) / \pi(\theta^*) \leq p^{ -k t_2}$ for any $\theta \in g^{-k}(\theta^*)$, and $|g^{-k}(\theta^*)| \leq p^{k t_1}$.  
Hence, 
\begin{align*}
    \frac{ \sum_{\theta \in \Theta} \pi(\theta) }{ \pi(\theta^*)} \leq \sum_{k=0}^\infty \frac{ \pi( g^{-k}(\theta^*) ) }{ \pi( \theta^* )} 
    \leq \sum_{k=0}^\infty  p^{-k(t_2 - t_1)} = \frac{1}{1 - p^{-(t_2 - t_1)}}, 
\end{align*}
from which the result follows. 
\end{proof}

\begin{proof}[Proof of Theorem~\ref{th:path}(iii)] 
We prove the claim for any Markov chain $\bP$ that satisfies Assumption~\ref{ass:general.P}, which  includes $\bP^h_{\rm{lazy}}$ as a special case (note that $\bP^h(\theta, g(\theta)) = 2 \bP^h_{\rm{lazy}} (\theta, g(\theta))$ for $\theta \neq \theta^*$). 
The existence of $g$ implies that $(\Theta, \cN)$ is connected and $\bP$ is irreducible. 

Let  $\cT = \{ \gamma_\cT (\theta, \eta) \colon  \theta, \eta \in \Theta, \text{ and } \theta \neq \eta \}$ be the canonical path ensemble induced by $g$, as constructed in the proof of Lemma~\ref{lm:path}.  
It is clear from construction that  for any $\theta \neq \theta'$, we have $| \gamma_\cT(\theta, \theta') | \leq | \gamma_\cT(\theta, \theta^*) |  + | \gamma_\cT(\theta', \theta^*) | \leq 2\ell_{\rm{max}}$,  where $\ell_{\rm{max}}$ is as  defined in Theorem~\ref{th:path}. 
By~\citet[Corollary 6]{sinclair1992improved}, which is a special case of Theorem~\ref{th:flow}, we have  
\begin{equation}\label{eq:sinclair1}
  \frac{1}{ \GAP(\bP) } \leq    \max_{ e \in \EDGE(\cN) } \frac{2\ell_{\rm{max}}}{ \rho(e) } \sum_{(\theta, \eta) \colon  e \in \gamma_\cT(\theta,\eta)  } \pi(\theta) \pi(\eta). 
\end{equation}   

Consider an arbitrary $e = (z, w) \in \EDGE(\cN)$. 
Observe that for $\cT$ constructed in Lemma~\ref{lm:path},  
\begin{align*}
\left\{ (\theta, \eta) \colon  e \in \gamma_\cT(\theta,\eta)   \right\} \neq \emptyset  \text{ only if } w = g(z) \text{ or } z = g(w). 
\end{align*} 
So, to bound the maximum term in~\eqref{eq:sinclair1}, we can assume $e = (z, w)$ for  $w = g(z)$, which implies that $z \neq \theta^*$ (the  case $e = (w, z)$ can be analyzed similarly).   
Define 
$$\Lambda(z) =  \{\theta \in \Theta \colon  z = g^k(\theta), \,  k \in \bbN\}$$ 
as the ancestor set of $z$ w.r.t. the transition function $g$, where we define  $g^0(z) = z$.  
If $(z, w) \in \gamma_\cT(\theta,\eta )$ for some $\theta \neq \eta$, according to our construction of $\cT$, it is straightforward to verify that $\theta \in \Lambda(z)$. 
Therefore, $\{(\theta, \eta) \colon  (z, w) \in \gamma_\cT(\theta, \eta ) \} \subseteq  \Lambda(z) \times \Theta$.
It follows that 
\begin{equation}\label{eq:rhoT}
\begin{aligned}
\frac{1}{ \rho(e) } \sum_{(\theta, \eta) \colon  e \in \gamma_\cT(\theta,\eta)  } \pi(\theta) \pi(\eta) & \leq  \frac{1}{ \rho(e) }
\sum_{(\theta, \eta) \in \Lambda(z) \times \Theta} \pi(\theta) \pi(\eta)  \\
& =  \frac{1}{ \rho(e) } 
\sum_{ \theta \in \Lambda(z)  } \pi(\theta)  \sum_{\eta \in \Theta} \pi(\eta)  \\
& =   \frac{ \pi(\Lambda(z)) }{\pi(z)  \bP(z, w)}. 
\end{aligned}
\end{equation}
Analogously to the proof of part (ii), we can write $\Lambda(z) = \bigcup_{k \in \bbN} g^{-k}(z)$, and by Condition~\ref{cond:unimodal}, we have  $|g^{-k}(z)| \leq p^{kt_1}$ and $\pi(\theta) / \pi(z) \leq p^{-k t_2}$ for any $\theta \in g^{-k}(z)$. 
%$\pi( g^{-k}(z) ) / \pi(z) \leq p^{-k (t_2 - t_1)}$. 
Since $t_2 > t_1$,  
\begin{equation}\label{eq:ancestor}
\frac{\pi(\Lambda(z))}{\pi(z)} = \sum_{k \in \bbN} \frac{\pi(g^{-k}(z))}{\pi(z)} \leq \sum_{k \in \bbN} p^{-k(t_2 -t_1)} = \frac{ 1 }{1 - p^{-(t_2 -t_1)}}. 
\end{equation}
Combining~\eqref{eq:sinclair1},~\eqref{eq:rhoT} and~\eqref{eq:ancestor}, we get
\begin{align*}
      \frac{1}{ \GAP(\bP) } \leq     \frac{2\ell_{\rm{max}}}{ \{ 1 - p^{-(t_2 -t_1)} \} \min_{\theta \neq \theta^*} \bP(\theta, g(\theta)) }. 
\end{align*}
The proof is then completed by invoking~\eqref{eq:sinclair2}.
\end{proof}

\begin{proof}[Proof of Theorem~\ref{th:path}(iv)] 
By definition of $\bP^h$, if $h \equiv 1$, we have $\bK^h(\theta, g(\theta)) = 1 / |\cN(\theta)|$. 
Using Condition~\ref{cond:unimodal}, we find that 
\begin{align*}
    \bP^h(\theta, g(\theta) ) = \;&  \frac{1}{|\cN(\theta)|} \min \left\{ 1, \;  
    \frac{\pi(g(\theta))}{ \pi(\theta)} \frac{ 1 / |\cN( g(\theta) )| }{ 1 / |\cN(\theta)| } \right\} \\
    \geq   \;&  \frac{1}{|\cN(\theta)|} \min \left\{ 1, \;  p^{t_2 - t_1}  \right\}   
  =  \frac{1}{|\cN(\theta)|} \geq p^{-t_1}, 
\end{align*}
%By Condition~\ref{cond:unimodal}, $\pi(g(\theta)) / \pi(\theta) \geq p^{t_2}$, and $|\cN(\theta)|/ |\cN(g(\theta))| \geq p^{ -% t_1}$,  
which yields the claim.  
\end{proof}

\subsection{A general method for constructing flows}\label{sec:supp.flow}
We propose a general method for constructing a flow $\phi$, which is motivated by Condition~\ref{cond:unimodal}. 
We say  $f  \colon \Gamma(\cN) \rightarrow [0, 1]$ is a unit flow if 
$\sum_{ \gamma \in \Gamma_{\theta \eta}(\cN) } f(\gamma) = 1$ for any $\theta \neq \eta$.  
To construct a flow $\phi$ as described in Theorem~\ref{th:flow}, it suffices to find a unit flow $f$ first and then let $\phi(\gamma) = f(\gamma) \pi(\theta) \pi(\eta)$ for each $\gamma \in \Gamma_{\theta \eta }(\cN)$.  
Let $R > 1$ be a fixed constant and define 
%\begin{equation}%\label{eq:def.NR}
\begin{align*}
     \cN_R(\theta) =\;& \left \{ \theta' \in \cN(\theta) \colon \frac{\pi(\theta')}{\pi(\theta)} \geq R \right\}, \quad \text{ for } \theta \neq \theta^*,  \\ 
    \EDGE(\cN_R) =\;& \left\{ (\theta, \eta) \colon \theta \neq \theta^*, \eta \in \cN_R(\theta) \right\}.  
\end{align*}
%\end{equation}
Assume that there exists some $\theta^* \in \Theta$  such that $\cN_R(\theta)$ is not empty for any $\theta \neq \theta^*$ (this would be true if Condition~\ref{cond:unimodal} holds); note $\cN_R$ is ``asymmetric''. 
Now we construct a unit flow $f_R$ by only using edges in $\EDGE(\cN_R)$ and their reversals. 
%$e$ such that  $e \in \EDGE(\cN_R)$ or $\cev{e} \in \EDGE(\cN_R)$. %in the set $\{ (\theta, \eta)\colon \eta \in \cN_R(\theta) \text{ or } \theta \in \cN_R(\eta) \}.$  

Let $\bP$ be given by Assumption~\ref{ass:general.P}, and 
define another transition matrix $\bP_R$ by letting $\bP_R(\theta^*, \theta^*) = 1$ and 
\begin{equation*}\label{eq:def.Pc}
\begin{aligned}
\bP_R(\theta, \theta')  = 
 \frac{ \bP(\theta, \theta') }{\bP(\theta, \cN_R(\theta) ) }  \ind_{ \cN_R(\theta)}(\theta')   , \quad    \text{  for any } \theta \neq \theta ^*,   
\end{aligned}
\end{equation*}
Clearly, $\bP_R$ has one (and only one) absorbing state, $\theta^*$. 
Thus, for any $\theta \neq \theta^*$, we can construct a unit flow from $\theta$ to $\theta^*$ by running the Markov chain $\bP_R$ with initial state $\theta$; the weight of a path from $\theta$ to $\theta^*$ is simply given by the probability of the chain $\bP_R$ moving along the path. 
Reversing the paths from $\theta$ to $\theta^*$ and keeping the weights unchanged, we obtain a unit flow from $\theta^*$ to $\theta$. 
To construct a unit flow from $\theta$ to $\eta$ for some distinct $\theta, \eta \in \Theta \setminus \{\theta^*\}$, 
we treat $\theta^*$ as the ``hub'' and concatenate each path of $\bP_R$ from $\theta$ to $\theta^*$ and each path (after reversal) of $\bP_R$ from $\eta$ to $\theta^*$; that is, we only consider paths in $\Gamma_{\theta \eta}(\cN)$ that visit $\theta^*$ once. 
A formal definition of $f_R$ is given below, followed by a toy example illustrating the construction of $f_R$. 
We prove   a useful inequality for the ``load'' of an edge in $f_R$ in Lemma~\ref{lm:flow}. 

\begin{definition}\label{def:flow.fR}
For $k \geq 1$ and $\gamma = (\theta_0,  \theta_1, \dots, \theta_k) \in \Gamma(\cN)$, define $f_R(\gamma)$ as follows.  
\begin{enumerate}[label=(\roman*)]
    \item If $\theta^*$ does not occur in $\gamma$ or $\theta^*$ occurs at least twice in $\gamma$, let $f_R(\gamma) = 0$. 
    \item If $\theta_k = \theta^*$, let $f_R (\gamma) = \prod_{i=1}^{k} \bP_R (\theta_{i - 1}, \theta_i)$. 
    \item If $\theta_0 = \theta^*$, let $f_R(\gamma) = f_R(\cev{\gamma})$. 
    %where $f_R(\cev{\gamma})$ is defined in (i). 
    \item If $\theta_j = \theta^*$ for some $1 \leq j \leq k - 1$, let $\gamma_1 = (\theta_0, \dots, \theta_{j - 1}, \theta^*)$ and $\gamma_2 = (\theta^*, \theta_{j + 1}, \dots, \theta_k)$ and define $f_R(\gamma) = f_R(\gamma_1)   f_R(\gamma_2)$. % where $f_R(\gamma_1), f_R(\gamma_2)$ are defined in (i) and (ii), respectively. 
    %\item If $\theta_i \neq \theta^*$ for $i = 0, \dots, k$, let $f_R(\gamma) = 0$. 
\end{enumerate}
\end{definition}

\begin{example}
Let $\Theta = \{\theta_1, \dots, \theta_9, \theta^* \}$ and $\bP_R$ be given by Figure~\ref{fig:flow}. Then, from $\theta_1$ to $\theta^*$ there are three paths, and the weights are given by 
\begin{align*}
    f_R(  (\theta_1, \theta_2, \theta_4, \theta^*) ) =   \frac{1}{3}, \quad 
    f_R(  (\theta_1, \theta_2, \theta_5, \theta^*) ) =   \frac{1}{6}, \quad 
    f_R(  (\theta_1, \theta_3,  \theta^*) ) = \frac{1}{2}. 
\end{align*}
From $\theta_3$ to $\theta^*$  there is only one path $(\theta_3,   \theta^*)$. Hence, from $\theta_1$ to $\theta_3$ there are $3 \times 1 = 3$ paths; for example, the weight of the path $\gamma = (\theta_1, \theta_3, \theta^*, \theta_3 )$ is given by 
$$ f_R( \gamma) 
= f_R( (\theta_1, \theta_3,  \theta^*)) f_R( ( \theta^*,   \theta_3) )  = 1/2.  $$
Note that $\gamma$ does not contain repeated edges, though the state $\theta_3$ occur twice.  It is also clear  that $\bP_R^3(\theta, \theta^*) = 1$ for any $\theta \in \Theta$; that is, from any $\theta$ we can move to $\theta^*$ in at most 3 steps. 
\end{example}
 
\begin{figure}
    \centering
    \includegraphics[width=0.6\linewidth]{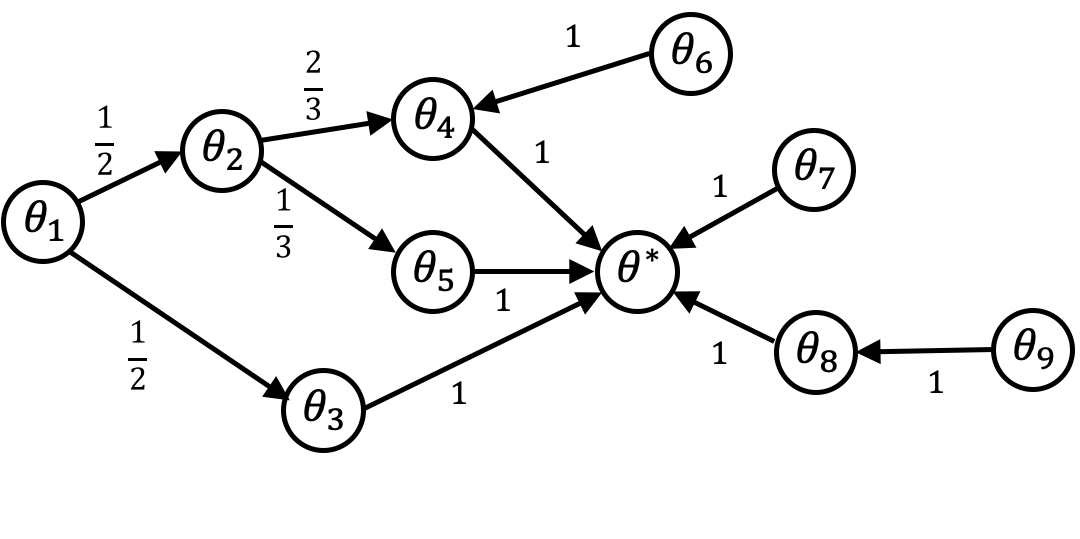}
    \caption{An example of the Markov chain $\bP_R$.   
    If $\bP_R(\theta, \theta') > 0$, then we draw an arrow from $\theta$ to $\theta'$ and denote the value of $\bP_R(\theta, \theta') $ near the arrow. 
    }
    \label{fig:flow}
\end{figure}

\begin{lemma}\label{lm:flow}
Let $\theta^* \in \Theta$ and assume  $\cN_R(\theta)$ is  non-empty for every $\theta \neq \theta^*$. The function $f_R \colon \Gamma(\cN) \rightarrow [0, 1]$ constructed in Definition~\ref{def:flow.fR} is a unit flow. Further, for any  $e = (z, w)$ such that $e \in \EDGE(\cN_R)$ or $\cev{e} \in \EDGE(\cN_R)$,  we have 
\begin{align*}
%\max\left\{ \sum_{\gamma \in \Gamma_{\theta \eta}(\cN) \colon \cev{e} \in \gamma } f_R(\gamma), \;   \sum_{\gamma \in \Gamma_{\theta \eta}(\cN) \colon e \in \gamma } f_R(\gamma) \right\} 
\sum_{\gamma \in \Gamma_{\theta \eta}(\cN) \colon e \in \gamma } f_R(\gamma) \leq \bP_R(z, w), \quad \forall \, \theta, \eta \in \Theta, \, \theta \neq \eta. 
\end{align*}
\end{lemma}
\begin{proof}
Since the Markov chain $\bP_R$  can only stay at the state $\theta^*$, we can use 
\begin{equation}\label{eq:def.Gamma.NR}
    \Gamma(\cN_R) = \{ \gamma = (\theta_0, \theta_1, \dots, \theta_k) \colon k \geq 1, \; \theta_i \in \cN_R(\theta_{i-1}) \text{ for } i = 1, \dots, k  \} 
\end{equation}
to denote  all possible paths of $\bP_R$. Let $\Gamma(\theta, \eta; \cN_R)$ denote the paths in $\Gamma(\cN_R)$ that start at $\theta$ and end at $\eta$. 
Fix some $e = (z, w) \in \EDGE(\cN_R)$.  
We prove that 
\begin{align*}
    \sum_{ \gamma \in \Gamma_{\theta \eta}(\cN) } f_R(\gamma) = 1, \quad 
    \sum_{\gamma \in \Gamma_{\theta \eta}(\cN) \colon e \in \gamma } f_R(\gamma) \leq \bP_R(z, w)
\end{align*}
by considering three  cases for the choice of $(\theta, \eta)$.  For $\cev{e} = (w, z)$, the second inequality can be proved by a symmetry argument. 

First, let $\theta \neq \theta^*$ and $\eta = \theta^*$.  
It suffices to consider paths in $\Gamma(\theta, \theta^*; \cN_R)$ since any other $\gamma$ has $f_R(\gamma) = 0$.  
Fix an arbitrary integer $j \geq |\Theta|$. 
Observe that $\bP_R^j(\theta, \theta^*) = 1$. 
Hence, by our definition of $f_R$, $\sum_{ \gamma \in \Gamma(\theta, \theta^*; \cN_R) } f_R (\gamma)$ 
is just the sum of the probabilities of all $j$-step paths of the Markov chain $\bP_R$ with initial state $\theta$, which must be one. 
Similarly, $\sum_{\gamma \in \Gamma(\theta, \theta^*; \cN) \colon e \in \gamma } f_R(\gamma)$ is the sum of the probabilities of all $j$-step paths that traverse the edge $e$. Hence, it is less than or equal to $\bP_R(z, w)$, where the equality is attained if all $j$-step paths visit the state $z$. 

Second, let $\theta = \theta^*$ and $\eta \neq \theta^*$. By our definition of $f_R$, we have 
\begin{align*}
    \sum_{ \gamma \in \Gamma( \theta^*, \eta; \cN) } f_R (\gamma) 
    = \sum_{ \gamma \in \Gamma( \theta^*, \eta; \cN) } f_R (\cev{\gamma})
    = \sum_{ \gamma \in \Gamma( \eta, \theta^*; \cN) } f_R (\gamma) = 1. 
\end{align*}
The edge $e$ is not traversed by any $\gamma$ if $\cev{\gamma} \in \Gamma( \eta, \theta^*; \cN)$. 

Third, consider distinct $\theta, \eta \in \Theta \setminus \{\theta^*\}$. 
For any $\gamma \in \Gamma_{\theta \eta}(\cN)$, note that $f_R(\gamma) = 0$ unless we can write $\gamma = \gamma_1 \gamma_2$ for some $\gamma_1 \in \Gamma(\theta, \theta^*; \cN_R)$ and $\gamma_2 \in \Gamma( \theta^*, \eta; \cN_R)$. 
Therefore, 
\begin{align*}
       \sum_{ \gamma \in \Gamma( \theta, \eta; \cN) } f_R (\gamma)   =  \;&       \sum_{ \gamma_1 \gamma_2 \colon \gamma_1 \in \Gamma(\theta, \theta^*; \cN_R), \gamma_2 \in \Gamma( \theta^*, \eta; \cN_R) } f_R (\gamma)  \\
       = \;&  \sum_{  \gamma_1 \in \Gamma(\theta, \theta^*; \cN_R)  }  \sum_{    \gamma_2 \in \Gamma( \theta^*, \eta; \cN_R) } f_R (\gamma_1) f_R(\gamma_2) = 1. 
\end{align*}
The edge $e$ can only be traversed by the $\gamma_1$ segment of $\gamma$, and thus 
$$ \sum_{\gamma \in \Gamma_{\theta \eta}(\cN) \colon e \in \gamma } f_R(\gamma_1) \leq \sum_{\gamma_1 \in  \Gamma(\theta, \theta^*; \cN_R) \colon e \in \gamma } f_R(\gamma) \leq 
\bP_R(z, w), $$ 
by what we have proved for the first case.  
\end{proof}

\subsection{Proof of Theorem~\ref{th:path.better}} \label{sec:supp.proof.path.better}

\begin{proof}%[Proof of Theorem~\ref{th:path.better}]
We prove the claim for any Markov chain $\bP$ that satisfies Assumption~\ref{ass:general.P}.  
Let $R = p^{t_2}$. 
By Condition~\ref{cond:unimodal}, $\cN_R(\theta)$ is  non-empty for every $\theta \neq \theta^*$. So, we can use the unit flow $f_R$ constructed in Definition~\ref{def:flow.fR} to define a flow  $\phi$ by letting $\phi(\gamma) = f_R(\gamma) \pi(\theta) \pi(\eta)$ for $\gamma \in \Gamma_{\theta \eta}(\cN)$. 
To apply Theorem~\ref{th:flow},   it only remains to specify the  length function. 
Observe that $\sum_{\gamma \in \Gamma(\cN) \colon e \in \gamma} f_R (\gamma) > 0$ if and only if  $e \in \EDGE(\cN_R)$ or $\cev{e} \in \EDGE(\cN_R)$. 
%Hence, it suffices to define the length of edge $e = (z, w)$ 
Define
\begin{align*}
\ell( e ) = \ell (\cev{e} ) = \pi(z)^{-q},  \quad \text{ if } e = (z, w), w \in \cN_R(z), 
\end{align*}
for some $q \in (0, 1)$. 
If $\gamma = (\theta_0 = \theta, \theta_1, \dots, \theta_k = \theta^*) \in \Gamma(\theta, \theta^*; \cN_R)$, then 
\begin{align*}
|\gamma|_\ell = \;&  \sum_{i = 0}^{k - 1} \pi(\theta_i)^{-q}  
\leq    \sum_{i = 0}^{k - 1} \pi(\theta)^{-q}  p^{- i t_2   q}   \leq  \frac{\pi(\theta)^{-q}}{  1 - p^{-t_2 q}}. 
\end{align*} 
For any  $\gamma \in \Gamma_{\theta \eta}(\cN)$ such that $f_R(\gamma) > 0$, we can write $\gamma = \gamma_1 \gamma_2$ where $\gamma_1 \in \Gamma(\theta, \theta^*; \cN_R)$ and $\cev{\gamma}_2 \in \Gamma(\eta, \theta^*; \cN_R)$, which yields
\begin{align*}
|\gamma|_\ell  = |\gamma_1|_\ell + |\gamma_2|_\ell \leq  \frac{\pi(\theta)^{-q} + \pi(\eta)^{-q}}{  1 - p^{-t_2 q}}. 
\end{align*}

Fix some  $e = (z, w) \in \EDGE(\cN_R)$.  (By symmetry, the case $\cev{e} \in \EDGE(\cN_R)$ can be analyzed by the same argument.)
By Lemma~\ref{lm:flow} and the above bound on $|\gamma|_\ell$, we obtain that 
\begin{align*}
\sum_{\gamma \in \Gamma_{\theta \eta }(\cN) \colon e \in \gamma}  \phi(\gamma) |\gamma|_\ell
 = \;&   \pi( \theta ) \pi(\eta ) \sum_{\gamma \in \Gamma_{\theta \eta }(\cN) \colon e \in \gamma}   f_R(\gamma) |\gamma|_\ell    \\
 \leq  \;&    \frac{\pi(\theta)^{1 - q} \pi(\eta) + \pi(\theta) \pi(\eta)^{1 - q}}{  1 - p^{- t_2 q}} \bP_R(z, w). 
\end{align*}

Recall $\Gamma(\theta, \eta; \cN_R)$ defined in~\eqref{eq:def.Gamma.NR}. Let 
\begin{align*}
    \Lambda(z) = \left\{ \theta \in \Theta \colon \Gamma(\theta, z; \cN_R) \neq \emptyset \right\} \cup \{z \}. 
\end{align*} 
%By definition, if $\eta \in \Lambda(\theta)$ and $\eta \neq \theta$, then there is a path $\gamma = (e_1, \dots, e_k) \in \Gamma_{\eta \theta}(\cN)$ such that  $e_i \in \EDGE(\cN_R)$ for each $i$. 
Note that if a path $\gamma = (\theta_0 = \theta, \theta_1, \dots, \theta_k) $ with $f_R(\gamma) \geq 0$ traverses the edge $e$, we must have $\theta \in \Lambda(z)$. 
Hence,  
\begin{equation}\label{eq:proof.path.better.eq1}
\begin{aligned}
\sum_{\theta \neq \eta }\sum_{\gamma \in \Gamma_{\theta \eta }(\cN) \colon e \in \gamma}  \phi(\gamma) |\gamma|_\ell
\leq  \;& \sum_{\theta \in \Lambda(z)} \sum_{\eta  \in \Theta } \sum_{\gamma \in \Gamma_{\theta \eta }(\cN) \colon e \in \gamma}   \phi(\gamma) |\gamma|_\ell \\ 
 \leq  \;&    \frac{\bP_R (z, w) }{  1 - p^{- t_2 q}}  \sum_{\theta \in \Lambda(z)} \sum_{\eta \in \Theta }  \left[ \pi(\theta)^{1 - q} \pi(\eta) + \pi(\theta) \pi(\eta)^{1 - q}\right]. 
\end{aligned}
\end{equation}
%The rest is similar to the proof of Theorem~\ref{th:path}. 
Since $\cN$ is symmetric, for any $\theta$, 
\begin{align*}
    |\{ \eta \in \Theta \colon  \theta \in \cN_R(\eta) \} |  \leq 
    | \cN(\theta) | \leq p^{t_1}. 
\end{align*}
Applying this argument recursively, we find that for any $z \neq \theta^*$, 
\begin{align*}
 \sum_{\theta \in \Lambda(z) } \pi(\theta)^{c} \leq \sum_{k=0}^\infty \pi(z)^{c} p^{( t_1 - t_2 c) k} \leq \frac{\pi(z)^{c} }{1 - p^{t_1 - t_2 c}}, 
\end{align*}
for any $c > 0$  such that the series converges. Similarly, 
\begin{equation*} 
\sum_{\eta \in \Theta} \pi(\eta)^{c} = \sum_{\eta \in \Lambda(\theta^*)} \pi(\eta)^{c} \leq  \frac{ \pi(\theta^*)^c }{1 - p^{t_1 - t_2 c}} \leq    \frac{ 1 }{1 - p^{t_1 - t_2 c}}. 
\end{equation*}
Plugging these bounds into~\eqref{eq:proof.path.better.eq1}, we obtain that 
\begin{align*}
\sum_{\theta \neq \eta }\sum_{\gamma \in \Gamma_{\theta \eta }(\cN) \colon e \in \gamma}  \phi(\gamma) |\gamma|_\ell
\leq   \;&   \frac{\bP_R (z, w) }{  1 - p^{- t_2 q}} \left\{ \sum_{\theta \in \Lambda(z)} \pi(\theta)^{1 - q} + 
 \sum_{\theta \in \Lambda(z)} \sum_{\eta \in \Theta }  \pi(\theta) \pi(\eta)^{1 - q}\right\} \\
 \leq \;& \frac{\bP_R (z, w) }{  1 - p^{- t_2 q}} \left\{ \frac{\pi(z)^{1 - q} }{1 - p^{t_1 - t_2 (1 - q)}} +    \frac{\pi(z)}{(1 - p^{t_1 - t_2})( 1 - p^{t_1 - t_2(1 - q)} )} \right\} \\
 \leq  \;& \frac{\bP_R (z, w) }{  1 - p^{- t_2 q}}  \frac{  \pi(z)^{1 - q} }{1 - p^{t_1 - t_2 (1 - q)}} \left( 1 + \frac{1}{1 - p^{t_1 - t_2}} \right). 
\end{align*}
Choose $q = (t_2 - t_1) / (2 t_2)$ such that 
\begin{align*}
    t_1 - t_2 (1 - q) = \frac{t_1 - t_2}{2} = -t_2 q. 
\end{align*} 
Using $\bP_R(z, w) = \bP(z, w) / \bP(z, \cN_R(z))$ and $\rho(e) \ell(e) = \pi(z)^{1 - q} \bP(z, w)$, we find that 
\begin{align*}
  \frac{1}{\rho(e) \ell(e)} \sum_{\theta \neq \eta }\sum_{\gamma \in \Gamma_{\theta \eta }(\cN) \colon e \in \gamma}  \phi(\gamma) |\gamma|_\ell
  \leq     \frac{1 + (1 - p^{t_1 - t_2})^{-1}  }{ [1 - p^{ (t_1 - t_2)/2}]^2 }   \bP(z, \cN_R(z))^{-1}. 
\end{align*}
The proof is then completed by applying~\eqref{eq:sinclair2} and Theorem~\ref{th:flow} . %noting that $ (z, w) \in \EDGE(\cN_R)$ implies $z \neq \theta^*$. 
We have the coefficient $2$ on the numerator of the bound given in the theorem  because we consider the lazy version. 
\end{proof}

\newpage 
\section{Preliminaries for graphical models}\label{sec:prelim}
\setallcounters

\subsection{Markov equivalence and independence maps} \label{sec:prelim.mec}

Below are some useful results for checking whether one DAG is Markov equivalent or an I-map of another DAG.  
The skeleton of a DAG is the unique undirected graph obtained by replacing all edges in the DAG with undirected ones. 
A v-structure is a triple $i \rightarrow j \leftarrow k$ (note that there is no edge between $i$ and $k$). 
By an edge reversal, we mean to change an existing edge $i \rightarrow j$ to $j \rightarrow i$. 
We say the edge $i \rightarrow j$ is covered if and only if $\Pa_i = \Pa_j \setminus \{i\}$. 
Two nodes $i, j$ are said to be adjacent in $G$ if $i \rightarrow j \in G$ or $j \rightarrow i \in G$. 

\begin{lemma}\label{lm:equiv1}
Two DAGs are Markov equivalent if and only if they have the same skeleton and v-structures. 
\end{lemma}
\begin{proof}
See~\citet{verma1991equivalence}.
\end{proof}

\begin{lemma}\label{lm:equiv2}
Two DAGs $G_1, G_2$ are Markov equivalent if and only if we can transform $G_1$ to $G_2$ by a sequence of covered edge reversals. 
\end{lemma}
\begin{proof}
See~\citet[Theorem 2]{chickering1995transformational}.
\end{proof}

\begin{lemma}\label{lm:chickering}
Let the DAG $H$ be an I-map of the DAG $G \neq H$. 
We can create a DAG $G'$ from $G$ by at least one of the following two operations such that $H$ is  an I-map of $G'$. 
\begin{enumerate}[label=(\roman*)]
    \item Reverse a covered edge $i \rightarrow j$ in $G$ such that $j \rightarrow i \in H$. 
    \item Add an edge $i \rightarrow j$ to $G$ such that $i, j$ are adjacent in $H$.  
\end{enumerate}
Consequently, there exists a sequence of DAGs $(G_0 = G, G_1,  \dots, G_{m - 1}, G_m = H)$ such that for each $i =  1, \dots, m$,  $H$ is an I-map of $G_i$,  and $G_i$ is obtained from $G_{i - 1}$ by either a covered edge reversal or a single-edge addition. 
\end{lemma}
\begin{proof}
See Theorem 4 and the remark following it in~\citet{chickering2002optimal}. 
\end{proof}

\subsection{Gaussian DAG models} \label{sec:prelim.normal}
We provide some background on the decomposition of $\Sigma(B, \Omega)$ given in~\eqref{eq:mod.chol}.
It is called the modified Cholesky decomposition of $\Sigma$ (up to permutation of rows and columns), and $B$ is the modified Cholesky factor (after scaling). 
Let $\chol(\sigma) = \bigcup_{G \in \dag^\sigma} \chol(G)$ where $\chol(G)$ is as given in~\eqref{eq:space.chol}.  
Observe that, equivalently, $\chol(\sigma)$ can be defined by
\begin{equation} \label{eq:def.chol.sigma}
\begin{aligned}
\chol(\sigma) = & \{ (B, \Omega) \colon B \in \bbR^{p \times p} , \,  B_{i j} = 0 \text{ if }  \sigma^{-1}(i) \geq \sigma^{-1}(j) , \text{ for any } i, j \in [p]; \\ 
 &  \Omega = \diag(\omega_1, \dots, \omega_p), \, \omega_i > 0 \text{ for any } i \in [p]\}. 
\end{aligned}
\end{equation}

\begin{lemma}\label{lm:chol.unique2}
For any positive definite  matrix $\Sigma \in \bbR^{p \times p}$ and $\sigma \in \bbS^p$, the decomposition $\Sigma = (I - B^\top)^{-1} \Omega (I - B)^{-1}$ for $(B, \Omega) \in \chol(\sigma)$ exists and is unique. 
\end{lemma}
\begin{proof}
First, permute the rows and columns of $\Sigma$ using $\sigma$ so that $B$ becomes upper triangular after permutation. The result then follows from the existence and uniqueness of Cholesky decomposition for positive definite matrices. 
See also~\citet[Lemma 2.1]{aragam2015learning}. 
\end{proof}

\begin{lemma}\label{lm:minimal}
Let $\Sigma = (I - B^\top)^{-1} \Omega (I - B)^{-1}$ for some $(B, \Omega) \in \cup_{\sigma \in \bbS^p} \chol(\sigma)$. 
Let $G$ be a DAG such that $i \rightarrow j \in G$ if and only if $B_{ij} > 0$. 
Then $G$ is a minimal I-map of $\NM_p(0, \Sigma)$; that is, $\NM_p(0, \Sigma)$ is Markovian w.r.t. $G$ but not Markovian w.r.t. any sub-DAG of $G$. 
\end{lemma}
\begin{proof}
See~\citet[Theorem 1]{peters2014identifiability}. 
\end{proof}

\begin{lemma}\label{lm:chol.unique}
Let $\NM_p(0, \Sigma)$ be a non-degenerate multivariate normal distribution Markovian w.r.t. a $p$-vertex DAG $G$. Then the decomposition $\Sigma = (I - B^\top)^{-1} \Omega (I - B)^{-1}$ for $(B, \Omega) \in \chol(G)$ exists and is unique. 
\end{lemma}
\begin{proof}
Let $\sX = (\sX_1, \dots, \sX_p) \sim \NM_p(0, \Sigma)$.  
The Markovian assumption implies that the density of $\sX$  factorizes according to $G$. 
Further, for any $S \subseteq [p] \setminus \{j\}$, $\sX_j \mid \sX_S$  follows a normal distribution and the conditional expectation is linear in $\sX_S$, which shows that $(B, \Omega) \in \chol(G)$ exists. 
Let $\sigma$ be a topological ordering of $G$.  Since $\chol(G) \subseteq \chol(\sigma)$, the uniqueness follows from Lemma~\ref{lm:chol.unique2}. 
\end{proof}

\begin{lemma}\label{lm:chol.perfect}
Let $G$ be a $p$-vertex DAG and with edge set $E$. %$E = \{ (i, j) \colon i \rightarrow j \in G \}$. 
Let $B$ be a $p \times p$ matrix such that $B_E = \{ B_{ij}\colon (i, j) \in E\}$ is sampled from an absolutely continuous distribution on $\bbR^{|E|}$ and other entries of $B$ are zeroes. 
Let $\Omega = \diag(\omega_1, \dots, \omega_p)$ be a diagonal matrix where $(\omega_1, \dots, \omega_p)$ is sampled from an absolutely continuous distribution on $(0, \infty)^p$. 
Let $\Sigma = (I - B^\top)^{-1} \Omega (I - B)^{-1}$.
Then $\NM_p(0, \Sigma)$ is perfectly Markovian w.r.t. $G$ almost surely. 
\end{lemma}
\begin{proof}
See~\citet[Theorem 3.2]{spirtes2000causation}. 
\end{proof}

\begin{example}
We give an example for minimal I-maps of Gaussian DAG models. 
Let $Z_1, Z_2, Z_3$ be independent standard normal random variables. Define 
\begin{equation}\label{eq:sem1}
    \sX_1 = Z_1, \quad 
    \sX_2 = b_1 X_1 + Z_2, \quad 
    \sX_3 = b_2 X_2 + Z_3. 
\end{equation} 
The DAG model $G$ corresponding to this SEM is $1 \rightarrow 2 \rightarrow 3$, which only has one ordering $\sigma = c(1, 2, 3)$. 
Denote the weighted adjacency matrix of $G$ by $B$, which has $B_{12} = b_1$ and $B_{23} = b_2$ (and all other entries are zero). 
Normality of $(Z_1, Z_2, Z_3)$ implies that $\sX = (\sX_1, \sX_2, \sX_3) \sim \NM_3(0, \Sigma)$ where 
\begin{align*}
    \Sigma = \begin{bmatrix}
    1 \;&  b_1 \;& b_1 b_2 \\
    b_1 \;& b_1^2 + 1 \;& (b_1^2 +1 )b_2 \\
    b_1 b_2 \;& (b_1^2 +1 )b_2 \;& b_1^2 b_2^2 + b_2^2 + 1 \\ 
    \end{bmatrix} = (I - B^\top)^{-1} (I - B)^{-1}. 
\end{align*}
%where 
% \begin{align*}
%     B = \begin{bmatrix}
%     0 \;&  b_1 \;& 0\\
%     0 \;& 0 \;& b_2 \\
%     0 \;&  0 \;& 0 \\ 
%     \end{bmatrix}. 
% \end{align*} 
A routine calculation shows that~\eqref{eq:sem1} can be equivalently written as  
\begin{equation}\label{eq:sem2}
\begin{aligned}
    \sX_1 =\;& Z_1,  \\
    \sX_3 =\;& b_1 b_2 \sX_1 + b_2Z_2 +  Z_3, \\
    \sX_2 =\;& \frac{b_1}{b_2^2 + 1} \sX_1 + \frac{b_2}{b_2^2 + 1} \sX_3 + \frac{1}{b_2^2 + 1} Z_2 - \frac{b_2}{b_2^2 + 1} Z_3. 
\end{aligned}
\end{equation}
This SEM clearly corresponds to the ordering $\tau = (1, 3, 2)$, from which one obtains the modified Cholesky decomposition factor $(B_\sigma, \Omega_\sigma)$; that is, $\Sigma =  (I - B^\top_\sigma)^{-1} \Omega_\sigma (I - B_\sigma)^{-1}$ where
\begin{align*}
    B_\sigma = \begin{bmatrix}
    0 \;&  b_1 /(b_2^2 + 1) \;& b_1 b_2 \\ 
    0 \;& 0 \;& 0 \\
    0 \;&  b_2 / (b_2^2 + 1) \;& 0 \\
    \end{bmatrix}, \quad \quad  \Omega_\sigma = \begin{bmatrix}
    1 \;&  0 \;& 0 \\ 
    0 \;&  1/(b_2^2 +1) \;& 0 \\
    0 \;&  0 \;& b_2^2 + 1 \\
    \end{bmatrix}. 
\end{align*}
In particular, the $i$-th diagonal element of $\Omega_\sigma$ is the error variance of $\sX_i$ in the SEM~\eqref{eq:sem2}. 
\end{example}

\subsection{CPDAGs and CPDAG operators} \label{sec:supp.cpdag}
  
A  partially directed acyclic graph (PDAG) contains directed and/or undirected edges but no directed cycles (a cycle is directed if it has at least one directed edge).  
For a PDAG $H$, we use $\adj_j(H)$ to denote the set of all nodes that are connected to node $j$ by either a directed or undirected edge, in which case we say the two nodes are adjacent, and $\und_j(H) = \adj_j(H) \setminus ( \Pa_j(H) \cup \Ch_j(H) )$ to denote the set of nodes that are connected to $j$ by an undirected edge.  

For an equivalence class $\cE$, its CPDAG or essential graph, denoted by $\EGG{\cE}$, is the unique PDAG with the same skeleton as any $G \in \cE$ such that an edge in $\EGG{\cE}$ is directed if and only if the edge is directed in the same orientation in every $G \in \cE$. 
For convenience, we will also denote $\EGG{\cE}$ by  $\EGG{G}$ for any $G \in \cE$. 
Given a CPDAG $\EGG{\cE}$, a CPDAG operator defines some specific modification of $\EGG{\cE}$, such as adding or deleting edges. 
After we apply such an operator to  $\EGG{\cE}$,  we may get a PDAG that is not a CPDAG, but it may be converted to a CPDAG in a unique way. 

\begin{definition}[consistent extension]\label{def:consist.ext}
A DAG $G$ is a consistent extension of a PDAG $H$ if (i) $G$ and $H$ have the same skeleton and v-structures, and (ii) each directed edge in $H$ has the same orientation as in $G$. 
\end{definition}

\begin{definition}
A CPDAG operator is valid if  the modified graph is a PDAG and has a consistent extension. 
% and (ii) all modified edges occur in the resulting CPDAG. 
\end{definition}  

Note that~\citet{he2013reversible} used a slight different   definition of ``valid'' operators: they further required that a valid operator must also be ``effective'', which means that all modified edges occur in the resulting CPDAG~\citep[Definition 3]{he2013reversible}. 
By Definition~\ref{def:consist.ext}, all consistent extensions of a given PDAG are Markov equivalent. Hence, a valid CPDAG operator always results in a unique CPDAG (i.e., a unique equivalence class).    

 \citet{he2013reversible} considered the following six types of CPDAG operators: insert/delete an undirected edge, insert/delete a direct edge, and  make/remove a v-structure.  
``Make a v-structure'' means to convert a subgraph $i - j - k$ in the CPDAG to $i \rightarrow j \leftarrow k$; similarly, ``remove a v-structure'' means to convert a v-structure $i \rightarrow j \leftarrow k$ to $i - j - k$.  
However, there exist CPDAGs $\EGG{\cE}$ and $\EGG{\cE'}$ such that  we can get  $\EGG{\cE'}$ from $\EGG{\cE}$ by one of these operators, but not vice versa. 
To overcome this problem, \citet[Definition 9]{he2013reversible} constructed a set of rules defining, for each CPDAG, which of the above six operators are allowed.  
The resulting set of operators is both  ``irreducible'' and ``reversible'' on the unrestricted CPDAG space. 

\citet{chickering2002optimal} introduced two types of CPDAG operators for implementing the GES algorithm, $Insert$ and $Delete$, which can be defined as follows.

\begin{definition}[Insert$(i, j, S)$]\label{def:ges.insert}
Given a CPDAG $H$, ``Insert$(i, j, S)$'' operator is defined for non-adjacent nodes $i, j$ and a set $S \subseteq \und_j(H) \setminus \adj_i(H)$. It modifies $H$ by (i) inserting the edge $i \rightarrow j$, and (ii) for each $k \in S$, directing the edge between $k$ and $j$ as $k \rightarrow j$.  
\end{definition}

\begin{definition}[Delete$(i, j, S)$]\label{def:ges.del}
Given a CPDAG $H$, ``Delete$(i, j, S)$'' operator is defined for adjacent nodes $i, j$ connected as either $i - j$ or $i \rightarrow j$  and a set $S \subseteq \und_j(H) \cap \adj_i(H)$. It modifies $H$ by (i) deleting the edge between $i$ and $j$, and (ii) for each $k \in S$, directing the edge between $k$ and $j$ as $j \rightarrow k$ and any undirected edge between $k$ and $i$ as $i \rightarrow k$. 
\end{definition}

There are easily testable conditions for checking whether some $Insert$ or $Delete$ operator is valid; see~\citet[Theorems 15 and 17]{chickering2002optimal}. 
Given a CPDAG $H$ with $p$ nodes, let $\ges(H)$ denote the set of all  $Insert$ and $Delete$ operators defined for $H$ (which may or may not be valid). 
Let $d$ denote the degree of $H$.  Then,  we always have $ |\ges(H)|  \leq  p^2 2^d$. 
Further, the time complexity of converting the resulting PDAG to CPDAG is $O(  p d^3 )$~\citep{chickering2002optimal}. 
For each $k \in S$, the operator $Insert(i, j, S)$ creates a new v-structure $i \rightarrow j \leftarrow k$.
Similarly, for each $k \in S$, the operator $Delete(i, j, S)$ creates a new v-structure $i \rightarrow k \leftarrow j$.
Hence, given a CPDAG $H$, the only $Insert$ operators that are not distinguishable (i.e., two operators result in the same CPDAG) are   $Insert(i, j, \emptyset)$ and $Insert(j, i, \emptyset)$ when $\Pa_i(H) = \Pa_j(H)$ by~\citet[Theorem 34]{chickering2002learning}, 
and the  only $Delete$ operators that are not distinguishable are $Delete(i, j, S)$ and $Delete(j, i, S)$ when $i, j$ are connected by an undirected edge in $H$. 
This observation reveals that the number of operators in $\ges(H_1)$ that can convert a CPDAG $H_1$ to another CPDAG $H_2$ is always the same as the number of operators in $\ges(H_2)$ that can convert $H_2$ to $H_1$. 
To see this, fix an arbitrary CPDAG $H_2$ which can be obtained by applying $Insert(i, j, S)$ to some CPDAG $H_1$ (which implies $i, j$ are not adjacent in $H_1$). 
If  the edge $i \rightarrow j$ is directed in $H_2$, then there is only one $Delete$ operator that can transform $H_2$ back to $H_1$. 
If $i, j$ are connected by an undirected edge in $H_2$, it implies that $S = \emptyset$ and $\Pa_i(H_1) = \Pa_j (H_1)$ by~\citet[Theorem 34]{chickering2002learning}. 
Thus, both $Insert(i, j, \emptyset)$ and $Insert(j, i, \emptyset)$ can convert $H_1$ to $H_2$, and there are two  $Delete$ operators that can transform $H_2$ back to $H_1$.

\newpage 

\section{Proofs for Section~\ref{sec:main} }\label{sec:supp.proofs}
\setallcounters

\subsection{Proof of Lemma~\ref{lm:bound.nb}} \label{sec:proof.bound.nb}

We need two auxiliary lemmas. 

\begin{lemma}\label{lm:equiv.node1}
Let $G_1$ and $G_2$ be two Markov equivalent DAGs such that $\Pa_j(G_1) = \Pa_j(G_2)$ and $i \notin \Pa_j(G_1)$.  
Suppose that $G'_1 = G_1 \cup \{ i \rightarrow j \}$ and $G'_2  = G_2 \cup \{i \rightarrow j \}$ are both DAGs. 
Then $G'_1, G'_2$ are Markov equivalent. 
\end{lemma}
\begin{proof}
By Lemma~\ref{lm:equiv1}, $G'_1$ and $G'_2$ have the same skeleton. It suffices to show that $G'_1$ and $G'_2$ share the same v-structures. One can see that a v-structure remains unchanged unless it involves both $i$ and $j$. There are only two cases where adding the edge $i \rightarrow j$ affects some v-structure: 
 \begin{enumerate}%[({Case} 1)]
     \item  For some node $k$, a new v-structure $i \rightarrow j \leftarrow k$ is formed. 
     \item  For some node $k$, an existing v-structure $i \rightarrow k \leftarrow j$ is shielded.   
 \end{enumerate}
We show that if Case 1 or Case 2 happens in $G'_1$, it also happens in $G'_2$.
If $i \rightarrow j \leftarrow k$ is a v-structure in $G'_1$ for $k \in \Pa_j(G_1)$, $i, k$ are not adjacent in $G'_1$ (``adjacent'' means two nodes are connected by an edge). Then $i, k$ must be non-adjacent in $G'_2$ as well since $G'_1$ and $G'_2$ have the same skeleton.  By the assumption  $\Pa_j(G_1) = \Pa_j(G_2)$, we find that $k \in \Pa_j(G_2)$  and thus $G'_2$ contains the v-structure $i \rightarrow j \leftarrow k$. 
If $G_1$ has a v-structure $i \rightarrow k \leftarrow j$ for some node $k$, so does $G_2$ by Lemma~\ref{lm:equiv1}. Adding $i \rightarrow j$ shields the v-structure both in $G'_1$ and $G'_2$. 
\end{proof}

\begin{lemma}\label{lm:equiv.node2}
Let $G_1$ and $G_2$ be two Markov equivalent DAGs such that $\Pa_j(G_1) = \Pa_j(G_2)$ and $i \in \Pa_j(G_1)$.  
Let   $G'_1 = G_1 \setminus \{ i \rightarrow j \}$ and $G'_2  = G_2 \setminus \{i \rightarrow j \}$. 
%Suppose that $G'_1 = G_1 \setminus \{ i \rightarrow j \}$ and $G'_2  = G_2 \setminus \{i \rightarrow j \}$ are both DAGs. 
Then $G'_1, G'_2$ are Markov equivalent. 
\end{lemma}
\begin{proof}
The proof is similar to that of Lemma~\ref{lm:equiv.node1}. we show that $G'_1$ and $G'_2$ share the same v-structures. There are only two cases where deleting the edge $i \rightarrow j$ affects the v-structure. 
 \begin{enumerate}%[({Case} 1)]
     \item  For some node $k$, a new v-structure $i \rightarrow k \leftarrow j$ is formed. 
     \item  For some node $k$, an existing v-structure $i \rightarrow j \leftarrow k$ is broken up. 
 \end{enumerate}
Let $i \rightarrow k \leftarrow j$ be a v-structure in $G'_1$. 
Observe that the assumptions $\Pa_j(G_1) = \Pa_j(G_2)$ and that $G_1$ and $G_2$ are Markov equivalent imply  $\Ch_j(G_1) = \Ch_j(G_2)$.  
Since $k \in \Ch_j(G_1)$, $G'_2$ also has the edge $k \leftarrow j$. 
Further, we must have $i \rightarrow k \in G'_2$, since  $G_2$ is acyclic. 
This shows that $i \rightarrow k \leftarrow j$ is also a v-structure in $G'_2$. If $G_1$ has a v-structure $i \rightarrow j \leftarrow k$ for some node $k$, so does $G_2$. Removing $i \rightarrow j$ destroys the v-structure in both $G'_1$ and $G'_2$.
\end{proof}
 
%%%% surprising, this result is not useful... 
% \begin{lemma}\label{lm:equiv.max.in}
% Let $G_1$ and $G_2$ be two Markov equivalent DAGs. Then, $G_1$ and $G_2$ have the same maximum in-degree. 
% \end{lemma}
% \begin{proof}
% By Lemma~\ref{lm:equiv2}, it suffices to show that any covered edge reversal does not change the maximum in-degree.  
% Let $G_1, G_2$ be two DAGs that differ by a covered edge reversal. Thus, there exist $i\neq j$ such that $i \rightarrow j \in G_1$, $j \rightarrow i \in G_2$, $\Pa_i(G_1) = \Pa_j(G_1) \setminus \{i\}$, and all the other edges are exactly the same in the two DAGs. 
% Let $k = |\Pa_i(G_1)|$, which implies $|\Pa_j(G_1)| = k + 1$. 
% Clearly, $|\Pa_i(G_2)| = k + 1$ and $|\Pa_j(G_2)| = k$. 
% Since $\Pa_l(G_1) = \Pa_l(G_2)$ for any $l \neq i, j$, we conclude that the maximum in-degree of $G_1$ is the same as that of $G_2$. 
% \end{proof}

\begin{proof}[Proof of Lemma~\ref{lm:bound.nb}]
Fix an arbitrary $\cE \in \cpdag(\din, \dout)$.
By the definition of $\cpdag(\din, \dout)$,  there exists some $G_0 \in \cE$ such that $G_0 \in \dag(\din, \dout)$. 
Since any Markov equivalent DAGs must have the same skeleton, for any $G \in \cE$, the degree of any node is bounded by $\din + \dout$ and the set of adjacent nodes is fixed. 
Define $\Pa_j(\cE) = \{ \Pa_j(G) \colon G \in \cE  \}$, the collection of all possible parent sets of node $j$. It then follows that $|\Pa_j(\cE)| \leq 2^{\din + \dout}$. 

Consider $\adds(\cE)$. 
Suppose that $\cE'$ is obtained by adding the edge $i \rightarrow j$ to some $G \in \cE$. 
By Lemma~\ref{lm:equiv.node1}, for $G_1, G_2 \in \cE$, $G_1 \cup \{i \rightarrow j\}$ and $G_2 \cup \{i \rightarrow j\}$ are still Markov equivalent if the resulting graphs are DAGs and $\Pa_j(G) = \Pa_j(G')$. 
But $|\Pa_j(\cE)| \leq  2^{\din + \dout}$ implies that adding the edge $i \rightarrow j$ can  yield at most $2^{\din + \dout}$ different equivalence classes. 
Hence, $|\adds(\cE)| \leq p (p - 1) 2^{\din + \dout}$ since there are at most $p(p-1)$ directed edges we can add. 

A similar argument can be applied to $\dels(\cE)$. For any $G \in \cE$, we have $|G| \leq p  \din$, but these edges may be directed in either way. Using Lemma~\ref{lm:equiv.node2}, we then find that 
$|\dels(\cE)| \leq 2 p \din 2^{\din + \dout}$. 
Finally, for any $G\in \cE$,  $|\swaps(G)| \leq p (p - 1)(\din + \dout)$. Hence,  $|\swaps(\cE)| \leq p (p - 1) (\din + \dout)  2^{\din + \dout}$. Combing the results for the three cases, we obtain the asserted upper bound on $|\nbg(\cE)|$. 
%To prove the lower bound on $|\nbg(\cE)|$, fix an arbitrary $G \in \cE$. For any $i \neq j$ such that $i$ precedes $j$ in the topological ordering of $G$, we can either add $i \rightarrow j$ to $G$ (or perform a swap if necessary), or remove it from $G$. The asserted lower bound then follows.
\end{proof} 

\subsection{Proof of Lemma~\ref{lm:pigeonhole}} \label{sec:proof.exist}

\begin{proof} 
We prove  by contradiction. 
Assume that for any $j \in [p]$ such that $g_j^\sigma(G) \neq G$, we have $g_j^\sigma(G) \notin \dag^\sigma(\din, \dout)$. 
It follows that there exists no $j$  such that $\Pa_j(G^*_\sigma) \subsetneqq \Pa_j(G)$. 
Otherwise, $g_j^\sigma(G)$ can be obtained from $G$ by removing some incoming edge of node $j$ and thus  $g_j^\sigma(G) \in \dag^\sigma(\din, \dout)$. 

Since  there is no strictly overfitted node,  there must exist some underfitted node. Let 
\begin{align*}
U = \{ u \in [p] \colon \Pa_u(G^*_\sigma) \not\subseteq \Pa_u(G)  \}
\end{align*}
be the set of all underfitted nodes.  
Recall that for any $u \in U$, $g^\sigma_u(G)$ is constructed by adding an incoming edge of node $u$ to $G$. 
Define 
\begin{align*}
A = \{ a \in [p] \colon  g^\sigma_u(  G ) = G \cup \{ a  \rightarrow u \}  \text{ for some } u \in U \}. 
\end{align*}
Fix an arbitrary $a \in A$ and suppose that $g^\sigma_j(G) = G \cup \{a \rightarrow j\}$ for some $j$. 
If $| \Ch_{a}(G) | < \dout$, one can verify that $g_j^\sigma(G) \in \dag^\sigma(\din, \dout)$ 
(if $\Pa_j(G) < \din$, we can simply add the edge $a  \rightarrow j$; if $\Pa_j(G) = \din$,  we perform a swap). 
Thus,  $| \Ch_{a}(G) | = \dout$ for every $a \in A$. 

For any node $u \in U$, there exists some node $a(u) \in A$  such that $a(u) \rightarrow u$   is in $G^*_\sigma$ but not in $G$. 
But any node $a$ can have at most $d^*_\sigma$ outgoing edges in $G^*_\sigma$, which implies that  
\begin{equation}\label{eq:exist1}
|A | \geq \frac{|U|}{d^*_\sigma}. 
\end{equation}
For each $a \in A$, define 
\begin{align*}
c_a = \left|  \{ u \in U \colon  g^\sigma_u(  G ) = G \cup \{ a  \rightarrow u \}   \} \right|, \quad F_a = \Ch_a(G) \setminus \Ch_a(G^*_\sigma). 
\end{align*}
So $c_a$ is the number of nodes in $U$  which we want to connect with the parent node $a$. 
Observe that $c_a \leq | \Ch_a(G^*_\sigma ) \setminus \Ch_a(G)  |$ and $\sum_{a \in A} c_a = |U|$. 
Using $| \Ch_a(G) | = \dout$ and $|\Ch_a(G^*_\sigma)| \leq d^*_\sigma$, we find that  $|F_a| \geq \dout - d^*_\sigma + c_a$. 
Hence, 
\begin{equation}\label{eq:exist2}
 \sum_{a \in A} |F_a|  \geq \sum_{a \in A} \left( \dout - d^*_\sigma + c_a \right)  
  \geq \frac{|U| \dout }{d^*_\sigma}  
\end{equation}
where the last step follows from~\eqref{eq:exist1} and   $\sum_{a \in A} c_a = |U|$. 

For any node $f_a \in F_a$, the edge $a \rightarrow f_a$ is in $G$ but not in $G^*_\sigma$. 
Define
\begin{align*}
E = \{ (a, f) \colon a \in A, \, f \in F_a \}. 
\end{align*}
Clearly, $| E | = \sum_{a \in A} |F_a | $. 
Since the maximum in-degree of $G$ is at most $\din$, for any $f \in [p]$,  we have $| \{ a \in A \colon  (a, f) \in E  \} | \leq \din$. 
Consider the set $\bar{F} = \bigcup_{a \in A} F_a$. By~\eqref{eq:exist2} we have that 
\begin{align*}
| \bar{F}  |  \geq   \frac{ |E| }{\din} \geq   \frac{|U| \dout }{d^*_\sigma \din} > |U|, 
\end{align*}
since we assume $\dout > \din d^*_\sigma$. 
For any node $f \in \bar{F}$, we have $\Pa_f(G) \neq \Pa_f (G^*_\sigma)$. 
Because we have already shown that no node in $G$ can be strictly overfitted, all nodes in $\bar{F}$ must be underfitted; thus, we must have $|\bar{F}| \leq |U|$, which yields the contradiction.
\end{proof}

\subsection{Proof of Corollary~\ref{coro:path.dag}} \label{sec:proof.exist2}
\begin{proof}
Consider part (i) first. 
By Lemma~\ref{lm:pigeonhole}, for any $G \in \dag^\sigma(\din, \dout)$ such that $G \neq G^*_\sigma$,   there exists some $j \in [p]$ such that $G' = g_j^\sigma(G) \in \dag^\sigma(\din, \dout)$ and $G' \neq G$.  
Define $g^\sigma(G) = g_j^\sigma(G)$. Then, by the definition of $g^\sigma_j$, we always have $\HD(g^\sigma(G), G^*_\sigma) < \HD(G, G^*_\sigma)$ and 
\begin{equation}\label{eq:coro.path.dag.1}
\begin{aligned}
\score( g^\sigma(G) ) - \score(G) = \;&  \score( g^\sigma_j (G) ) - \score(G) \\
= \;& \sum_i  \left\{  \score_i ( g^\sigma_j (G)  ) - \score_i(G) \right\}  \\
= \;& \scorej (g^\sigma_j (G) ) - \scorej(G) \\
=\;&   \scorej (g^\sigma_j ( \Pa_j(G)) ) - \scorej(\Pa_j(G)). 
\end{aligned} 
\end{equation}
Since $\HD(G, G^*_\sigma) \leq (d^*_\sigma + \din) p$ for any $G \in \dag^\sigma(\din, \dout)$, the claim then follows.  

For part (ii) of Corollary~\ref{coro:path.dag}, observe that for any $G \neq G^*_\sigma$,  if~\eqref{eq:conj.scorej} holds, then it follows from~\eqref{eq:coro.path.dag.1} that $e^{\score(g^\sigma(G))} / e^{\score(G)} \geq p^t$. 
It is also clear from construction that $g^\sigma(G) \in \nb(G)$. 
Hence, $g^\sigma$ is the canonical transition function that satisfies part (ii) of Condition~\ref{cond:unimodal} with $t_2 = t$. 
It only remains to bound the neighborhood size. 
Clearly, $|  \left(  \adds(G) \cup \dels(G)  \right)  \cap \dag^\sigma(\din, \dout)| \leq p(p - 1) / 2$ since there are only $p(p - 1) / 2$ directed edges compatible with ordering $\sigma$. 
For the swap moves, observe that if $\sigma(j) = k$, then there are at most $k-1$ incoming edges of node $j$ we can add and $\din$ incoming edges of node $j$ we can remove. Hence, $|  \swaps(G)  | \leq  \sum_{k=1}^p (k - 1) \din  \leq \din p (p - 1) / 2$. 
Without loss of generality, we can assume $\din \leq p - 1$, which gives 
\begin{equation}\label{eq:dag.nb.size}
| \nb(G) \cap \dag^\sigma(\din, \dout) | \leq \frac{ (\din + 1)p (p  - 1)}{2} \leq \frac{ \din p^2 }{2}. 
\end{equation}
This shows that part (ii) of Condition~\ref{cond:unimodal} holds with $t_1 = 3$. 
\end{proof}

\subsection{Proof of Theorem~\ref{th:ges.path}} \label{sec:proof.ges.path}

Consider the claim  $G^*_\sigma \in \dag^\sigma(\din, \dout)$ first. 
Without loss of generality, assume $\din \leq p, \dout \leq p$.  If $d^* = 0$, then $G^*$ is the empty DAG and the claim holds trivially. So, we can assume $d^* \geq 1$. But then $\min\{ d^* \din + 1, p \} \leq \dout$ implies that $\dout \geq \din \geq d^*$. Since $d^*$ is the maximum degree of all minimal I-maps, this shows that $G^*_\sigma \in \dag^\sigma(\din, \dout)$ for each $\sigma$. 
 
To prove the existence of function $g$, we first derive an auxiliary lemma. 

\begin{lemma}\label{lm:imap.map}
Assume $d^* \leq \din \leq \dout$.  
Consider $G^*_\sigma$ for some $\sigma \in \bbS^p$ such that $\cE = \EG{G^*_\sigma} \neq \cE^*$. 
There exist $ \cE'\in \cpdag(\din, \dout), G \in \cE$ and $\tau \in \bbS^p$ such that 
(i)   $\cE' \in \dels(\cE)$ and $\CI(\cE) \subset \CI(\cE') \subseteq \CI(\cE^*)$, and 
(ii)  $G \in \dag^\tau(\din, \dout)$ and $\cE' = \EG{  g_j^\tau(G) }$ for some $j \in [p]$. 
\end{lemma}

\begin{proof} 
The first conclusion of the lemma is essentially the same as~\citet[Lemma 10]{chickering2002optimal}. 
To prove it, note that since $\cE = \EG{ G^*_\sigma}$, we have $\CI( \cE) = \CI(G^*_\sigma) \subset \CI(\cE^*)$. 
%By~\citet[Theorem 4]{chickering2002optimal} (see also Lemma~\ref{lm:chickering}), 
By Lemma~\ref{lm:chickering}, there exist some finite $m$ and a sequence of DAGs, $(G_0 = G^*, G_1, \dots, G_{m - 1}, G_m = G^*_\sigma)$, such that, for each $k \in [m]$, $\CI(G_k) \subseteq \CI ( G_{k - 1} )$ and $G_k$ is obtained from $G_{k-1}$ by either a covered edge reversal or an edge addition. 
Let $\ell = \max\{ j \leq m \colon |G_j| = |G_\sigma^*| - 1 \}$, which clearly exists. 
Then, $G_\ell \in \dels( G_{\ell + 1} )$, $G_{\ell + 1} \in \cE$ by Lemma~\ref{lm:equiv2},  and $\CI(G_{\ell + 1}) \subset \CI(G_{\ell}) \subseteq \CI(G^*)$. 

Let $\tau$ be a topological ordering of $G_{\ell + 1}$. 
For each $i \in [p]$, we  have $\Pa_i( G^*_\tau) \subseteq \Pa_i(G_{\ell + 1})$, since  $G_{\ell + 1}$ is an I-map of $G^*$ but $G^*_\tau$ is the unique minimal I-map with ordering $\tau$. 
Meanwhile,  $\CI(G_\ell) \subset \CI(G^*)$ and $|G_\ell| < |G_{\ell + 1}|$ imply that there exists some node $j$ such that $\Pa_j( G^*_\tau) \subsetneqq \Pa_j(G_{\ell + 1})$.   
Consider the DAG $G' = g_j^\tau( G_{\ell + 1})$, which by definition satisfies that   $G' = G_{\ell + 1} \setminus \{ i \rightarrow j \}$ for some $i \neq \Pa_j(G^*_\tau)$. 
Hence, $G' \in \dels(G)$, $\cE' = \EG{G'} \in \dels(\cE)$ and $\CI(\cE) = \CI(G_{\ell + 1}) \subset \CI(\cE')\subseteq \CI(\cE^*)$.
To conclude the proof, notice that the maximum degree of $G_{\ell + 1}$ is bounded by $d^*$ since $G_{\ell + 1}$ and $G^*_\sigma$ are Markov equivalent.
It then follows from the assumption $d^* \leq \din \leq \dout$ that  $G_{\ell + 1} \in \dag^\tau(\din, \dout)$ and  thus $\cE' \in \cpdag(\din, \dout)$.  
\end{proof}

\begin{proof}[Proof of Theorem~\ref{th:ges.path}]
We give an explicit construction of  $g$ that satisfies the required conditions.  
%Recall the definition of $\DS(\cE^*)$ in~\eqref{eq:def.dist}. 
%This two-step procedure suggests that we may measure how ``close'' an equivalence class is to $\cE^*$ using the function $\DS$ defined below. 
For each $\cE \in \cpdag(\din, \dout)$, define the ``distance'' from $\cE$ to $\cE^*$ by 
\begin{equation}\label{eq:def.dist}
        \DS(\cE) = \min  \left\{ \HD(G, G^*_\sigma) + |G^*_\sigma| - |G^*| \colon G \in \cE \cap \dag^\sigma(\din, \dout), \sigma \in \bbS^p   \right\},
\end{equation}
where we recall $\HD$ denotes the Hamming distance.  Note that for any $\sigma$, $| G^*_\sigma| \geq |G^*|$ since $G^*_\sigma$ is an I-map of $G^*$; recall  Lemma~\ref{lm:chickering}. %(which implies that if $i \rightarrow j \in G^*$, then $i, j$ are connected in $G^*_\sigma$ as well). 
Thus, $\DS(\cE) = 0$ if and only if $\cE = \cE^*$.
We will construct $g$ on $(\cpdag(\din, \dout), \nbg)$ such that $\DS( g(\cE) ) < \DS( \cE )$ for any $\cE \neq \cE^*$. 

Let $(\bar{G}(\cE), \bar{\sigma}(\cE))$ be the pair that attains the minimum in the definition of $\DS(\cE)$ (if there are multiple such pairs, fix one of them).   
The pair $(\bar{G}(\cE), \bar{\sigma}(\cE))$ can be seen as a ``canonical'' representation of the equivalence class $\cE$. 
Fix an arbitrary  $\cE \in \cpdag(\din, \dout)$, and we define $g(\cE)$ as follows. 
\begin{enumerate}%[(1)]
    \item Let $(G, \sigma) =  (\bar{G}(\cE), \bar{\sigma}(\cE))$  be its canonical representation. 
    \item If $G \neq G^*_\sigma$, by Lemma~\ref{lm:pigeonhole}, there exists some $j$ such that $g_j^\sigma(G) \in \dag^\sigma(\din, \dout)$ and $g_j^\sigma(G) \neq G$. Let $G' = g_j^\sigma(G)$ and $g( \cE ) = \EG{G'}  \in \cpdag(\din, \dout)$. 
    \item If $G = G^*_\sigma$ but $\cE \neq \cE^*$, by Lemma~\ref{lm:imap.map}, there exist some $G_0 \in \cE, \tau \in \bbS^p, j \in [p]$ such that $G_0 \in \dag^\tau(\din, \dout)$, $G' = g_j^\tau(G_0) \in \dels(G_0)$, and $G'$ is still an I-map of $G^*$. 
    Let $g(\cE) = \EG{G'} \in \cpdag(\din, \dout)$.
    \item If $\cE = \cE^*$, let $g(\cE) = \cE$.
\end{enumerate}
It follows from the definition of $\nbg$ that $g(\cE) \in \nbg(\cE)$ for every $\cE \neq \cE^*$. 
Consider $g_j^\sigma(G)$ found in step (2). By Lemma~\ref{lm:pigeonhole} and the definition of $g_j^\sigma$, we have 
\begin{align*}
\DS( \EG{ g_j^\sigma(G) } )  \leq  \HD(g_j^\sigma(G),  G^*_\sigma) + |G^*_\sigma|  - |G^*| < \HD(G, G^*_\sigma) + |G^*_\sigma|  - |G^*| = \DS(\cE)
\end{align*}
since $g_j^\sigma(G) \in \dag^\sigma(\din, \dout)$ and $(G, \sigma)$ is chosen to be the canonical representation of $\cE$.  
If $g(\cE)$ is defined in step (3) as $g_j^\tau(G_0)$ for some $G_0 \in \cE$ (and $G = G^*_\sigma \in \cE$),  we claim  
\begin{align*}
\DS( \EG{ g_j^\tau(G_0) } )  \leq  \HD(g_j^\tau(G_0),  G^*_\tau) + |G^*_\tau|  - |G^*| 
<  |G^*_\sigma|  - |G^*| = \DS(\cE).
\end{align*}
This is because $g_j^\tau(G_0)$ is an I-map of $G^*$ with ordering $\tau$, and thus $\HD(g_j^\tau(G_0),  G^*_\tau) = | g_j^\tau(G_0) | - |G^*_\tau|$. The above inequality then follows upon noticing that $|G^*_\sigma| = |G_0| > |g_j^\tau(G_0)|$. 

Hence, for any $\cE \neq \cE^*$, we have $\DS(g(\cE)) < \DS(\cE)$. Since $\DS(\cE) = 0$ if and only $\cE = \cE^*$,  $g$ induces a canonical path from $\cE$ to $\cE^*$ for every $\cE \in \cpdag(\din, \dout)$. 
To bound the length of such paths, it suffices to bound $\DS(\cE)$. Observe that 
\begin{align*}
     \DS(\cE)  \leq \;& \max_{G \in \dag^\sigma(\din, \dout), \sigma \in \bbS^p} \HD(G, G^*_\sigma) + \max_{\sigma \in \bbS^p} (|G^*_\sigma| - |G^*|) \\
     \leq \;& \max_{G \in \dag(\din, \dout), \sigma \in \bbS^p} \left( |G| +  |G^*_\sigma| \right)   +  \max_{\sigma \in \bbS^p} |G^*_\sigma| \\
     \leq \;& (d^* + \din) p + d^* p = (2 d^* + \din)p,
\end{align*}
which concludes the proof.  
\end{proof}  

\subsection{Proof of Corollary~\ref{coro:path.cpdag}} \label{sec:proof.coro.cpdag}
\begin{proof}
Fix an arbitrary $\cE \neq \cE^*$, and let $g(\cE) = \EG{ g_j^\sigma(G) }$ for  $j \in [p], \sigma \in \bbS^p$ and $G \in \cE \cap \dag^\sigma(\din, \dout)$ as given in Theorem~\ref{th:ges.path}. 
The proof of Theorem~\ref{th:ges.path} also shows that $ g_j^\sigma(G) \in \nb(G) \cap  \dag^\sigma(\din, \dout)$, and thus $g(\cE) \in \nb(\cE) \cap \cpdag(\din, \dout)$. 
By the score equivalence and~\eqref{eq:coro.path.dag.1}, 
\begin{equation}\label{eq:score.to.scorej}
\score( g(\cE) ) - \score(\cE) =  
  \score( g^\sigma_j (G) ) - \score(G)   
= \scorej (g^\sigma_j ( \Pa_j(G)) ) - \scorej(\Pa_j(G)). 
\end{equation} 
So,   if~\eqref{eq:conj.scorej} holds, $e^{\score( g(\cE))} / e^{\score(\cE)} \geq p^t$, which completes the proof. 
\end{proof}

\newpage 
\section{High-dimensional empirical variable selection}\label{sec:supp.sel}
\setallcounters

By the SEM representation, the DAG selection problem can be seen as a series of variable selection problems. In this supplement, we prove some high-dimensional consistency results for a single variable selection problem using the empirical normal-inverse-gamma prior.    
We still use $\X$ to denote an $n \times p$  data matrix and we use $y$ to denote a response vector. Let 
\begin{equation}\label{eq:def.cS}
\cS(m, d) =  \{  S \subseteq [m] \colon   |S| \leq d   \} 
\end{equation}
denote the set of candidate models for some $m, d \in [p]$; that is, we want to select at most $d$ variables from $\sX_1, \dots, \sX_m$ to explain $y$.  
We allow $m$ to be less than $p$, which is different from standard variable selection setups,  so that the results can be easily applied to DAG selection problems. In particular, without loss of generality, one may assume $y = X_{m + 1}$ if $m < p$. 

\subsection{Model, prior and posterior distributions}\label{sec:supp.var.modl}
Let  $y \in \bbR^n$, $\X  \in \bbR^{n \times p}$ and $m, d \in [p]$ be given, and $\cS(m, d)$ be as defined in~\eqref{eq:def.cS}. 
Consider the following  empirical Bayes model for variable selection~\citep{martin2017empirical},  
\begin{align*}
y  =\;&  \Z_S \beta_S + \varepsilon, \quad \varepsilon \sim  \NM_n( 0,  \,  \omega I ), \\ 
\pi_0( S ) \propto \;&    \left(  c_1p^{c_2} \right)^{-|S|}  \ind_{\cS(m, d)} (S),     \\ 
\pi_0( \omega)  \propto \;&   \omega ^{- \kappa/2 - 1},   \\
\beta_S \mid S, \omega \sim \;&   \NM_{|S|} \left\{  (\Z_S^\top \Z_S)^{-1} \Z_S^\top y , \, \frac{\omega}{\gamma} ( \Z_S^\top \Z_S)^{-1}  \right\}, 
\end{align*}
where    $I$ denotes the identity matrix  and  $c_1 > 0, c_2 \geq 0, \kappa \geq 0, \gamma > 0$ are hyperparameters. 
Using a fractional likelihood with exponent $\alpha \in (0, 1)$, we find that 
\begin{equation}\label{eq:def.post}
\post(  S ) \propto   c_1^{ - |S|} p^{- c_2 |S|}  \left( 1 + \frac{\alpha}{\gamma} \right)^{- |S| / 2} 
(y^\top \oproj_S  y)^{ -(\alpha n + \kappa) / 2 } \ind_{\cS(m, d)}(S) , 
\end{equation}
where $\oproj_S$ (and $\proj_S$)  denotes the projection matrix,  
\begin{equation}\label{eq:def.Phi}
\proj_S = \Z_S (\Z_S^\top \Z_S)^{-1} \Z_S^\top, \quad \oproj_S = I - \proj_S. 
\end{equation}
The exponentiation of the posterior score function $\scorej$ defined in~\eqref{eq:post.modular} has exactly the same form as~\eqref{eq:def.post}. 
Thus, the analysis of the DAG selection and structure learning problem can be reduced to that of all possible nodewise variable selection problems. 

As in Definition~\ref{def:gj}, given a true model $S^* \subseteq [m]$ (which will be defined later),  we define a transition function $g \colon \cS(m, d) \rightarrow \cS(m, d)$ by 
\begin{equation}\label{eq:def.gS}
\begin{aligned}
g(S) = \left\{ \begin{array}{cc}
S, &   \text{ if } S = S^*, \\
 \argmax_{ S' \in \cN_{\mathrm{del}}^*(S)} \post(  S' ), \quad &  \text{ if } S^* \subset S, \\
  \argmax_{ S'\in \cN_{\mathrm{add}}^*(S) } \post(  S'), \quad &  \text{ if }  S^* \not\subseteq S, \, |S| < d \\
 \argmax_{S'\in \cN_{\mathrm{swap}}^*(S) } \post( S'), \quad &  \text{ if } S^* \not\subseteq S, \, |S| = d, 
\end{array}
\right.
\end{aligned}
\end{equation}
where 
\begin{equation*}
\begin{aligned}
& \cN_{\mathrm{add}}^*(S) =  \{   S \cup \{k\} \colon  k \in S^* \setminus S \},  \quad 
 \cN_{\mathrm{del}}^*(S) =  \{  S \setminus \{\ell\} \colon \ell \in S \setminus S^* \}, \\
& \cN_{\mathrm{swap}}^*(S) = \{  (S \cup \{k\}) \setminus \{\ell\} \colon  k \in S^* \setminus S, \;  \ell \in S \setminus S^*\}. 
\end{aligned}
\end{equation*} 
The goal of this supplement is to obtain  a lower bound  on $\post(g(S)) / \post(S)$ for all $S \in \cS(m, d) \setminus S^*$.  If the bound is sufficiently large, it then yields the strong consistency of the variable selection procedure. 
We will consider two sets of conditions. 
In the first scenario (see Section~\ref{sec:supp.var.consist}), we treat $\X_{[m]}$ as fixed and impose conditions on the noise part of $y$, while in the second (see Section~\ref{sec:supp.var.consist.covar}),  we treat both $\X_{[m]}$ and $y$
as random and, assuming $y = X_{m + 1}$, start by estimating the covariance matrix for the joint distribution of $(\sX_1, \dots, \sX_m, \sX_{m+1})$. 

\subsection{High-dimensional consistency: first scenario} \label{sec:supp.var.consist}
Suppose that $y$ can be written as 
\begin{equation}\label{eq:true.model}
y =   \ty +   \epsilon, \quad  \text{  where }    \ty = \Z_{S^*} \beta_{S^*}^*  
\end{equation} 
for some $S^* \subseteq [m]$,  $\beta_{S^*}^* \in \bbR^{|S^*|}$.  The vector $\ty$ represents the signal part of $y$.  
We make the following assumptions. The first one controls the multicollinearity of the data, and the second assumes the prior parameters are properly chosen. The last three assumptions concern  the behavior of the true model, where we treat the predictors (i.e., $\X_{[m]}$) as fixed and require the signal size  (true regression coefficient) be sufficiently large and the errors be ``well behaved''.  
%Note that $m$ may be equal to $p$, and we do not assume that $y$ is a column of $\X$. 

\begin{theorem}\label{th:var.sel}
Consider the distribution $\post$ given in~\eqref{eq:def.post} defined on $\cS(m, d)$ for some $m \leq p$ and $d \geq 1$,  where $\X  \in \bbR^{n \times p}$ and  $y$ is given by~\eqref{eq:true.model}. 
Suppose the following hold. 
\begin{enumerate}[label=(E1.\arabic*), ref=(E1.\arabic*)]
\item There exist $\vmin, \vmax > 0$  such that $y^\top y \leq n \vmax$ and \label{ass:eigen}
\begin{align*} 
 n \vmin \leq \lmin(\Z_S^\top \Z_S)  \leq \lmax(\Z_S^\top \Z_S) \leq n \vmax,
\end{align*}
for any $S \subseteq [m] $ with $|S| \leq 2d + 1$. 
\item The vector $\epsilon$ in~\eqref{eq:true.model} satisfies the following  for some  $ \omega^* \in [\vmin, \vmax]$ and $\rho \geq 2$. 
\label{ass:error}
\begin{gather*} 
   \min_{S \colon |S| \leq d}   \,     \epsilon^\top  \oproj_S  \epsilon      \geq    n \ses / 2,     \\ 
   \max_{S \colon |S| \leq d} \max_{j \notin S}  \,   \epsilon^\top  (\proj_{ S \cup \{j\} } - \proj_S )  \epsilon        \leq   \rho   \ses \log p.  
\end{gather*} 
\item Prior parameters satisfy that $\kappa \leq n$,    $c_1 \sqrt{ 1 + \alpha / \gamma} \in [1, p]$, and $c_2 \geq (\alpha + 1 )\rho + t$ for some constant $t > 0$.    \label{ass:prior}
\item The model space parameter $d$ and true model $S^*$ in~\eqref{eq:true.model} satisfy \label{ass:size} 
  $$ \left\{  \frac{ 4  \vmax^2   (\vmax - \vmin)^2 }{   \vmin^4  }   + 1  \right\} |S^*|  \leq  d.$$ 
\item There exists a  constant $\Cbeta \geq 8t / 3$ such that   \label{ass:beta}
\begin{align*}
\betamin^2 = \min \{ |\beta_j^*|^2 \colon \beta_j^* \neq 0 \} \geq    5 (  \Cbeta + 4 c_2  )     \frac{  \vmax^2 \log p }{\alpha \vmin^2 n}. 
\end{align*}  
\end{enumerate}
Then, for the function $g$ defined in~\eqref{eq:def.gS}, we have 
\begin{align*}
\frac{\post( g(S) )}{\post(S)} \geq p^t, \quad \quad \forall \, S \in \cS(m, d) \setminus \{S^*\}. 
\end{align*}
\end{theorem}
\begin{proof}
It follows from Lemmas~\ref{lm:overfit},~\ref{lm:underfit} and~\ref{lm:saturate} to be proved below, 
 each for one subcase in the definition of $g$ (except the case $S = S^*$). 
\end{proof}

%\subsection{Proof of Theorem~\ref{th:var.sel}}\label{sec:supp.consist.proof}
%We prove three lemmas below, each for one subcase in the definition of $g$ (except the case $S = S^*$). Theorem~\ref{th:var.sel} then follows. 

First, consider an overfitted model $S$ such that $S^* \subset S$.  
To show that the posterior probability increases when we remove some covariate in $S \setminus S^*$, we only need Assumption~\ref{ass:error}, which controls the behavior of $\epsilon$,  and Assumption~\ref{ass:prior}, which requires the penalty on model size be sufficiently large.  
 
\begin{lemma}\label{lm:overfit}
Suppose  Assumptions~\ref{ass:error} and~\ref{ass:prior} hold. 
For any $S \in \cS(m, d)$ such that $S^* \subset S$ and any $j \in S \setminus S^*$,  we have 
\begin{align*}
\frac{ \post(S) }{ \post( S \setminus \{j\}) } \leq  p^{-c_2 + (\alpha +  1) \rho}  \leq p^{-t}. 
\end{align*}
\end{lemma}
\begin{proof} 
Let $S' = S \setminus \{j\}$  for any $j \in S \setminus S^*$. Then,  it follows from~\eqref{eq:def.post} and the inequality $1 + x \leq e^x$ that 
\begin{align*}
\frac{\post(S)}{\post(S')}  
\leq \;&  c_1^{-1} (1 + \alpha / \gamma)^{-1/2} p^{-c_2} \exp\left\{ \frac{\alpha n + \kappa}{2}  \frac{y^\top  (\proj_S - \proj_{S'})   y}{ y^\top \oproj_S y }    \right\} \\
= \;& c_1^{-1} (1 + \alpha / \gamma)^{-1/2} p^{-c_2} \exp\left\{ \frac{\alpha n + \kappa}{2}  \frac{\epsilon^\top  (\proj_S - \proj_{S'})   \epsilon}{ \epsilon^\top \oproj_S \epsilon }    \right\}. 
\end{align*}
The second step follows from the observations 
$S^* \subseteq S$ and $S^* \subseteq S'$.  
A routine calculation using Assumptions~\ref{ass:error} and~\ref{ass:prior}  then yields the result. 
\end{proof}

For an underfitted model $S \in \cS(m, d)$ (underfitted means $S^* \setminus S \neq \emptyset$), bounding $\post(g(S)) / \post(S)$ is much more difficult due to the collinearity in the design matrix.   We need to use Corollary~\ref{lm:aux} proved in Section~\ref{sec:supp.aux}.  
To show that any underfitted model $S$ has a neighboring model with much larger posterior probability, we  consider two subcases according as $S$ is saturated (we say $S$ is saturated  if $|S| = d$).  Note that an underfitted and saturated model exists only if $d < m$.  
If $S$ is unsaturated, we can add some covariate in $S^* \setminus S$ so that the reduction in residual sum of squares would be significant.  If $S$ is saturated, we perform a swap move: add some covariate in $S^* \setminus S$ and remove another in $S \setminus S^*$. 

\begin{lemma}\label{lm:underfit}
Suppose  Assumptions~\ref{ass:eigen},~\ref{ass:error},~\ref{ass:prior} and~\ref{ass:beta} hold. 
Let $S \in \cS(m, d)$ be an underfitted model.  There exists some $k \in S^* \setminus S$ such that 
\begin{equation*} 
\frac{\post(S)}{   \post(S \cup \{k\})}  
\leq     p^{  - 3 c_2  -  C_\beta / 2 }.   
\end{equation*}
\end{lemma}
\begin{proof}%[Proof of Lemma~\ref{lm:underfit}]
Let $S' = S \cup \{k\}$  for some $k \in S^* \setminus S$. Then,
\begin{equation}\label{eq:under1}
\frac{\post(S)}{   \post(S')}  
\leq   c_1  (1 + \alpha / \gamma)^{1/2} p^{c_2} \exp\left\{ - \frac{\alpha n + \kappa}{2}  \frac{y^\top  (\proj_{S'} - \proj_S)   y}{ y^\top \oproj_S y }    \right\}. 
\end{equation}
By Assumption~\ref{ass:eigen}, $y^\top \oproj_S y \leq y^\top y \leq n \vmax$. 
By Corollary~\ref{lm:aux} and Assumption~\ref{ass:beta},  $k$ can be chosen such that 
\begin{equation}\label{eq:bd2}
\norm{ (\proj_{S'} - \proj_S)   \ty}_2^2 \geq  \frac{n \vmin^2}{\vmax} \betamin^2 \geq  5  (\Cbeta + 4 c_2) \frac{\vmax \log p}{\alpha}
\geq     (  \Cbeta^{1/2} +  4 c_2^{1/2})^2 \frac{\vmax \log p}{\alpha}. 
\end{equation}
By Assumptions~\ref{ass:error} and~\ref{ass:prior},  
\begin{equation}\label{eq:err.bd}
\norm{ (\proj_{S'} - \proj_S)   \epsilon }_2^2 \leq  \rho \ses \log p <   \frac{ c_2   \ses \log p}{\alpha}
\leq \frac{ c_2   \vmax \log p}{\alpha}.
\end{equation}
The reverse triangle inequality   then yields that 
\begin{align*}
\norm{ (\proj_{S'} - \proj_S)  y}_2^2 \geq  \left\{   \sqrt{  \Cbeta} + 3 \sqrt{c_2} \right\}^2    \frac{  \vmax \log p}{\alpha}  
\geq  (  \Cbeta + 9 c_2) \frac{  \vmax \log p}{\alpha}. 
\end{align*}
Thus, the exponent in~\eqref{eq:under1} can be bounded by 
\begin{align*}
\frac{\alpha n + \kappa}{2}  \frac{  y^\top  (\proj_{S'} - \proj_S)   y}{ y^\top \oproj_S y } 
 \geq       \frac{  \Cbeta + 9 c_2 }{2}  \log p . 
\end{align*}  
By Assumptions~\ref{ass:error} and~\ref{ass:prior}, $c_2 \geq 2$ and thus $9c_2 / 2 \geq 4c_2 + 1$. 
The proof is completed upon  recalling that $c_1  (1 + \alpha / \gamma)^{1/2} p^{c_2}  \leq p$. 
\end{proof}

\begin{lemma}\label{lm:saturate}
Suppose Assumptions~\ref{ass:eigen},~\ref{ass:error},~\ref{ass:prior},~\ref{ass:size} and~\ref{ass:beta}  hold.
Let $S \in \cS(m, d)$  be an underfitted model with $|S| = d$. 
There exist some $k \in S^* \setminus S$ and $j \in S \setminus S^*$ such that 
\begin{equation*} 
\frac{\post(S)}{   \post( (S \cup \{k\}) \setminus \{j\} )}   \leq  p^{ -3  \Cbeta /8 }. 
\end{equation*}
\end{lemma}
\begin{proof}%[Proof of Lemma~\ref{lm:saturate}]
Let $S' = ( S \cup \{k\} ) \setminus \{j\}$ for some $k \in S^* \setminus S$ and $j \in S \setminus S^*$. 
Then, 
\begin{equation*} 
\frac{\post(S)}{   \post(S')}  
\leq    \exp\left\{ - \frac{\alpha n + \kappa}{2}  \frac{y^\top  (\proj_{S'} - \proj_S)   y}{ y^\top \oproj_S y }    \right\}. 
\end{equation*}
By Corollary~\ref{lm:aux}, we can pick $k$ and $j$ such that 
\begin{align*}
\norm{ (\proj_{S \cup \{k\} } - \proj_S)  \ty }_2^2  \geq   \frac{n \vmin^2  \norm{ \beta^*_{S^* \setminus S} }_2^2 }{ \vmax |  S^* \setminus S |  }, 
\quad \norm{ (\proj_{S' \cup \{j\} } - \proj_{S'})  \ty }_2^2  \leq \frac{n \vmax (\vmax - \vmin)^2  \norm{ \beta^*_{S^* \setminus S} }_2^2 }{ \vmin^2 |  S \setminus S^* |  }. 
\end{align*}
By Assumption~\ref{ass:size}, 
\begin{align*}
\frac{\norm{ (\proj_{S \cup \{k\} } - \proj_S)  \ty }_2^2  } { \norm{ (\proj_{S' \cup \{j\} } - \proj_{S'})  \ty }_2^2 } \geq  \frac{ \vmin^4  |  S \setminus S^* | }{ \vmax^2   (\vmax - \vmin)^2  |  S^* \setminus S | }
\geq  \frac{ \vmin^4  (d - |S^*|)}{ \vmax^2   (\vmax - \vmin)^2  |  S^*  | } \geq 4. 
\end{align*}
Then, using~\eqref{eq:bd2},~\eqref{eq:err.bd} and triangle inequalities, we find that 
\begin{align*}
 \norm{  (\proj_{S \cup \{k\} } - \proj_S)   y}_2^2  \geq \;&  \left\{   \sqrt{  \Cbeta} + 3 \sqrt{c_2} \right\}^2    \frac{  \vmax \log p}{\alpha},  \\
\norm{ (\proj_{S' \cup \{j\} } - \proj_{S'})  y }_2^2  \leq \;&  \left\{  \frac{ \sqrt{  \Cbeta} }{2} + 3 \sqrt{c_2} \right\}^2    \frac{  \vmax \log p}{\alpha}. 
\end{align*}
Hence, 
\begin{align*}
\norm{  (\proj_{S'} - \proj_S)   y}_2^2 = \;&  \norm{  (\proj_{S \cup \{k\} } - \proj_S)   y}_2^2 
-  \norm{ (\proj_{S' \cup \{j\} } - \proj_{S'})  y }_2^2  
\geq     \frac{ 3 \Cbeta \vmax \log p}{4\alpha}. 
\end{align*}
The result then follows by a calculation similar to the proof of Lemma~\ref{lm:underfit}. 
\end{proof}

\subsection{High-dimensional consistency: second scenario}\label{sec:supp.var.consist.covar} 
We now re-derive the high-dimensional consistency result for the model described in Section~\ref{sec:supp.var.modl} by treating both $\X$ and $y$ as random and viewing linear regression as orthogonal projection. This result is useful when one wants to use random matrix theory to study DAG selection or structure learning problems. We will use the following notation.  
Given  a  matrix $M$, let $\OPnorm{M} = \sup_{  \norm{b}_2 = 1 } \norm{ Mb  }_2 $ denote its operator norm. 
Given a positive definite matrix $\Sigma$  and sets (or integers) $A, B$, we use $\Sigma_{A, B}$ to denote the submatrix of $\Sigma$ with rows indexed by $A$ and columns indexed by $B$, and, by an abuse of notation, we simply write 
$\Sigma_A = \Sigma_{A, A}$\footnote{This notation is only used for covariance matrices and only used in %Sections~\ref{sec:supp.var.consist.covar},~\ref{sec:supp.proof.var.sel2} and~\ref{sec:supp.aux}. 
Supplements~\ref{sec:supp.sel} and~\ref{sec:supp.proof.consist}. 
Note that its meaning is different from $\X_S$, which denotes  the submatrix of $\X$ with columns indexed by $S$.}.  
Further, we use
\begin{equation}\label{eq:partial}
\Sigma_{A, B \mid C} =  \Sigma_{A, B} - \Sigma_{A, C} (\Sigma_{C})^{-1} \Sigma_{C, B}
\end{equation}
to denote the ``partial covariance'' and write $\Sigma_{A \mid C} = \Sigma_{A, A \mid C}$.

The following theorem has the same conclusion as Theorem~\ref{th:var.sel}, but we replace Assumption~\ref{ass:error} in Theorem~\ref{th:var.sel} with $\zeta = O( \sqrt{  n^{-1} d \log p  } )$ where $\zeta$ measures the difference between $\Sigma^*$ and the empirical covariance matrix; all the other four assumptions are very similar.    
\begin{theorem}\label{th:var.sel2}
Consider the distribution $\post$ given in~\eqref{eq:def.post} defined on $\cS(m, d)$ for some $m \in [p - 1]$ and $d \geq 1$,  where $\X  \in \bbR^{n \times p}$, $y = X_{m + 1}$ and each row of $\X_{[m + 1]}$ is an i.i.d. copy of a random vector $\sX = (\sX_1, \dots, \sX_{m+1})$ with mean zero and covariance matrix $\Sigma^*$. 
Define 
\begin{equation*}
   \beta^*  = ( \Sigma^*_{[m]} )^{-1} \Sigma^*_{[m], m+1},   \quad S^* = \{ j \colon \beta^*_j \neq 0 \}, \quad \hat{\Sigma} = n^{-1} \X^\top \X. 
\end{equation*} 
Assume the following hold.  
\begin{enumerate}[label=(E2.\arabic*), ref=(E2.\arabic*)]
\item   There exist constants $\vmax \geq 1 \geq \vmin$  such that  $\vmin \leq \lmin(\Sigma^*) \leq \lmax(\Sigma^*)  \leq \vmax$ and,  for any $S \in \cS(m + 1, 2d +1 )$,   $  \vmin  \leq    \lmin( \hat{\Sigma}_S ) \leq   \lmax(  \hat{\Sigma}_S ) \leq  \vmax.$ \label{ad21}
\item There exists a constant $\Cerr > 0$ such that $\zeta^2 \leq \Cerr    n^{-1} d \log p  $ where \label{ad22}
\begin{equation*}
    \zeta = \max_{  S   \in \cS(m + 1, 2d )} \OPnorm{  \hat{\Sigma}_S  - \Sigma^*_S }.
\end{equation*}
\item Prior parameters satisfy that $\alpha \in (0, 1]$,  $\kappa \leq n$,    $c_1 \sqrt{ 1 + \alpha / \gamma} \in [1, p]$, and $$c_2 \geq \Cpen   \frac { d  \vmax^4 }{ \vmin^6 },$$ 
for some constant $\Cpen > 0$.  \label{ad23}
\item The model space parameter $d$ and true model size $|S^*|$ satisfy 
$$ \left\{  \frac{ 20  \vmax^2   (\vmax - \vmin)^2 }{   \vmin^4  }   + 1  \right\} |S^*|  \leq d. $$ \label{ad24}
\item There exists a constant $\Cbeta$ such that  \label{ad25}
\begin{align*}
\betamin^2 = \min \{ |\beta_j^*|^2  \colon \beta_j^* \neq 0 \} \geq   \Cbeta  \frac{ c_2  \vmax^2 \log p }{\alpha \vmin^2 n}. 
\end{align*}  
\end{enumerate}
Given any $t > 0$, we can choose sufficiently large $\Cpen$ and $\Cbeta$, which only depend on $t$ and $\Cerr$, such that  
\begin{align*}
\frac{\post( g(S) )}{\post(S)} \geq p^t, \quad \quad \forall \, S \in \cS(m, d) \setminus \{S^*\}. 
\end{align*}
\end{theorem}
%\begin{proof}
%See Section~\ref{sec:supp.proof.var.sel2}. 
%\end{proof}

\begin{remark}
There is no loss of generality in assuming $\vmax \geq 1 \geq \vmin$ since one can always scale $\Sigma^*$. 
By the modified Cholesky decomposition, there exist a strictly upper triangular matrix $B^*$ and a diagonal matrix $\Omega^*$ such that $\Sigma^* = (I - (B^*)^\top)^{-1} \Omega^* (I - B^*)^{-1}$.  
By the block matrix inversion formula, one can show that 
$\beta^* =  B^*_{[m], m+1}$. 
Further,   for any $j \in [m]$,  the block matrix inversion formula yields that 
\begin{equation}\label{eq:supp.cov1}
\Sigma^*_{j, m+1} = \Sigma^*_{j, [m]} ( \Sigma^*_{[m]})^{-1} \Sigma_{[m], m + 1}^*
 =  \Sigma^*_{j, [m]} \beta^* = \Sigma^*_{j, S^*} \beta^*_{S^*}, 
\end{equation} 
which will be useful in the proof of Theorem~\ref{th:var.sel2}. 
\end{remark}

%\subsection{Proof of Theorem~\ref{th:var.sel2}}\label{sec:supp.proof.var.sel2}
Before we prove Theorem~\ref{th:var.sel2}, we derive a lemma for controlling the error propagation via matrix operations, which is the main technical difference between the proofs of Theorem~\ref{th:var.sel} and Theorem~\ref{th:var.sel2}.
 Note that for any matrix $M$, we have $\OPnorm{M} \geq \max_{i, j} |M_{i j}|$, so Lemma~\ref{lm:err.prop} also yields a bound on the entrywise maximum  error. 
\begin{lemma}\label{lm:err.prop}
Let $\Sigma^*, \hat{\Sigma} \in \bbR^{(m+1) \times (m+1)}$ be positive definite matrices  that satisfy Assumption~\ref{ad21} in Theorem~\ref{th:var.sel2}.  
For any $S,  T \in \cS(m + 1, d)$ such that $S \cap T  = \emptyset$, 
\begin{align*}
\OPnorm{  \hat{\Sigma}_{ T \mid S}  - \Sigma^*_{T \mid S}} \leq \;& \zeta \vmax^2 / \vmin^2, \\
\OPnorm{  (\hat{\Sigma}_{ S})^{-1} \hat{\Sigma}_{S, T}  - ( \Sigma^*_{S})^{-1} \Sigma^*_{S, T} } \leq \;& 2 \zeta \vmax / \vmin^2, 
\end{align*}
where $\zeta = \max_{  S   \in \cS(m + 1, 2d )} \OPnorm{  \hat{\Sigma}_S  - \Sigma^*_S }.$
\end{lemma}

\begin{proof} 
Let $\Sigma = \hat{\Sigma}$ or $\Sigma^*$. 
Since $(\Sigma_{T \mid S})^{-1}$ is a principal submatrix of $(\Sigma_{S \cup T})^{-1}$,  
\begin{align*}
  \vmax^{-1}  \leq   \lmin( (\Sigma_{S \cup T})^{-1}  )  \leq  \lmin( (\Sigma_{T \mid S})^{-1} ) \leq \lmax( (\Sigma_{T \mid S})^{-1} ) \leq \lmax( (\Sigma_{S \cup T})^{-1}  ) \leq  \vmin^{-1}. 
\end{align*}
Given invertible matrices $A, B$ of the same dimension,  $A^{-1} - B^{-1} =  A^{-1}(B  - A) B^{-1}$. 
Hence, by the sub-multiplicative property of operator norm, 
\begin{align*} 
\OPnorm{  (\hat{\Sigma}_{ T \mid S})^{-1} -   (\Sigma^*_{T \mid S})^{-1} } \leq \OPnorm{  (\hat{\Sigma}_{S \cup T})^{-1} -   (\Sigma^*_{S \cup T})^{-1} } \leq \zeta / \vmin^2. 
\end{align*}
Applying the same argument again, we get $\OPnorm{  \hat{\Sigma}_{ T \mid S}  - \Sigma^*_{T \mid S}} \leq \zeta \vmax^2 / \vmin^2$.   

To prove the second inequality, we use the identity $AB - A'B'   =  A(B - B') + (A - A') B'$ for matrices $A, A', B, B'$.  Since $\hat{\Sigma}_{S, T} - \Sigma^*_{S,T} $ is a submatrix of $\hat{\Sigma}_{S \cup T} - \Sigma^*_{S \cup T} $, we have $\OPnorm{ \hat{\Sigma}_{S, T} - \Sigma^*_{S,T} } \leq \zeta$.  
It follows that 
\begin{align*}
\OPnorm{  (\hat{\Sigma}_{ S})^{-1} \hat{\Sigma}_{S, T}  - ( \Sigma^*_{S})^{-1} \Sigma^*_{S, T} } \leq \;&  \zeta / \vmin + \zeta \vmax/\vmin^2,  
\end{align*}
which yields the asserted bound since $\vmax \geq  \vmin$. 
\end{proof}

\begin{proof}[Proof of Theorem~\ref{th:var.sel2}]
%The proof is  similar to that of Theorem~\ref{th:var.sel}.  
Since $y = X_{m+1}$, we can write 
\begin{align*}
n^{-1} y^\top \oproj_S y =  \hat{\Sigma}_{m+1 \mid S},  \quad 
n^{-1} y^\top  ( \proj_{S \cup \{j\} } - \proj_S ) y =      ( \hat{\Sigma}_{j \mid S} )^{-1} (\hat{\Sigma}_{j, m+1 \mid S})^2, 
\end{align*}
for any $S \subset [m]$  and $j \in [m] \setminus S$.  As in the proof of Theorem~\ref{th:var.sel}, we split the proof into three cases. 

\paragraph*{Case 1: overfitted}  Suppose $S^* \subset S \in \cS(m, d)$ and let $S'  = S\setminus \{j\}$ for some $j \in S \setminus S^*$.  By Assumption~\ref{ad23}, we have 
\begin{align*} 
\frac{\post(S)}{ \post(S')} \leq p^{-c_2} \exp\left\{  
\frac{  n ( \hat{\Sigma}_{j \mid S'} )^{-1} (\hat{\Sigma}_{j, m+1 \mid S'})^2  }{ \hat{\Sigma}_{m+1 \mid S} } \right\}. 
\end{align*}
Assumption~\ref{ad21} implies that $\hat{\Sigma}_{m+1 \mid S} \geq \vmin$ and $( \hat{\Sigma}_{j \mid S'} )^{-1}  \leq \vmin^{-1}$.   
By~\eqref{eq:supp.cov1}, we have $\Sigma^*_{j, m+1 \mid S'}
 = \Sigma^*_{j, S^* \mid S'} \beta^*_{S^*} = 0$  since $S^* \subseteq S'$.  Apply Lemma~\ref{lm:err.prop} to get   $| \hat{\Sigma}_{j, m+1 \mid S'} | \leq \zeta \vmax^2 / \vmin^2.$ 
It follows by Assumption~\ref{ad22} that 
\begin{align*}
\frac{\post(S)}{ \post(S')} \leq p^{-c_2} \exp\left\{    \frac{n \zeta^2 \vmax^4}{\vmin^6} \right\}
\leq \exp \left\{  \left(- c_2 +   \frac{ \Cerr  d  \vmax^4 }{ \vmin^6 }\right) \log p \right\},  
\end{align*}
Clearly, $d \vmax^4 / \vmin^6 \geq 1$ since $\vmin \leq \vmax$ and we assume $\vmin \leq 1$. 
Hence, it suffices to choose $\Cpen \geq \Cerr + t$ so that $\post(S)/  \post(S') \leq p^{-t}$. 
 
\paragraph*{Case 2: underfitted and unsaturated} 
Let $S \in \cS(m, d - 1)$ be such that $S^* \setminus S \neq \emptyset$. 
By Assumptions~\ref{ad21} and \ref{ad23},  for any $k \in S^* \setminus S$, 
\begin{equation}\label{eq:underfit.random.matrix.tmp1}
\frac{\post(S)}{ \post(S \cup \{k\})} \leq p^{c_2 + 1} \exp\left\{ -\frac{\alpha n}{2} \frac{ (\hat{\Sigma}_{k, m+1 \mid S})^2 }{ \vmax^2 } \right\}. 
\end{equation}   
Define
\begin{equation}\label{eq:def.bs}
\hat{b} =  (\hat{\Sigma}_{S \cup S^*} )^{-1} \hat{\Sigma}_{S \cup S^*, m+1}, \quad  b^* = (\Sigma^*_{S \cup S^*})^{-1} \Sigma^*_{S \cup S^*, m+1}. 
\end{equation}
Applying the first part of Lemma~\ref{lm:lin.cov} to $\hat{\Sigma}_{ S \cup S^* \cup \{m + 1\}}$ with $T = S^* \setminus S, V = S$ and $l = |S \cup S^*|$,  we get 
\begin{align*}
  \sum_{k \in T} (\hat{\Sigma}_{k, m+1 \mid S})^2  \geq \lmin(
   \hat{\Sigma}_{  T \mid S})^2     \norm{ \hat{b}_{T}  }_2^2 \geq    \vmin^2  \norm{ \hat{b}_{T}  }_2^2. 
\end{align*} 
Observe that $b^*_{T} = \beta^*_{T}$; this can be proved by applying the block matrix inversion formula to $(\Sigma^*_{[m]})^{-1}$. 
Then,  by  Lemma~\ref{lm:err.prop},  for any $k \in S^* \setminus S$, we have 
$|\hat{b}_k -  \beta^*_k | \leq \OPnorm{ \hat{b} - b^*} \leq 2\zeta \vmax / \vmin^2$, and thus 
$$(\hat{b}_k)^2 \geq  \left(    |\beta^*_k|  -   2 \zeta \vmax / \vmin^2 \right)^2.$$  
By Assumptions~\ref{ad22},~\ref{ad23} and~\ref{ad25}, 
\begin{align*}
  |\beta^*_k|^2 \geq 
   \Cbeta  \frac{ c_2  \vmax^2 \log p }{\alpha \vmin^2 n} 
   \geq  \Cbeta \Cpen \frac{  d \vmax^6 \log p   }{\alpha \vmin^8 n} 
   \geq \frac{ \Cbeta \Cpen}{\Cerr}   \frac{    \zeta^2 \vmax^6}{ \vmin^8} 
   \geq \frac{ \Cbeta \Cpen}{\Cerr}  \left( \frac{    \zeta \vmax}{ \vmin^2}\right)^2. 
\end{align*}
%By Assumptions~\ref{ad22},~\ref{ad23} and~\ref{ad25}, one can quickly show that if  $\Cbeta$ is sufficiently large,   $(\hat{b}_k)^2 \geq (\beta^*_k  )^2 / 2 $ for all sufficiently large $n$. 
%It then follows from Assumption~\ref{ad22} that 
As long as  $\Cbeta \Cpen / \Cerr$ is sufficiently large, we have $(\hat{b}_k)^2 \geq (\beta^*_k  )^2 / 2 $. It  follows  that 
\begin{equation}\label{eq:underfit.random.matrix.tmp2}
  \sum_{k \in S^* \setminus S} (\hat{\Sigma}_{k, m+1 \mid S})^2  \geq \frac{\vmin^2}{2} \norm{ \beta^*_{S^* \setminus S}}_2^2. 
\end{equation} 
Choosing an optimal $k$, we get $(\hat{\Sigma}_{k, m+1 \mid S})^2 \geq \vmin^2 \betamin^2 / 2$. Using~\eqref{eq:underfit.random.matrix.tmp1}  and Assumption~\ref{ad25}, we find that   
\begin{align*}
\frac{\post(S)}{ \post(S \cup \{k\})} \leq    p^{- ( \Cbeta / 4 - 1) c_2 + 1}. 
\end{align*}   
Consequently, we only need to choose sufficiently large $\Cbeta = \Cbeta(t, \Cerr, \Cpen)$ so that $\post(S)/\post(S') \leq p^{-t}$. 

\paragraph*{Case 3: underfitted and saturated} 
Let $S \subseteq [m]$ be such that $S^* \setminus S \neq \emptyset$ and $|S| = d$.  
For any $k \in S^* \setminus S$ and $j \in S \setminus S^*$, by Assumption~\ref{ad21}, 
\begin{align*}
& \frac{\post(S)}{ \post(S')} \leq  \exp\left\{ -\frac{\alpha n (R_1 - R_2)}{2\vmax}    \right\},  \\
\text{where } \; &R_1 =  (\hat{\Sigma}_{k \mid S} )^{-1} (\hat{\Sigma}_{k, m+1 \mid S})^2, \quad 
R_2 =  (\hat{\Sigma}_{j \mid S'} )^{-1} (\hat{\Sigma}_{j, m+1 \mid S'})^2, 
\end{align*}
and $S' = (S \cup \{k\}) \setminus \{j\}$. By~\eqref{eq:underfit.random.matrix.tmp2}, there exists some $k \in S^* \setminus S$ such that 
\begin{equation}\label{eq:tmp1}
R_1 \geq  \frac{\vmin^2 }{2 \vmax} \frac{\norm{\beta^*_{S^* \setminus S} }_2^2}{|S^* \setminus S |} \geq \frac{\vmin^2 }{2 \vmax} \betamin^2 \geq \frac{ \Cbeta \Cpen}{2 \Cerr}   \frac{    \zeta^2 \vmax^5}{ \vmin^6}. 
\end{equation} 
Let $\hat{b}$ and $b^*$ be as defined in~\eqref{eq:def.bs}.  
To bound $R_2$, we apply the second part of Lemma~\ref{lm:lin.cov} to $\hat{\Sigma}_{S \cup S^* \cup \{m + 1\} }$.  Letting $U = S \cup \{k\}$, $V = U \cap S^*$,  $T = (U \setminus S^*) \cup (S^* \setminus U)$  and $W = U \setminus S^*$, 
we get that  
\begin{equation}\label{eq:tmp2}
\sum_{j \in W} \left\{ (\hat{\Sigma}_{j \mid U \setminus \{j\} } )^{-1} \hat{\Sigma}_{j, m+1 \mid U \setminus \{j\} }\right\}^2 
\leq   \norm{ (\hat{\Sigma}_U)^{-1} \hat{\Sigma}_{U, T}    \hat{b}_T }_2^2. 
\end{equation} 
Since $S^* \subset S \cup S^*$, we have $b^* = \beta^*_{S \cup S^*}$, which implies $b^*_T = \beta^*_T$.  
By  Lemma~\ref{lm:err.prop} and Assumptions~\ref{ad22},~\ref{ad23} and~\ref{ad25},  for any $i \in S^* \setminus U$, we have  $(\hat{b}_i)^2 \leq 2 (\beta^*_i  )^2$ as long as $\Cbeta \Cpen / \Cerr$ is sufficiently large.  
Hence, applying  Lemma~\ref{lm:err.prop}  again, we obtain 
\begin{align*}
    \norm{\hat{b}_{S^* \setminus U} }_2^2 \leq  2 \norm{\beta^*_{S^* \setminus U}}_2^2,  \quad  \quad 
    \norm{\hat{b}_{U \setminus S^*} }_2^2 \leq  2 |U \setminus S^*| \zeta^2 \vmax^2 / \vmin^4.  
\end{align*} 
Since $U \setminus S^*$ and $S^* \setminus U$ are disjoint, it follows that 
\begin{align*}
    \norm{ (\hat{\Sigma}_U)^{-1} \hat{\Sigma}_{U, T}    \hat{b}_T }_2^2 
=\;&  \norm{ (\hat{\Sigma}_U)^{-1} \hat{\Sigma}_{U, U \setminus S^*}    \hat{b}_{U \setminus S^*} +  (\hat{\Sigma}_U)^{-1} \hat{\Sigma}_{U, S^* \setminus U}    \hat{b}_{S^* \setminus U} }_2^2,  \\
\leq \;&  2 \norm{ \hat{b}_{U \setminus S^*} }^2 + 2  \norm{ (\hat{\Sigma}_U)^{-1} \hat{\Sigma}_{U, S^* \setminus U}    \hat{b}_{S^* \setminus U}  }^2 \\
\leq \;&  4 |U \setminus S^*| \zeta^2 \vmax^2 / \vmin^4 + 2 \norm{ (\hat{\Sigma}_U)^{-1} \hat{\Sigma}_{U, S^* \setminus U}  }_{\mathrm{op}}^2 
\norm{ \hat{b}_{S^* \setminus U}  }_2^2 \\
\leq \;&  4 |U \setminus S^*| \frac { \zeta^2 \vmax^2 }{  \vmin^4 }+ 
4  \frac{  (\vmax - \vmin)^2 }{\vmin^2}  \norm{\beta^*_{S^* \setminus U}}_2^2,
\end{align*}
where in the last step we have used Lemma~\ref{lemma.a3} to bound the operator norm. Choosing an optimal $j$ in~\eqref{eq:tmp2}, we get 
\begin{align*}
    R_2 =\;& (\hat{\Sigma}_{j \mid S'} ) \left\{ (\hat{\Sigma}_{j \mid S'} )^{-1} (\hat{\Sigma}_{j, m+1 \mid S'}) \right\}^2  
    \leq   4 \vmax \left(   \frac { \zeta^2 \vmax^2 }{  \vmin^4 } + 
   \frac{  (\vmax - \vmin)^2 }{\vmin^2 |U \setminus S^*|}  \norm{\beta^*_{S^* \setminus U}}_2^2  \right) \\
   \leq \;& 8 \max \left\{  \frac { \zeta^2 \vmax^3 }{  \vmin^4 }, \;   \frac{ \vmax  (\vmax - \vmin)^2 }{\vmin^2 |U \setminus S^*|}  \norm{\beta^*_{S^* \setminus U}}_2^2  \right\}. 
\end{align*}
If the first term in the maximum is larger, by~\eqref{eq:tmp1}, we only need to let $\Cbeta \Cpen / \Cerr$ be sufficiently large so that $R_1 / R_2 \geq 5/4 $.
If the second term is larger (note  $S^* \setminus U = S^* \setminus S$),   one can use~\eqref{eq:tmp1} and Assumption~\ref{ad24} to show that $R_1 / R_2 \geq 5/4 $. 
To summarize, we can choose $k \in S^* \setminus S$ and $j \in S \setminus S^*$ such that
\begin{align*}
    \frac{\post(S)}{ \post( (S \cup \{k\}) \setminus\{j\} )} \leq  \exp\left\{ -\frac{\alpha n R_1 }{8\vmax}    \right\} \leq  \exp\left\{ -\frac{\alpha n \vmin^2 }{16 \vmax^2} \betamin^2    \right\}. 
\end{align*} 
A routine  calculation then completes the proof. 
\end{proof}

\subsection{Auxiliary lemmas}\label{sec:supp.aux}
In this section, we prove some useful  results for bounding the change in  residual sum of squares for  optimal addition or deletion moves.  
We first prove a general linear algebra result in Lemma~\ref{lm:lin.cov} and then use it to obtain Corollary~\ref{lm:aux}, which is similar to Lemma 8 of~\citet{yang2016computational}~\citep[cf.][Lemma 1]{an2008stepwise}.  
Note that Lemma 8 of~\citet{yang2016computational}  requires an irrepresentability assumption and involves  the constant $\max_{S \in \cS(m, d)}  \OPnorm{ (\hat{\Sigma}_S)^{-1} \hat{\Sigma}_{S,  S^* \setminus S }}.$ 
%$\max_{S \in \cS(m, d)}  \OPnorm{ (\Z_S^\top \Z_S )^{-1} \Z^\top_S \Z_{ S^* \setminus S }}.$ 
But we directly bound this constant using Lemma~\ref{lemma.a3}.  

\begin{lemma}\label{lm:lin.cov}
Let $\Sigma \in \bbR^{(l+1)  \times (l+1)}$ be positive definite  for some $l \in \bbN$  and  
define $\beta = (\Sigma_{[l]})^{-1} \Sigma_{[l], l+1}$. 
Let $V, U$ be nonempty sets such that $V \subset U \subseteq [l]$, and let $T = [l] \setminus V$ and $W = U \setminus V$.  Then, 
\begin{align*}
\sum_{k \in T} (\Sigma_{k, l+1 \mid V} )^2 \geq \;&   (\lmin( \Sigma_{ T \mid V } ))^2  \norm{ \beta_{T}   }_2^2, \\
\sum_{j \in W}  \left\{  (\Sigma_{j \mid U \setminus \{j \} })^{-1}   \Sigma_{j, l+1 \mid U \setminus \{j \} } \right\}^2 \leq \;&  \norm{(\Sigma_{U})^{-1} \Sigma_{U, T}\beta_T }_2^2. 
% \norm{ (\Sigma_{U})^{-1} \Sigma_{U, T} }^2_{\mathrm{op}}  \norm{\beta_T}_2^2.
\end{align*}
\end{lemma}
\begin{proof} 

Using $\Sigma_{k, l+1} = \Sigma_{k, [l]} \beta$ for any $k \in [l]$, we find that 
\begin{align*}
\Sigma_{k, l+1 \mid V}   = \Sigma_{k, [l] \mid V} \beta =  \Sigma_{k, 1 \mid V} \beta_1 + \cdots + \Sigma_{k,l \mid V} \beta_l. 
\end{align*}
Observe that $ \Sigma_{k, i \mid V} =0$ if $i \in V$. Hence, $\Sigma_{k, l+1 \mid V} =  \Sigma_{k, T \mid V} \beta_T$. It follows that 
\begin{align*}
\sum_{k \in T}  (\Sigma_{k, l+1 \mid V} )^2  = \norm{ \Sigma_{T \mid V} \beta_T }_2^2 
\geq  (\lmin(\Sigma_{T \mid V}))^2 \norm{ \beta_T }_2^2. 
\end{align*} 
%The first asserted inequality then follows by choosing an optimal $k \in T$. 

For the second claim, without loss of generality, assume $U = \{1, 2, \dots, |U|\}$ and  define  
$$b  = (\Sigma_U)^{-1} \Sigma_{U, l+1} = (\Sigma_U)^{-1} \Sigma_{U, [l]} \beta.$$ 
We have $b_j =  (\Sigma_{j \mid U \setminus \{j \} })^{-1}   \Sigma_{j, l+1 \mid U \setminus \{j \} }$ for each $j \in U$; this can be proved by 
applying the block matrix inversion formula to $(\Sigma_{U})^{-1}$ with blocks $U \setminus \{j\}$ and $\{j\}$.  Note that we can write $b = b^{(1)} + b^{(2)}$, where 
\begin{align*}
b^{(1)} =  (\Sigma_U)^{-1}   \Sigma_{U, V} \beta_V, \quad 
b^{(2)} =  (\Sigma_U)^{-1}  \Sigma_{U, T} \beta_T,
\end{align*}
and  $b^{(1)}$ satisfies $b^{(1)}_W= 0$.  Therefore, for any $j \in W$,  $b_j = b^{(2)}_j$. Summing over $j \in W$, we get 
\begin{align*}
\sum_{j \in W} \left\{ (\Sigma_{j \mid U \setminus \{j \} })^{-1}   \Sigma_{j, l+1 \mid U \setminus \{j \} }\right\}^2 = \sum_{j \in W}  (b^{(2)}_j)^2 \leq \norm{ (\Sigma_U)^{-1}  \Sigma_{U, T} \beta_T}_2^2, 
\end{align*}
which yields the asserted inequality. 
%To conclude the proof, choose an optimal $j \in W$. 
\end{proof}

\begin{corollary}\label{lm:aux} 
Let $\ty$ be as defined in~\eqref{eq:true.model} for some $S^* \in \cS(m, d)$. Suppose 
\begin{align*}
n \vmin \leq  \min_{S \in \cS(m, 2d + 1)} \lmin(\Z_S^\top \Z_S ) \leq  \max_{S \in \cS(m, 2d + 1)} \lmax(\Z_S^\top \Z_S ) \leq n \vmax, 
\end{align*}
for some $\vmin, \vmax \in (0, \infty)$.  Let $\proj_S$ be as given in~\eqref{eq:def.Phi}.  For any $S \in \cS(m, d)$,  
\begin{align*} 
\sum_{k \in S^* \setminus S} \norm{ (\proj_{ S \cup \{k\}}  - \proj_S) \ty  }_2^2   \geq \;&   \frac{n  \vmin^2}{ \vmax}    \norm{ \beta^*_{S^* \setminus S} }_2^2, \\
\sum_{j \in S \setminus S^*}  \norm{ (\proj_{ S }  - \proj_{S \setminus \{j\} }) \ty  }_2^2  \leq \;& \frac{n \vmax  (\vmax - \vmin)^2   }{ \vmin^2    } \norm{ \beta^*_{S^* \setminus S} }_2^2. 
\end{align*} 
\end{corollary}

\begin{proof} 
Let $\hat{\Sigma} = n^{-1} \X^\top \X$. Using $y_0  = \X_{S^*} \beta^*_{S^*}$, the formula 
\begin{equation}\label{eq:proj.diff}
\proj_{S  \cup \{k\} } - \proj_{S } =   \oproj_{S  } X_k  ( X_k^\top  \oproj_{S }  X_k  )^{-1} X_k^\top \oproj_{S }, 
\end{equation}
and the eigenvalue assumption,  we find that 
\begin{align*}
n^{-1} \norm{ (\proj_{ S \cup \{k\}}  - \proj_S) \ty  }_2^2 =\;& ( \hat{\Sigma}_{k \mid S} )^{-1} (\hat{\Sigma}_{k,  S^* \mid S} \beta^*_{S^*})^2 \geq \vmax^{-1}  (\hat{\Sigma}_{k,  S^* \mid S} \beta^*_{S^*})^2, \\
n^{-1} \norm{ (\proj_{ S }  - \proj_{S \setminus \{j\} }) \ty  }_2^2 =\;& 
 ( \hat{\Sigma}_{j \mid S \setminus \{j\}} )^{-1} (\hat{\Sigma}_{j,  S^* \mid S \setminus \{j\} } \beta^*_{S^*})^2
\leq \vmax \left\{ ( \hat{\Sigma}_{j \mid S \setminus \{j\}} )^{-1} \hat{\Sigma}_{j,  S^* \mid S \setminus \{j\} } \beta^*_{S^*} \right\}^2. 
\end{align*}
Let $S' = S^* \setminus S$. 
By an argument analogous to the proof of Lemma~\ref{lm:lin.cov}, one can show that 
\begin{align*}
\sum_{k \in S^* \setminus S}  (\hat{\Sigma}_{k,  S^* \mid S} \beta^*_{S^*})^2  
=\;& \norm{ \hat{\Sigma}_{k,  S'  \mid S} \beta^*_{S' }   }_2^2,  \\
\sum_{j \in S \setminus S^*} \left\{ ( \hat{\Sigma}_{j \mid S \setminus \{j\}} )^{-1} \hat{\Sigma}_{j,  S^* \mid S \setminus \{j\} } \beta^*_{S^*} \right\}^2 
\leq \;& \norm{ (\hat{\Sigma}_{S})^{-1} \hat{\Sigma}_{S,  S'} \beta^*_{S'}  }_2^2.
\end{align*}
The asserted inequalities follow from the eigenvalue assumption and Lemma~\ref{lemma.a3} below. 
\end{proof} 

\begin{lemma}\label{lemma.a3}
Let $A = [ A_1  \; A_2  ]$  be an $n \times k$ matrix for some $k \leq n$,  $\lmax( A^\top A) = \nu_{\rm{max}}$ and $\lmin( A^\top A) =\nu_{\rm{min}}$.
Then, $\OPnorm{A_1^\top A_2} \leq \nu_{\rm{max}} - \nu_{\rm{min}}$. 
\end{lemma}
\begin{proof}
Suppose the dimension of $A_i$ is $n \times k_i$ for $i = 1, 2$. 
By the definition of operator norm,
\begin{align*}
\OPnorm{A_1^\top A_2} =\;& \max_{b_2  \in \bbR^{k_2} \colon \norm{b_2} = 1 } \norm{ A_1^\top A_2  b_2} \\ 
=\;&      \max\left\{   b_1^\top A_1^\top A_2  b_2 \colon   b_1 \in \bbR^{k_1}, b_2 \in \bbR^{k_2},  \norm{b_1} = \norm{b_2}  = 1  \right\}. 
\end{align*}
Since $\OPnorm{A_1} \leq \OPnorm{A} = \sqrt{ \nu_{\rm{max}} }$ and $\smin(A) = \sqrt{   \nu_{\rm{min} }}$, we have 
\begin{align*}
2 b_1^\top A_1^\top    A_2 b_2  =\;& \norm{  A_1 b_1 }^2_2 + \norm{  A_2 b_2 }^2_2 -   \norm{   (A_1b_1 -  A_2 b_2 )}^2_2  \\
\leq \;&   \nu_{\rm{max}} \norm{ b_1 }^2_2   +\nu_{\rm{max}}  \norm{ b_2 }^2_2   -  \nu_{\rm{min}} (\norm{b_1}_2^2 + \norm{b_2}_2^2). 
\end{align*}
Hence, if $\norm{b_1} = \norm{b_2} = 1$,  $  b_1^\top A_1^\top    A_2 b_2  \leq  \nu_{\rm{max}} - \nu_{\rm{min}}$. 
A similar argument  yields that $  b_1^\top A_1^\top    A_2 b_2  \geq    \nu_{\rm{min}} - \nu_{\rm{max}}$, 
which completes the proof. 
\end{proof}

\newpage  
\section{Proofs for Section~\ref{sec:high} }  \label{sec:supp.proof.consist} 
\setallcounters

\subsection{Empirical Bayes Gaussian DAG model} \label{sec:supp.emp} 
Let $\X_{(i)}$ denote the $i$-th row of the data matrix $\X$. 
We model the conditional distribution of $\X$ given $G$ by  
\begin{equation}\label{eq:model0}
\begin{aligned}
    \X_{(1)}, \dots, \X_{(n)} \mid B, \Omega  \overset{\rm{i.i.d.}}{\sim} \;& \NM_p(0, \Sigma(B, \Omega)),  \\
      \Sigma(B, \Omega) = \;&  (I  - B^\top)^{-1} \Omega  (I - B)^{-1}, \\ 
    (B, \Omega) \mid G \sim \;&  \pi_0(B, \Omega \mid G),   \quad \forall \, (B, \Omega) \in \chol(G ). 
\end{aligned}
\end{equation}  
%where $\pi_0(B, \Omega \mid G)$ is our empirical prior with support $\chol(G)$. 
Since a linear transformation of a normal random vector is still normal, the conditional distribution  of $X$ given $(B, \Omega)$ can also be expressed by the SEM, 
\begin{equation*}\label{eq:sem.original}
X_j =   \sum_{ i \neq j }  B_{i j} X_i + \varepsilon_j , \quad \varepsilon_j \sim \NM_n(0, \, \omega_j I), 
\end{equation*}
for $j = 1, \dots, p$, where $\varepsilon_1, \dots, \varepsilon_p$ are independent error vectors.  
We use the model~\eqref{eq:model0}  for two reasons. 
%Let $\mathcal{Q}$ denote the collection of all possible distributions of $\X_{(1)}$ that can be modeled by~\eqref{eq:model0} for some $(B, \Omega) \in \chol(G)$. 
%First, by Lemma~\ref{lm:chol.unique}, if $\Sigma$ is positive definite and $\NM_p(0, \Sigma)$ is Markovian w.r.t. $G$, then  $\NM_p(0, \Sigma) \in \mathcal{Q}$, where $\mathcal{Q}$ denotes the collection of all possible distributions of $\X_{(1)}$ that can be modeled by~\eqref{eq:model0} for some $(B, \Omega) \in \chol(G)$. 
First, by Lemma~\ref{lm:chol.unique}, if $\Sigma$ is positive definite and $\NM_p(0, \Sigma)$ is Markovian w.r.t. $G$, then there exists a unique pair $(B, \Omega) \in \chol(G)$ such that $\Sigma = \Sigma(B, \Omega)$.  
Second, by  Lemma~\ref{lm:chol.perfect}, if the edge weights (entries $B_{ij}$ for $i\rightarrow j \in G$) are sampled from an absolutely continuous distribution (which is true for our empirical prior), the resulting distribution $\NM_p(0, \, \Sigma)$ is almost surely perfectly Markovian w.r.t. $G$.  
There is little loss of generality in assuming that $\sX$ has mean zero,  since the normality implies that any CI statement about $\sX_1, \dots, \sX_p$ can be determined by using  $\Sigma$ alone.  
  
%Hence, the support of our prior  for $(B, \Omega)$ given $G$ can be thought of as the collection of all $p$-variate non-degenerate mean-zero normal distributions that are Markovian w.r.t. $G$, and our prior assigns probability $1$ to the normal distributions that are perfectly Markovian w.r.t. $G$. 
%%% degenerate claim? 
Recalling $\Omega = \diag(\omega_1, \dots, \omega_p)$ and using the notation $\beta_j(G)$ defined in Section~\ref{sec:cpdag.model}, we can express our empirical prior $\pi_0(B, \Omega \mid G)$  by  
\begin{equation}
\begin{aligned}\label{eq:model1}
 & \pi_0(B, \Omega \mid \Pa_1(G) = S_1, \dots,  \Pa_p(G) = S_p) \\
\propto  \;& \prod_{j=1}^p  \omega_j^{- \kappa /2 - 1} \NM_{|S_j|} \left( \beta_j(G);  \; (\X_{S_j}^\top \X_{S_j})^{-1} \X_{S_j}^\top X_j , \, \frac{\omega_j}{\gamma} (\X_{S_j}^\top \X_{S_j})^{-1}  \right),  
\end{aligned}
\end{equation}
where $\NM_q(b;  \mu, \Sigma)$ denotes the density function of $\NM_q(\mu, \Sigma)$ evaluated at $b$. Note that since $(B, \Omega)$ takes value in $\chol(G)$,  $\beta_j(G)$ contains all nonzero regression coefficients for the response vector $X_j$.   The prior mean for $\beta_j(G)$ is simply the ordinary-least-squares estimator. 
Let $L( B, \Omega )$ denote the likelihood function (the dependency on $\X$ is omitted). 
Since the empirical prior relies on the observed data, to counteract its effect  we use a fractional likelihood with exponent $\alpha \in (0, 1)$, which yields the conditional posterior distribution, 
\begin{equation*}
\post ( B,  \Omega \mid G)  \propto  \pi_0  ( B, \Omega \mid G)  L(B, \Omega)^\alpha 
=   \frac{    \pi_0( B, \Omega \mid G)  }{L(B, \Omega)^{1 - \alpha}}  L(B, \Omega) . 
\end{equation*}
This shows that the effective prior distribution for $(B, \Omega) \mid G$ is $\pi_0 (B, \Omega \mid G ) / L^{1 - \alpha}(B, \Omega)$; see~\citet{martin2017empirical, lee2019minimax} for more discussion. 
By a routine calculation using the conjugacy of normal-inverse-gamma prior, we obtain the fractional marginal  likelihood function $f_\alpha$ given in the main text, 
\begin{equation}\label{eq:f.alpha}
\begin{aligned}
f_\alpha(G)  = \;& \int  \pi_0  ( B, \Omega \mid G)  L(B, \Omega)^\alpha  d (B, \Omega) \\
=\;&  \left(   1+ \frac{\alpha}{\gamma} \right)^{- |G|/2} \prod_{j=1}^p ( X_j^\top   \oproj_{\Pa_j} X_j )^{ -(\alpha n + \kappa)/2}, 
\end{aligned}
\end{equation} 
where we recall $\oproj_S$ is defined by 
\begin{align*}
    \proj_S =   \X_S (\X_S^\top \X_S)^{-1} \X_S^\top, \quad \oproj_S =  I - \proj_S, \quad \quad \forall \, S \subseteq [p]. 
\end{align*}

We can use $f_\alpha$ to derive posterior distributions of DAGs or equivalence classes. For example, given a prior distribution of DAGs,   $\tilde{\pi}_0 (G)$, we can  calculate the posterior by 
$$\tilde{\pi}_n(G) =  \int \tilde{\pi}_n (G, B, \Omega) d (B, \Omega) 
\propto  \int  \tilde{\pi}_0 (G)  \pi_0(B, \Omega \mid G) L(B, \Omega)^\alpha d (B, \Omega)  = \tilde{\pi}_0(G) f_\alpha(G).$$
%Now suppose that we want to select a DAG model from a set $\cH \subseteq \dag$. A standard choice of the prior distribution for $G$ is given by  
%\begin{equation*}  
%\pi_0^{\cH}(  G  )  \propto  \ind_{ \cH } (G) \left( c_1p^{c_2} \right)^{-|G|},
%\end{equation*}
%Then, the posterior probability of a DAG $G$  can be computed by 
%\begin{align*}
%    \post^{\cH}(G) = \;& \int \post(G, B, \Omega) d (B, \Omega) \\
%    \propto \;&  \int  \pi_0^{\cH}(G) \pi_0  ( B, \Omega \mid G)  L(B, \Omega)^\alpha d (B, \Omega) \\
%     = \;& \pi_0^{\cH}(G) f_\alpha (G). 
%\end{align*}
%Taking the logarithm of $\post^{\cH}(G)$ with $\cH = \dag$, we obtain the expression for $\score(G)$ given in Section~\ref{sec:cpdag.model}. The posterior distribution $\post^\sigma$ given in~\eqref{eq:dag.sigma} is obtained by $\cH = \dag^\sigma(\din, \dout)$. 

\subsection{Proof of Lemma~\ref{lm:markov.equiv}}\label{sec:proof.like}
\begin{proof}
This is equivalent to proving that the marginal fractional likelihood  defined in~\eqref{eq:f.alpha} is the same for Markov equivalent DAGs (since they have the same skeleton and thus the same number of edges). 
By Lemma~\ref{lm:equiv2}, if two DAGs are Markov equivalent, then there exists a sequence of covered edge reversals that can transform one to the other. 
So it suffices to show that any covered edge reversal does not change the marginal likelihood. 
Let $G, G'$ be two DAGs that differ by a covered edge reversal. Thus, there exist $i\neq j$ such that $i \rightarrow j \in G$, $j \rightarrow i \in G'$, $\Pa_i(G) = \Pa_j(G) \setminus \{i\}$, and all the other edges are exactly the same in the two DAGs. 
%Let $\sigma$ be a topological ordering $G$ such that $\sigma(k) = i, \sigma(k + 1) = j$ for some $k \in [p]$; such an ordering always exists since $\Pa_i(G) = \Pa_j(G) \setminus \{i\}$. Let $\sigma' = (\sigma(1), \dots, \sigma(k-1), j, i, \sigma(k+1), \dots, \sigma(p))$ be a topological ordering of $G'$.
By~\eqref{eq:post.modular}, for $S = \Pa_i(G)$ we have 
\begin{align*}
    \frac{ \exp ( \score(G) ) }{\exp ( \score(G') ) } = \left( \frac{ X_i^\top  \oproj_S X_i \,  X_j^\top  \oproj_{S \cup \{i\} } X_j  }{ X_i^\top  \oproj_{S \cup \{j\}} X_i  \, X_j^\top  \oproj_{S  } X_j  } \right)^{-(\alpha n + \kappa) / 2}. 
\end{align*}
It then follows from~\eqref{eq:proj.diff} that 
\begin{align*}
   ( X_i^\top  \oproj_S X_i ) (  X_j^\top  \oproj_{S \cup \{i\} } X_j )
= \;& ( X_i^\top  \oproj_S X_i )  X_j^\top \left(  \oproj_{S } - \frac{ \oproj_S X_i X_i^\top \oproj_S }{ X_i^\top \oproj_S X_i } \right) X_j  \\
= \;& ( X_i^\top  \oproj_S X_i )(  X_j^\top\oproj_{S } X_j ) - ( X_j^\top \oproj_S X_i)^2.  
\end{align*}
By symmetry, we conclude that $ \score(G) = \score(G')$. 
\end{proof}

\subsection{Proof of Theorem~\ref{th:sel0} }\label{sec:proof.dag.sel}

\begin{proof}[Proof of Theorem~\ref{th:sel0}(i)]
We  will use Theorem~\ref{th:var.sel} to prove Theorem~\ref{th:sel0}. 
The main challenge is   to show the assumptions of Theorem~\ref{th:var.sel}  are satisfied for all the $p! \, p$  variable selection problems (there are $p!$ orderings and each corresponds to $p$ variable selection problems).  

For every $\sigma \in \bbS^p$, we have an SEM representation for the distribution $\NM_p(0, \Sigma^*)$ given by 
\begin{align*}
\sX  =  (B^*_\sigma)^\top  \sX  + \mathsf{e}_\sigma, \quad \mathsf{e}_\sigma \sim \NM_p \left(   0 ,  \,   \Omega^*_\sigma   \right). 
\end{align*}
where $(B^*_\sigma, \Omega^*_\sigma)$ is the modified Cholesky decomposition of $\Sigma^*$ given in Definition~\ref{def:imap}. 
Denote the diagonal elements of $\Omega^*_\sigma$ by  $\omega^*_{\sigma, 1}, \dots, \omega^*_{\sigma, p}$.
Using the data matrix, we can rewrite the SEM model as 
\begin{equation}\label{eq:true.model.perm}
X_j =   \sum_{ i \neq j }  (B^*_\sigma)_{i j} X_i + \varepsilon_{\sigma, j} , \quad \varepsilon_{\sigma, j} \sim \NM_n(0, \, \omega^*_{\sigma, j} I). 
\end{equation}
for $j = 1, \dots, p$. 
Note that the error vector $\epsilon_{\sigma, j}$ depends on the permutation $\sigma$. Define the standardized error vector by 
\begin{equation}\label{eq:all.error}
z_{\sigma, j} = (\omega^*_{\sigma, j})^{-1/2} \varepsilon_{\sigma, j}, \quad 
\cZ = \left\{ z_{\sigma, j} \colon  \sigma \in \bbS^p, \, j \in [p] \right\}. 
\end{equation}
Let $\beta^*_{\sigma, j} = (B^*_\sigma)_{\Pa_j(G^*_\sigma), j}$ be the subvector of the $j$-th column of $(B^*_\sigma)$ with entries indexed by $S^*_{\sigma, j} = \Pa_j(G^*_\sigma)$. 
As observed in~\citet[Section 7.4.1]{van2013ell}, $\beta^*_{\sigma, j}$ and $\varepsilon_{\sigma, j}$ only depend  on the set $S^*_{\sigma, j}$; see also~\citet[Proposition 8.5]{aragam2015learning}. 
Since the maximum degree of $G^*_\sigma$ is bounded by $d^*$, the number of possible parent sets for any node is at most $p^{d^*}$ and thus 
\begin{equation}\label{eq:Z.bound}
    |\cZ| \leq p \cdot p^{d^*}  = p^{d^* + 1}. 
\end{equation} 

Let $\bbV_{\sigma, j}$ denote the variable selection problem with response variable $\sX_j$, set of candidate predictor variables $\{ \sX_i \colon i \in \cA_p^\sigma(j) \}$ and true data-generating model given in~\eqref{eq:true.model.perm} 
parameterized by $( S^*_{\sigma, j},  \beta^*_{\sigma, j}, \omega^*_{\sigma, j}, \varepsilon_{\sigma, j} )$. 
Recall that $\cA_p^\sigma(j)$ is the index set of variables that precede $\sX_j$ in the permutation $\sigma$ and thus $S^*_{\sigma, j} \subseteq \cA_p^\sigma(j) $.  
For $\bbV_{\sigma, j}$ with arbitrary $\sigma \in \bbS^p$ and $j \in [p]$, 
Assumptions~\ref{ass:prior},~\ref{ass:size} and~\ref{ass:beta} used in Theorem~\ref{th:var.sel} directly follow from Assumptions~\ref{A:prior},~\ref{A:size} and~\ref{A:beta}, respectively.  Further, by Remark~\ref{rmk:beta.min}, we always have $\omega^*_{\sigma, j} \in (\vmin, \vmax)$ when Assumption~\ref{A:eigen} holds. 
For sufficiently large $n$, by Lemmas~\ref{lm:assA} and~\ref{lm:event} that we prove below using Assumptions~\ref{A:eigen} and~\ref{A:np},  Assumptions~\ref{ass:eigen} and~\ref{ass:error} in Theorem~\ref{th:var.sel} hold  for all $p! p$ variable selection problems in the set $\{\bbV_{\sigma, j} \colon \sigma \in \bbS^p, j \in [p] \}$ with $\rho = 4 \din + 6$ and probability at least $1 - 3p^{-1}$, where $\rho$ is as given in Assumption~\ref{ass:error}. 
 The claim then follows from Theorem~\ref{th:var.sel}. 
\end{proof}

\begin{lemma}\label{lm:assA}
If Assumption~\ref{A:eigen} holds and $\din \log p = o(n)$, then for sufficiently large $n$, 
\begin{align*}
         \bbP^* \Big\{ \;&  n \vmin  \leq   \min_{S \in \model (2 \din) } \lmin( \X_S^\top \X_S ) \leq  \max_{S \in \model (2 \din) } \lmax( \X_S^\top \X_S ) \leq n \vmax   \Big\} \geq 1 -  2 e^{- n \delta_0^2 / 16}, 
\end{align*}
where $\model (2 \din) = \{ S \subseteq [p] \colon |S| \leq 2 \din \}$.
\end{lemma}
\begin{proof}
For any $S \subseteq [p]$, let $A_S =(\Sigma^*_S)^{-1/2} \X_S^\top \X_S (\Sigma^*_S)^{-1/2} $ where $\Sigma^*_S$ denotes the $|S| \times |S|$ submatrix of $\Sigma^*$ with rows/columns indexed by $S$. 
By~\citet[Theorem 2.6]{rudelson2010non}, for any $S$ and $a > 0$,  we have 
\begin{align*}
\bbP^* \left\{ (\sqrt{n} - \sqrt{|S|} - a)^2 \leq  \lmin( A_S ) \leq \lmax( A_S ) \leq (\sqrt{n} + \sqrt{|S|} + a)^2 \right\} \geq 1 - 2 e^{-a^2 / 2}.
\end{align*}
By the submultiplicative property of operator norms, 
\begin{align*}
     \lmin( A_S ) \lmin(\Sigma^*_S) \leq   \lmin(\X_S^\top \X_S),  \quad \quad 
     \lmax( A_S ) \lmax(\Sigma^*_S) \geq   \lmax(\X_S^\top \X_S). 
\end{align*} 
Choose $a = \delta_0 \sqrt{n} / 2$ and $n \geq 8 \din / \delta_0^2$; the latter is allowed  since $\din = o(n)$. Using Assumption~\ref{A:eigen}, we find that 
 \begin{align*}
     \bbP^* \Big\{ \;&  n \vmin  \leq   \lmin( \X_S^\top \X_S ) \leq  \lmax( \X_S^\top \X_S ) \leq n \vmax   \Big\} \geq 1 - 2 e^{-a^2 / 2}.
 \end{align*}
Observe that $ \{ S \subseteq [p] \colon |S| \leq 2 \din \}$ contains less than $p^{2 \din}$ elements. 
By the  union bound and the assumption that $\din \log p = o(n)$, $ n \vmin  \leq   \lmin( \X_S^\top \X_S ) \leq  \lmax( \X_S^\top \X_S ) \leq n \vmax $ holds for all $S \in \model(2 \din)$ with probability at least $1 - 2  e^{-n \delta_0^2 / 8 + o(n)} $, from which the claim follows. 
\end{proof}

\begin{lemma}\label{lm:event}
Suppose Assumption~\ref{A:np} holds and $d^* \leq \din$.  
Let $z_{\sigma, j}$ be as defined in~\eqref{eq:all.error}. Then,  for sufficiently large $n$, 
\begin{align*}
      &  \bbP^* \left\{ \min_{\sigma \in \bbS^p} \min_{j \in [p]}\min_{ S  \in \model^\sigma(j, \din)}   
      %\min_{ S \subseteq \cA_p^\sigma(j)  \colon |S| \leq \din}   
       z_{\sigma, j}^\top  \oproj_S z_{\sigma, j}  \leq n/2   \right\} \leq e^{- n / 96}, \\[2pt]
      &  \bbP^* \left\{ 
      \max_{\sigma \in \bbS^p} \max_{j \in [p]}\max_{ S  \in \model^\sigma(j, \din)}    \max_{k \in \cA_p^\sigma(j)} 
z_{\sigma, j}^\top (\proj_{S \cup \{k\} } - \proj_S) z_{\sigma, j}  \geq (4 \din + 6)  \log p  \right\} \leq 2 p^{- 1}. 
\end{align*}  
\end{lemma}
\begin{proof}
Let $z \sim \NM_n(0, I)$. 
Given a fixed projection matrix $\oproj_S$, we have $z^\top  \oproj_S z \sim   \chi_{ n - |S|}^2$.  By~\citet[Lemma 1]{laurent2000adaptive}, for any $a > 0$, 
\begin{align*}
    \bbP \left\{  \frac{ z^\top  \oproj_S z }{ (n - |S|)   } \leq  1 - a  \right\} \leq e^{- (n - |S|) a^2 / 4 }. 
\end{align*}
Choosing $a = 1/3$ and sufficiently large $n$ so that $|S| \leq \din  \leq  n / 4$, we obtain that 
$$ \bbP \left\{   z^\top  \oproj_S  z \leq  n   / 2   \right\} \leq e^{- n / 48 }.$$

Observe that for any $S \subseteq \cA_p^\sigma(j)$,  $z_{\sigma, j}$ is independent of $\X_S$ under the probability measure $\bbP^*$. Hence, we can treat $\oproj_S$ as fixed and apply the union bound and~\eqref{eq:Z.bound} to get 
\begin{align*}
    \bbP^* \left\{   \min_{\sigma \in \bbS^p} \min_{j \in [p]}\min_{ S \subseteq \cA_p^\sigma(j)  \colon |S| \leq \din}    z_{\sigma, j}^\top  \oproj_S z_{\sigma, j}  \leq \frac{n}{2}   \right\} \leq p^{\din + d^* + 1}  e^{- n / 48 }. 
\end{align*}
The first asserted inequality then follows from the assumptions $\din \log p = o(n)$ and $d^* \leq \din$.  

For any $S \subseteq [p]$ and $k \notin S$, $\proj_{S \cup \{j\}} - \proj_S$ is another projection matrix with rank $1$. Hence, using a standard tail bound for Gaussian distribution, for any $a > 0$, we find that 
\begin{align*}
      \bbP \left\{   z^\top (\proj_{S \cup \{k\} } - \proj_S)z  \geq a  \log p  \right\} \leq 2 e^{-  a \log p / 2 }, 
\end{align*} 
if $z \sim \NM_n(0, I)$ independently of $\X_{S \cup \{k\} }$.  Another application of the union bound then yields the second inequality. 
\end{proof}

%\subsection{Proofs of Theorem~\ref{th:sel0}(ii) and (iii)} \label{sec:proof.sel.coro}
\begin{proof}[Proof of Theorem~\ref{th:sel0}(ii)]
For any $S \in \model^\sigma(j,  \, \din)$, define
\begin{align*}
   \cN_j^\sigma( S ) =\;&  \{ S' \in \model^\sigma(j,  \, \din) \colon \exists \, k, l \in [p]  \text{ s.t. }   S' = (S \cup \{k\}) \setminus \{l\}\}. 
\end{align*}
Note that we allow $k \in S$ and $l \in [p] \setminus S$ so that $\cN_j^\sigma(S)$ includes the models that can be obtained from $S$ by an addition, deletion or swap. 
Observe that the neighborhood relation defined by $\cN_j^\sigma$ is symmetric, and $|\cN_j^\sigma(S)| \leq 1 + p + (p - \din ) \din \leq p^2$ (if $p \geq 2$) for each $S \in \model^\sigma(j,  \, \din)$. 
By part (i), the triple $(\model^\sigma(j,  \, \din),  \cN_j^\sigma, e^{\score_j})$ satisfies Condition~\ref{cond:unimodal} with $t_1 = 2$, $t_2 = t$ and $g_j^\sigma$ being the canonical transition function. 
The result then follows from  Theorem~\ref{th:path}(ii). 
\end{proof}

\begin{proof}[Proof of Theorem~\ref{th:sel0}(iii)]
It suffices to prove the claim for $\dout = p$.   
Observe that 
\begin{align*}
     \sum_{G \in \dag^\sigma(\din, p)}   e^{ \score ( G )} 
    =   \sum_{G \colon \Pa_j(G) \in  \model^\sigma(j, \din) }   \prod_{j=1}^p e^{\scorej (\Pa_j(G)) }
     =  \prod_{j=1}^p \sum_{ S \in  \model^\sigma(j, \din)  } e^{ \scorej(S) }. 
\end{align*}
Thus, using part (ii), we find that  
\begin{align*}
    \frac{   \sum_{G \in \dag^\sigma(\din, p)}  e^{ \score (  G  )  } } {  e^{ \score (  G^*_\sigma ) } } 
    = \prod\limits_{j=1}^p \frac{   \sum_{ S \in  \model^\sigma(j, \din)  } e^{ \scorej(S) } }{ e^{ \scorej(S^*_{\sigma, j}) } } \leq (1 - p^{-(t - 2)})^{-p}
     \leq  \frac{1}{ 1 - p^{- (t - 3)}},
\end{align*}
for every $\sigma \in \bbS^p$, from which the result follows.
\end{proof}

\subsection{Proof of Theorem~\ref{th:cpdag.consist}}\label{sec:proof.cpdag.consist}
\begin{proof}
By  Lemma~\ref{lm:bound.nb}, for any $\cE \in \cpdag(\din, \dout)$, 
\begin{align*}
    \left| \nbg( \cE ) \right|  \leq 3 p (p - 1)(\din + \dout) p^{t_0}
    \leq 3 t_0  p^{t_0 + 2} \log_2  p = o(p^{t_0 + 3}).  
\end{align*}   
Let $\mathcal{F}$ denote the event on which the conclusion of Theorem~\ref{th:sel0} holds, which happens with probability at least $1 - 3 p^{-1}$. By Theorem~\ref{th:sel0}(i) and Corollary~\ref{coro:path.cpdag},  on $\mathcal{F}$, 
the triple $(\cpdag(\din, \dout), \nb, \post)$ satisfies Condition~\ref{cond:unimodal} with $t_1 = t_0 + 3$ and $t_2 = t$. 
The results then follow from Theorem~\ref{th:path}(i) and (ii).  
\end{proof}

\subsection{Consistency results for sub-Gaussian random matrices} \label{sec:supp.subgauss}
Now we generalize our consistency results by considering sub-Gaussian random matrices. 
%We first make a remark on the relationship between strong faithfulness and strong beta-min conditions. 
%Recall the notation $\Sigma_{A, B | C}$ introduced in~\eqref{eq:partial}. 
%If $\sX = (\sX_1, \dots, \sX_p)$ has a non-degenerate covariance matrix $\Sigma$, for any $i, j \in [p]$ and $S \subseteq [p] \setminus \{i, j\}$, we define 
%\begin{align*}
%\text{Partial correlation: }  & \mathrm{corr}(\sX_i, \sX_j \mid \sX_S)  =   \Sigma_{i, j | S}  / \sqrt{ \Sigma_{i | S} \Sigma_{j | S} }, \\
%\text{Partial regression coefficent: }  & \mathrm{coef}(\sX_i, \sX_j \mid \sX_S)  =   \Sigma_{i, j | S}  / \Sigma_{j | S}, 
%\end{align*}
%where $\mathrm{coef}(\sX_i, \sX_j \mid \sX_S) $ is the regression coefficient of $\sX_j$ when we project $\sX_i$ onto the $L^2$ space spanned by $(\sX_j, \sX_S)$. 

\begin{theorem}\label{th:sub-gauss}
Let $\X$ be an $n \times p$ random matrix, of which each row is an i.i.d. copy of a $p$-dimensional sub-Gaussian random vector with mean zero, covariance matrix $\Sigma^*$ and sub-Gaussian parameter bounded by a universal constant $C_{\rm{sub}} \in (0, \infty)$.    
Suppose $\NM_p(0, \Sigma^*)$ is perfectly Markovian w.r.t. a DAG $G^*$.  %for each $\sigma \in \bbS^p$, 
Let $(B_\sigma^*, \Omega_\sigma^*)$  be as given in Definition~\ref{def:imap}.  
%Let $d^*$ be given by~\eqref{eq:def.dstar} and assume the following hold. 
%For each $\sigma \in \bbS^p$, define $(B_\sigma^*, \Omega_\sigma^*)$ to be the unique pair in $\chol(\sigma)$  such that 
%$$(I - (B_\sigma^*)^\top)^{-1} \Omega_\sigma^* (I - B_\sigma^*)^{-1} = \Sigma^*.$$
%For each $\sigma \in \bbS^p$ and $j \in [p]$, define 
%\begin{align*}
%    S^*_{\sigma, j} = \left\{ i \colon  (B_\sigma^*)_{i j} \neq 0  \right\}, 
%\end{align*}
%and let $d^* = \max_{\sigma \in \bbS^p} \max_{j \in [p]} |S^*_{\sigma, j}|$.  %%% this is weaker, just not that convenient since eventually we need to use the pigeonhole lemma where we need max degree instead of max in-degree
%Assume the following hold. 
\begin{enumerate}[label=(F\arabic*), ref=(F\arabic*)]
    \item There exist  $\vmin, \vmax > 0$  and a universal constant $\delta_0 > 0$ such that \label{asg1}
    \begin{align*}
    \frac{\vmin}{(1 - \delta_0)^2} \leq \lmin(\Sigma^*) \leq \lmax(\Sigma^*) \leq \frac{\vmax}{(1 + \delta_0)^2}. 
    \end{align*}  
    \item The in-degree parameter $\din$ and $n,p$ satisfy that $\din \log p = o(n)$. \label{asg2}
    \item Prior parameters satisfy that $\alpha \in (0, 1]$, $\kappa \leq n$, $ c_1 \sqrt{ 1 + \alpha / \gamma} \in [1, p]$,  and $c_2 \geq  \Cpen  \din \vmax^4 / \vmin^6$ for some universal constant $\Cpen > 0$. \label{asg3}
    \item  $(\nu_0 + 1) \max_{\sigma \in \bbS^p}\max_{j \in [p]} |\Pa_j(G^*_\sigma)| \leq \din$ where $\nu_0 = 20 \vmax^2  \vmin^{-4} (\vmax - \vmin)^2$.   \label{asg4}
    \item There exists a universal constant $\Cb > 0$ such that    \label{asg5} 
    \begin{align*}
    \min_{\sigma \in \bbS^p} \min_{j \in [p]} |(B_\sigma^*)_{i j} \colon  (B_\sigma^*)_{i j} \neq 0 |  \geq \Cb \frac{ c_2 \vmax^2 \log p }{\alpha \vmin^2 n}. 
    \end{align*}
\end{enumerate} 
Let $t > 0$ be an arbitrary universal constant, and choose sufficiently large $\Cbeta, \Cpen$. For sufficiently large $n$, with probability at least $1 - 3p^{-\din }$,  
    \begin{equation*} 
    \min \left\{     
    \scorej( g_j^\sigma(S)) - \scorej(S) 
    \colon \sigma\in \bbS^p,  j \in [p], \, S \in \model^\sigma(j,  \, \din) \setminus \{  S^*_{\sigma, j} \}   \right\} \geq t \log p,
    \end{equation*}
    where  $S^*_{\sigma, j} = \left\{ i \colon  (B_\sigma^*)_{i j} \neq 0  \right\}$, 
     $\scorej$ is given by~\eqref{eq:post.modular} and  $\{ g_j^\sigma \colon j \in [p], \sigma \in \bbS^p \}$ is given by Definition~\ref{def:gj}.  That is, the conclusion of part (i) of Theorem~\ref{th:sel0} holds. 
\end{theorem}

\begin{remark}
The assumption that $\NM_p(0, \Sigma^*)$ is perfectly Markovian w.r.t.  $G^*$ implies that $B^*_\sigma$ is the weighted adjacency matrix of $G^*_\sigma$, the minimal I-map of $G^*$. %Though the actual distribution of $\sX$ may not be normal, 
Observe that the scoring criterion we use only depends on the data via the sample covariance matrix $n^{-1} X^\top X$, which explains why we impose the assumption on $\Sigma^*$ instead of the actual CI relations among the variables.  
One may also consider a DAG-perfect SEM model with sub-Gaussian errors as in~\citet{nandy2018high}, in which case the distribution of $\sX$ is perfectly Markovian w.r.t. $G^*$. 
\end{remark}

\begin{proof}
We will use Theorem~\ref{th:var.sel2}. 
Let $\Sigma^*_\sigma$ denote the permuted true covariance matrix such that $(\Sigma^*_\sigma)_{ij} = \mathrm{Cov}(\sX_{\sigma(i)}, \sX_{ \sigma(j) })$.
As in the proof of Theorem~\ref{th:sel0},  let $\bbV_{\sigma, j}$ denote the variable selection problem with response variable $\sX_j$ and set of candidate predictor variables $\{ \sX_i \colon i \in \cA_p^\sigma(j) \}$. The true covariance matrix associated with the problem $\bbV_{\sigma, j}$ is given by $(\Sigma^*_\sigma)_{[m + 1]}$ where $m = |\cA_p^\sigma(j) |$, which enables us to express $\sX_j$ by 
\begin{align*}
    \sX_j = \;&  \sum_{i \neq j} (B^*_\sigma)_{ij} \sX_i + \mathsf{e}_{\sigma, j}, 
\end{align*}
where $\mathsf{e}_{\sigma, j}$ is a random variable uncorrelated (but not necessarily independent) with $\{ \sX_i \colon i \in \cA_p^\sigma(j) \}$. 
It only remains to verify the assumptions used in Theorem~\ref{th:var.sel2}. 

%Clearly, under Assumption~\ref{asg1}, we always have $\vmin \leq \lmin( (\Sigma^*_\sigma)_{[m + 1]} ) \leq \lmax ( (\Sigma^*_\sigma)_{[m + 1]} ) \leq \vmax$ for any $\sigma \in \bbS^p$ and $0 \leq m \leq p - 1$. 
By Lemma~\ref{lm:subgauss} we prove below, under Assumptions~\ref{asg1} and~\ref{asg2}, for sufficiently large $n$,  Assumptions~\ref{ad21} and~\ref{ad22} hold  for all problems in $\{ \bbV_{\sigma, j} \colon \sigma \in \bbS^p, j \in [p]\}$  with probability at least $1 - 3p^{- \din}$, and the constant $\Cerr$ in Assumption~\ref{ad22} only depends on $C_{\rm{sub}}$ (and does not depend on $\sigma$ or $j$); denote this event set by $\mathcal{F}_n$.   
Assumptions~\ref{ad23},~\ref{ad24} and~\ref{ad25}  directly follow from Assumptions~\ref{asg3},~\ref{asg4} and~\ref{asg5}, respectively. 
Therefore, given any $t > 0$,  on the event $\mathcal{F}_n$ we can apply Theorem~\ref{th:var.sel2} to all problems in  $\{ \bbV_{\sigma, j} \colon \sigma \in \bbS^p, j \in [p]\}$, which means that for any $\sigma, j$, there exist sufficiently large $\Cpen, \Cbeta$ such that 
\begin{equation*} 
\min \left\{     
\scorej( g_j^\sigma(S)) - \scorej(S) 
\colon   S \in \model^\sigma(j,  \, \din) \setminus \{  S^*_{\sigma, j} \}   \right\} \geq t \log p. 
\end{equation*}
Further, by Lemma~\ref{lm:subgauss} and Theorem~\ref{th:var.sel2}, the choice of $\Cpen, \Cbeta$ only depends on universal constants $t$ and $C_{\rm{sub}}$. 
Hence,  we can choose $\Cpen, \Cbeta$ such that the above bound holds simultaneously for all $\sigma \in \bbS^p$ and $j \in [p]$. 
\end{proof}
 
\begin{lemma} \label{lm:subgauss}
Suppose Assumption~\ref{asg1} holds and $\din \log p = o(n)$. 
Then there exist universal constants $K_0, C_0 \in (0, \infty)$,  which only depend on $C_{\rm{sub}}$, such that,  for  sufficiently large $n$,  with probability at least $1 -  2 e^{-  C_0 n \delta_0^2 / 16} - 2 p^{-\din}$, 
\begin{align*}
&    \max_{S \in \model (2 \din) } \OPnorm{ n^{-1} \X_S^\top \X_S   - \Sigma^*_S }  \leq  K_0 \sqrt{\frac{\din \log p}{n}},  \text{ and} \\
& n \vmin  \leq   \min_{S \in \model (2 \din) } \lmin( \X_S^\top \X_S ) \leq  \max_{S \in \model (2 \din) } \lmax( \X_S^\top \X_S ) \leq n \vmax,   
\end{align*}
where $\model (2 \din) = \{ S \subseteq [p] \colon |S| \leq 2 \din \}$. % and  $\Sigma^*_S$ denotes the $|S| \times |S|$ principal submatrix of $\Sigma^*$ with rows and columns indexed by $S$. 
\end{lemma}
\begin{proof}
By~\citet[Remark 5.40]{vershynin2010introduction},  for any $a> 0$ and $S \subseteq [p]$, 
the following inequality holds with probability at least $1 - 2e^{- C_0 a^2}$,  
\begin{equation*}\label{eq:temp}
\OPnorm{ n^{-1} \X_S^\top \X_S   - \Sigma^*_S }  \leq   \max\{ u, u^2\}, \quad \text{where } u = C_1 \sqrt{\frac{|S|}{n} } + \frac{a}{\sqrt{n}}, 
\end{equation*}
and $C_0, C_1 > 0$ are universal constants. Choose $a = (K_0  / 2)\sqrt{ \din \log p }$ for some sufficiently large universal constant $K_0$. 
Observe that $|  \{ S \subseteq [p] \colon |S| \leq 2 \din \} | < p^{2 \din}$. 
Hence, the assumption $\din \log p = o(n)$ and a union bound yields that, for sufficiently large $n$, 
\begin{align*}
         \bbP^* \Big\{     \max_{S \in \model (2 \din) } \OPnorm{ n^{-1} \X_S^\top \X_S   - \Sigma^*_S }  \leq  K_0 \sqrt{\frac{\din \log p}{n}}  \Big\} \geq 1 -  2 p^{- \din}, 
\end{align*} 
where we have assumed again $K_0$ is sufficiently large. 

To prove the second part, define $A_S =   (\Sigma^*_S)^{-1/2} \X_S^\top \X_S (\Sigma^*_S)^{-1/2} $. 
By~\citet[Theorem 5.39]{vershynin2010introduction},  for any $a > 0$,  with probability at least $1 - 2 e^{-C_0 a^2}$, 
\begin{align*}
 (\sqrt{n} -  C_1 \sqrt{|S|} - a)^2 \leq  \lmin( A_S ) \leq \lmax( A_S ) \leq (\sqrt{n} + C_1 \sqrt{|S|} + a)^2. 
\end{align*} 
The rest is analogous  to the proof of Lemma~\ref{lm:assA}. 
%By the submultiplicative property of operator norms, 
% \begin{align*}
% \lmin( A_S ) \lmin(\Sigma^*_S) \leq   \lmin(\X_S^\top \X_S), \quad   \lmax( A_S ) \lmax(\Sigma^*_S) \geq   \lmax(\X_S^\top \X_S). 
% \end{align*}
% Choose $a = \delta_0 \sqrt{n} / 2$ and $n \geq 8 \din / \delta_0^2$; the latter is allowed  since $\din = o(n)$. It   follows from Assumption~\ref{A:eigen} that with probability at least $1 - 2 e^{-C_0 \delta_0^2  n / 8}$,  $n \vmin  \leq   \lmin( \X_S^\top \X_S ) \leq  \lmax( \X_S^\top \X_S ) \leq n \vmax  $.   
% The claim then follows by the union bound and the assumption  $\din \log p = o(n)$.  
\end{proof}
 
\newpage 

\section{Proofs and more examples for Sections~\ref{sec:mcmc} and~\ref{sec:disc} } \label{supp:proof.sec.mcmc}
\setallcounters

\subsection{Proof of Theorem~\ref{th:main.rapid}} \label{sec:proof.rwges}
\setcounter{figure}{0}
\setcounter{table}{0}

\begin{proof}[Proof of Theorem~\ref{th:main.rapid}] 
In the proof of Theorem~\ref{th:cpdag.consist}, we have shown that, with probability at least $1 - 3 p^{-1}$, the triple $(\cpdag(\din, \dout), \nb, e^\score)$ satisfies Condition~\ref{cond:unimodal} for 
\begin{align*}
 p^{t_1} =  3 t_0  p^{t_0 + 2} \log_2  p, \quad    t_2 = t. 
\end{align*}
The mixing time bound for \RWGES{} then follows from Theorem~\ref{th:path}(iv) and Theorem~\ref{th:path.better}.  
\end{proof}

\subsection{Proof of Corollary~\ref{lm:pi.min}} \label{supp:pi.min}

\begin{proof} 
Define $c_3 = c_1 \sqrt{1 + \alpha / \gamma}$. 
By~\eqref{eq:post.modular}, for any $G \in \dag^\sigma(\din, \dout)$, we have 
\begin{align*}
    \frac{ e^{ \score(G) } }{  e^{ \score(G^*_\sigma) } }
    = \;& \prod_{j = 1}^p \frac{ e^{ \scorej(\Pa_j(G))} }{ e^{ \scorej(\Pa_j(G^*_\sigma)) } } 
    = \left( c_3 p^{c_2} \right)^{|G^*_\sigma|-|G|} \prod_{j = 1}^p \left( \frac{X_j^\top  \oproj_{\Pa_j(G)} X_j }{ \varepsilon_{\sigma, j}^\top \oproj_{\Pa_j(G^*_\sigma)} \varepsilon_{\sigma, j}  } \right)^{-(\alpha n + \kappa )/2}  \\
    \geq \;& \left( c_3 p^{c_2} \right)^{- p \din } \prod_{j = 1}^p \left( \frac{X_j^\top  X_j }{ \varepsilon_{\sigma, j}^\top \oproj_{\Pa_j(G^*_\sigma)} \varepsilon_{\sigma, j}  } \right)^{-(\alpha n + \kappa )/2}  \\
    \geq \;& \left( c_3 p^{c_2} \right)^{- p \din } \prod_{j = 1}^p \left( \frac{ n \vmax  }{ n \omega^*_{\sigma, j} / 2 } \right)^{-(\alpha n + \kappa )/2} ,  
\end{align*}
where the last step follows from Lemmas~\ref{lm:assA} and~\ref{lm:event}. 
By Remark~\ref{rmk:beta.min}, $\omega^*_{\sigma, j} \geq \vmin$, which yields 
 \begin{equation}\label{eq:pi.min.1}
        \min_{G \in \dag^\sigma(\din, \dout)}  \frac{  \exp ( \score(G) ) }{ \exp (  \score(G^*_\sigma) ) }
        \geq   \left( c_3 p^{c_2}  \right)^{- p \din } \left( \frac{2\vmax}{\vmin} \right)^{ - p (\alpha n + \kappa) / 2}, 
        \quad \forall \, \sigma \in \bbS^p. 
\end{equation}

By Lemma~\ref{lm:imap.map}, there exists a Chickering sequence $(G_0 = G^*_\sigma, G_1, \dots, G_k = G^*)$ for some $k \leq pd^*$ such that, for $i \in [k]$, $G_i$ is obtained from $G_{i-1}$ by covered edge reversals and a single edge deletion. Since by removing any single edge from a DAG in $\dag^\tau$, its posterior score can increase by at most $c_3 p^{c_2}$, we have 
 \begin{equation}\label{eq:pi.min.2}
         \frac{\exp ( \score(  \EG{ G^*_\sigma } )  ) }{ \exp (  \score ( \cE^* ) ) } \geq (c_3 p^{c_2})^{ - p d^*}. 
 \end{equation}
For any $\cE \in \cpdag(\din, \dout)$, there exists a pair $(G, \sigma)$ such that $G \in \dag^\sigma(\din, \dout)$. Thus, 
\begin{equation*}
    \frac{ \post(\cE) }{ \post(\cE^*) }  
    = \frac{  e^{ \score(\EG{ G^*_\sigma } ) } }{  e^{  \score(\cE^*) }  }  \frac{ e^ {\score(G) } }{ e^{  \score (G^*_\sigma ) } },
\end{equation*}
Combining~\eqref{eq:pi.min.1} and~\eqref{eq:pi.min.2}, we obtain  the asserted bound on $\post(\cE)/\post(\cE^*)$. 
Under the setting of Theorem~\ref{th:cpdag.consist}, we have the strong selection consistency, which implies that $\post(\cE^*) \geq 1/2$ for all sufficiently large $n$. 
Hence, $\log \post(\cE)$ can be bounded a polynomial of $n$ and $p$. This completes the rapid mixing proof for \RWGES{}.
\end{proof}

\subsection{Proof of Theorem~\ref{th:dag.rapid}} \label{supp:proof.dag.rapid} 
\begin{proof}  
By Theorem~\ref{th:sel0}(i) and Corollary~\ref{coro:path.dag},  with probability at least $1 - 3p^{-1}$, 
the triple $(\dag^\sigma(\din, \dout), \nb, \post^\sigma)$ satisfies Condition~\ref{cond:unimodal} for $p^{t_1}  =   \din p^2 / 2$ and $t_2 = t$.  (Note that though Corollary~\ref{coro:path.dag} used $t_1 = 3$,  by~\eqref{eq:dag.nb.size} the neighborhood size actually can be bounded by $\din p^2 / 2$.)
The mixing time bound then follows from Theorem~\ref{th:path}(iv) and Theorem~\ref{th:path.better}.  
\end{proof}

%\newpage 
%\section{Proofs for the examples} \label{sec:supp.proof.exs}

\subsection{Proof of Example~\ref{ex:slow1}}\label{sec:proof.slow1}

\begin{figure}[t]
    \begin{center}
    \includegraphics[width=0.98\linewidth]{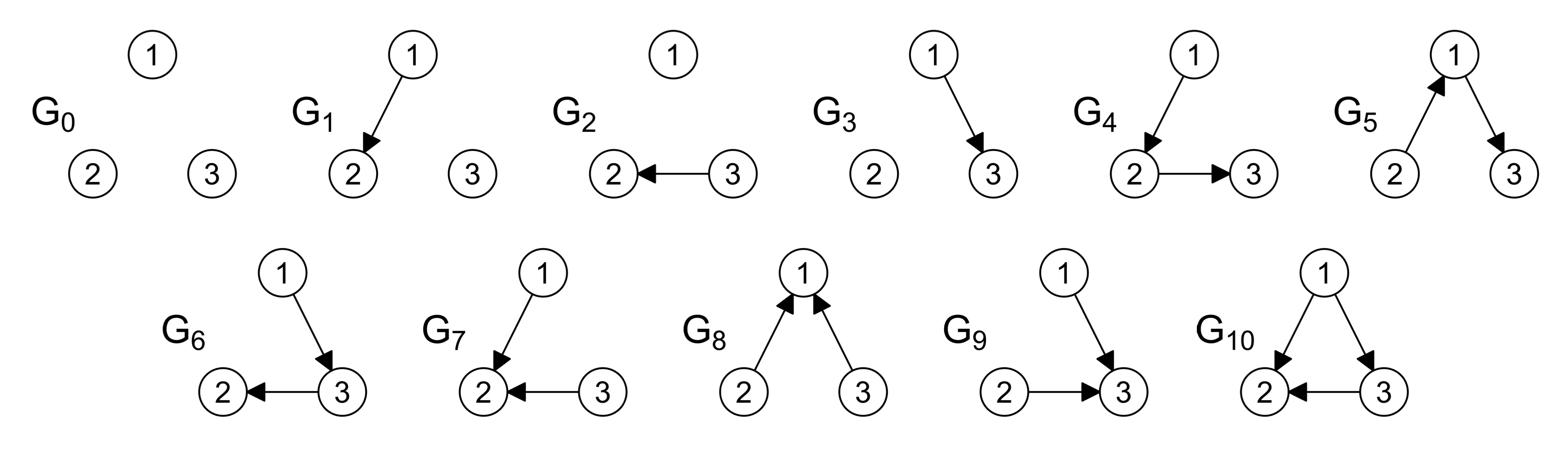}\\ 
    \caption{DAG models with three vertices. Each DAG represents a unique equivalence class. Any other $3$-vertex DAG model is Markov equivalent to one of these DAGs. 
    For Example~\ref{ex:slow1}, the true DAG is $G^* = G_4$ and the local mode is the equivalence class of $\tG = G_7$.
    For Example~\ref{ex:slow2.b}, $G^* = G_9$ and $\tG = G_{10}$. 
    } \label{fig.slow1}
    \end{center}
\end{figure}

\begin{proof}
Consider the space $\dag(2, 2)$ for $p=3$, the collection of all $3$-vertex DAG models. 
By~\citet{andersson1997characterization},  there are $11$ labeled equivalence classes, which we show in Figure~\ref{fig.slow1} (for each equivalence class we plot one DAG member).   
The true DAG is given by $G^* = G_4$ and the local mode is the equivalence class that contains $\tG = G_7$. 
Consider $\sigma=(1, 3, 2)$, which is a topological ordering of $\tG$. 
The corresponding SEM representation can be written as 
\begin{align*}
    X_1 =\;& z_1,  \\
    X_3 =\;& b_1 b_2 X_1 + b_2z_2 +  z_3, \\
    X_2 =\;& \frac{b_1}{b_2^2 + 1} X_1 + \frac{b_2}{b_2^2 + 1} X_3 + \frac{1}{b_2^2 + 1} z_2 - \frac{b_2}{b_2^2 + 1} z_3. 
\end{align*} 
\noindent We prove the slow mixing by verifying the two conditions in Theorem~\ref{th:slow}: 
\begin{enumerate}%[(i)]
     \item $\bP(\EG{G_7}, \EG{G_7}) \geq 1 - e^{- c \sqrt{n} }$ for some universal constant $c$. 
     \item $\post( \EG{G_7} ) \leq 1/2$ for all sufficiently large $n$.  
\end{enumerate}

Consider (i) first. 
Since $\EG{ G_7 } = \{ G_7 \}$, the set $\nbg( {\EG{ G_7 }} )$ is determined by all the neighbors of $G_7$. 
Observe that $G_1, G_2 \in \dels(G_7)$ and $G_{10} \in \adds(G_7)$.
It can be shown that for $K \geq 4$ and sufficiently large $n$, 
\begin{align*}
    \frac{\post( \EG{G_7} )}{ \post( \EG{G_1} )} =\;& \exp \left\{  -c_2 \log p + \frac{\alpha n}{2} \log (1 + b_2^2)
    \right\}    \geq   p^{c_2  / 2}, \\
      \frac{\post( \EG{G_7} )}{ \post( \EG{G_2} )} =\;& \exp \left\{  -c_2 \log p + \frac{\alpha n}{2} \log \frac{(b_1^2 + 1)(b_2^2  + 1)}{b_1^2 b_2^2 + b_2^2 + 1} \right\}    \geq   p^{c_2  / 2},  \\
     \frac{\post( \EG{G_7} )}{ \post( \EG{G_{10}} )} =\;&   \exp \left\{  c_2 \log p - \frac{\alpha n}{2} \log \frac{b_1^2 b_2^2 + b_2^2 + 1}{b_2^2  + 1}    \right\}    \geq   p^{c_2  / 2}.
\end{align*}
To prove the first,  we use  $b_1^2 = b_2^2 =  o (1)$ and  $\log(1 + x) \sim x$ as $x \downarrow 0$   to get $\alpha n \log(1 + b_2^2) \sim K c_2 \log p$. Hence, if $K \geq 4$ and $n$ is sufficiently large, $\alpha n \log(1 + b_2^2) \geq 3 c_2 \log p$. 
The proof of the second inequality is similar. For the last one,  we use $b_1^2 b_2^2 = o(b_1^2) $  to show that $c_2 \log p$ has a larger order than the other term in the exponent. 
Since $c_2 = \sqrt{n}$, by the Metropolis rule, for any $\cE' \neq \EG{G_7}$, 
\begin{align*}
    \bP(\EG{G_7}, \cE') = \bK(\EG{G_7}, \cE') \min \left\{  1,  \, \frac{\post(\cE') \bK(\cE', \EG{G_7}) }{\post( \EG{G_7}) \bK(\EG{G_7}, \cE') }  \right\} \leq \frac{\post(\cE')  }{\post( \EG{G_7})} <  p^{- \sqrt{n} / 2}. 
\end{align*}
It then follows that (i) holds since $|\nbg( \EG{G_7} )|$ is bounded. Note that if $p$ goes to infinity, we can still use $|\nbg( \EG{G_7} )| \leq 3 p^2$ to show (i).
 
To prove (ii), we only need to compare $G_7$ with the true model $G_4$. 
Another routine calculation yields that, for large $n$,  
\begin{align*}
     \frac{ \post( \EG{G_4} )}{\post( \EG{G_7} )} =\;&   \exp \left\{     \frac{\alpha n}{2} \log \frac{b_1^2 b_2^2 + b_2^2 + 1}{b_2^2 + 1}    \right\}  \\
        \geq  \;& \exp\left\{  \frac{\alpha n}{4} b_1^2 b_2^2 \right\} 
     =   \exp\left\{   \frac{K^2 (\log p)^2 }{4 \alpha }  \right\}. %\geq       p^{4 \alpha^{-1} \log p }, 
\end{align*}
Since $p = 3$,  as long as we choose $K$ such that $K^2 \geq 4/\alpha$, we have 
$\post( \EG{G_4} ) / \post( \EG{G_7} ) \geq  3^{\log 3} \approx 3.34$, from which (ii) follows.  
\end{proof}

\subsection{Slow mixing examples for a CPDAG sampler}\label{sec:proof.slow2}

\citet[Definition 9]{he2013reversible} constructed a ``perfect'' set of CPDAG operators: 
insert/delete an undirected edge, insert/delete a direct edge, and  make/remove a v-structure; see Supplement~\ref{sec:supp.cpdag} for details. 
This set of CPDAG operators is said to be reversible and irreducible, which means that the induced neighborhood relation on $\cpdag$ is symmetric and neighborhood graph is connected (i.e., a random walk using this set of operators is irreducible). 
Let $\cNc$ denote the corresponding neighborhood function. 
A random walk MH algorithm can be constructed by proposing a state in $\cNc(\cdot)$ uniformly at random. This algorithm was implemented in~\citet{castelletti2018learning}. 
Hence, compared with \RWGES{}, the only difference is that the neighborhood $\nbg(\cdot)$ is replaced by $\cNc(\cdot)$.  
Unfortunately, this ``reversible'' CPDAG sampler can be slowly mixing even if the strong beta-min condition holds. 

\begin{example}\label{ex:slow2.b}
This is a more explicit construction of Example~\ref{ex:slow2}. 
Let $p = 3$ and $n$ tend to infinity. The extension to the case $p = n$ is  straightforward. 
Let  the true SEM model be given by 
\begin{align*}
    X_1 = z_1, \quad  X_2 =  z_2,  \quad X_3 = a_1 X_1  + a_2 X_2 + z_3, 
\end{align*}
where $z_1, z_2, z_3$ are as given in Example~\ref{ex:slow1} but the coefficients $a_1, a_2 > 0$ are assumed to be fixed.  
Hence, the true DAG model $G^*$ is $1 \rightarrow 3 \leftarrow 2$, which is a CPDAG itself. 
Choose prior parameters $c_1, c_2, \kappa, \alpha, \gamma$ as in Example~\ref{ex:slow1}. 
Let $\tcE$ be the equivalence class that contains all complete DAGs;   $\EGG{\tcE}$ is a complete undirected graph.  
The set $\cNc(\tcE)$ contains three equivalence classes  represented by CPDAGs of the form  $i - j - k$,  but $\cE^* \notin \cNc(\tcE)$. Below we prove that the random walk MH algorithm equipped with the neighborhood function $\cNc$ is slowly mixing since it can get stuck at $\tcE$ for exponentially many steps.  
\end{example}

\begin{proof}
We still use the numbering in Figure~\ref{fig.slow1}. 
The true DAG is $G^* = G_9$, and the local mode we consider is the equivalence class generated by $\tG = G_{10}$. By~\citet[Definition 9]{he2013reversible}, $\cNc( \EG{G_{10}} ) = \{ \EG{G_4}, \EG{G_5}, \EG{G_6} \}$. Since $p, a_1, a_2, \alpha$ are fixed constants and $c_2 = \sqrt{n}$, for sufficiently large $n$, 
we find that 
\begin{align*}
    \frac{\post( \EG{G_{10}} )}{ \post( \EG{G_9} )} =\;& p^{-c_2} = p^{-\sqrt{n} }, \\
    \frac{\post( \EG{G_{10}} )}{ \post( \EG{G_4} )} =\;& \exp \left\{  -c_2 \log p + \frac{\alpha n}{2} \log  (a_1^2 + 1) \right\}    \geq   e^{c n}, \\
      \frac{\post( \EG{G_{10}} )}{ \post( \EG{G_5} )} =\;& \exp \left\{  -c_2 \log p + \frac{\alpha n}{2} \log (a_2^2 + 1)  \right\}    \geq   e^{c n},  \\
     \frac{\post( \EG{G_{10}} )}{ \post( \EG{G_{6}} )} =\;&  \exp \left\{ - c_2 \log p + \frac{\alpha n}{2} \log \frac{(a_1^2 + 1)(a_2^2 + 1)}{a_1^2 + a_2^2 + 1}  \right\}    \geq   e^{c n}, 
\end{align*} 
for some universal constant $c > 0$. Thus, for the random walk MH algorithm using neighborhood relation $\cNc$, we have $\bP(\EG{G_{10}}, \EG{G_{10}}) \geq 1 - 3e^{ - cn}$. 
Since $\EG{G_{10}}$ has negligible posterior probability, the chain is slowly mixing by Theorem~\ref{th:slow}. 
\end{proof}

We give another more complicated example with $5$ nodes. 

\begin{example}\label{ex:slow3}
%We give another slow mixing example for the random walk MH algorithm using $\cNc$ on $\cpdag$. 
Let $p=5$ and the true DAG  $G^*$ be as given in Figure~\ref{fig5}. 
Consider the DAG $H = G \cup \{1 \rightarrow 4\}$.  
By (dd2) in Definition 9 of~\citet{he2013reversible}, $\EG{G^*} \notin \cNc(\EG{H})$. 
Note that since $\cNc$ defines a symmetric relation, this is equivalent to claiming that we cannot move from  $\EGG{G^*}$ to $\EGG{H}$, which is easy to prove: if we add $1 \rightarrow 4$ to $\EGG{G^*}$, all consistent extensions of the resulting PDAG cannot have the v-structures $1 \rightarrow 4 \leftarrow 3$ and $1 \rightarrow 4 \leftarrow 2$, which exist in $\EGG{H}$.  
Actually, according to~\citet[Definition 9]{he2013reversible}, there are only 8 possible operations that we may apply to the CPDAG of $H$, which we list in Table~\ref{tb2:neigbhor}. 
Each operation uniquely defines a resulting CPDAG and thus $|\cNc(\EG{H})| = 8$. 
In Figure~\ref{fig6}, we plot a member DAG for each equivalence class in $\cNc(\EG{H})$.   

If $\score$ satisfies Condition~\ref{cond:local.consist} (local consistency), one can verify that $\EG{H}$ has a larger posterior probability than all 8 neighboring equivalence classes.  
First,  since $H$ is an I-map of $G^*$, $H_1, H_2, H_3$ should all have smaller scores.  
For $j = 4, \dots, 8$, observe that $H_j$ is obtained from $H$ by removing an edge between two nodes that are not conditionally independent given any set of other nodes.  Hence, $\EG{H}$ is a local mode on $(\cpdag, \cNc)$. 
\end{example}

On the restricted space $\cpdag(\din, \dout)$, the maximum size of $\cNc(\cdot)$ tends to be much smaller than that of $\nbg(\cdot)$ since the former can only grow polynomially in $\din + \dout$. 
However, for two different DAGs $G, G'$ such that $G' = G \cup \{i \rightarrow j \}$, we may have $\EG{G} \notin \cNc(\EG{G'})$ because removing the edge between $i$ and $j$ from $\EGG{G'}$ does not result in a unique CPDAG (more precisely,  the modified graph has no consistent extension; see Supplement~\ref{sec:supp.cpdag} for details). 
If $G$ happens to be the true DAG model, then $G'$ is likely to be a local mode that can trap the algorithm for an enormous number of iterations, since other modifications of $G'$ either yield a DAG that is not an I-map of $G$ or a DAG with more ``redundant'' edges. 
Nevertheless, as originally proposed in~\citet{he2013reversible}, one can define a random walk on $(\cpdag, \cNc)$ for efficiently sampling CPDAGs from the uniform distribution.

\begin{figure}[!htp]
    \centering
    \includegraphics[width=0.98\linewidth]{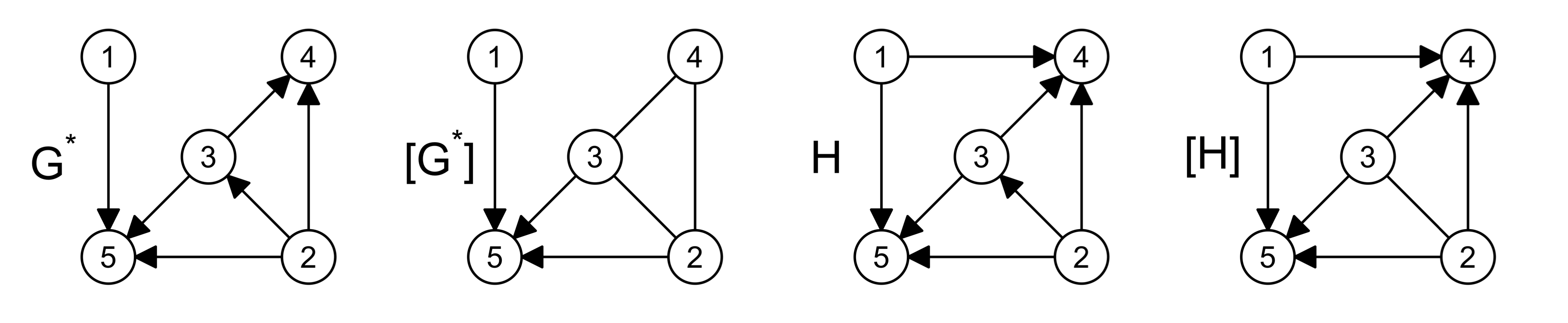}
    \caption{A slow mixing example for the random walk MH algorithm using neighborhood function $\cNc$.
    $G^*$: the true DAG; $\EG{G^*}$: the CPDAG of $G^*$; $H$: a DAG representing a local mode; $\EG{H}$: the CPDAG of $H$.}
    \label{fig5}
\end{figure}

\clearpage 

\begin{table}[!htp]
\centering
\caption{Edge operations that may be applied to the CPDAG of $H$ in Figure~\ref{fig5}. } \label{tb2:neigbhor}
%In the operator names, ``U'' means an undirected edge, ``D''  a directed edge, ``V'' a v-structure.
\begin{tabular}{|c|c|}
\hline
\textbf{Operator} & \textbf{Related edge(s)} \\ \hline
InsertU (insert an undirected edge) & $1 - 2$, $1 - 3$ , $4 - 5$ \\ \hline
DeleteU (delete an undirected edge)  & $2 - 3$  \\ \hline
InsertD (insert a directed edge)  & None \\ \hline
DeleteD  (delete a directed edge) &  $3 \rightarrow 4$,   $3 \rightarrow 5$, $2 \rightarrow 4$, $2 \rightarrow 5$ \\ \hline
MakeV  (make a v-structure)  & None        \\ \hline
RemoveV  (remove a v-structure) &    None   \\ \hline
\end{tabular}
\end{table}

\begin{figure}[!htp]
    \centering
    \includegraphics[width=0.98\linewidth]{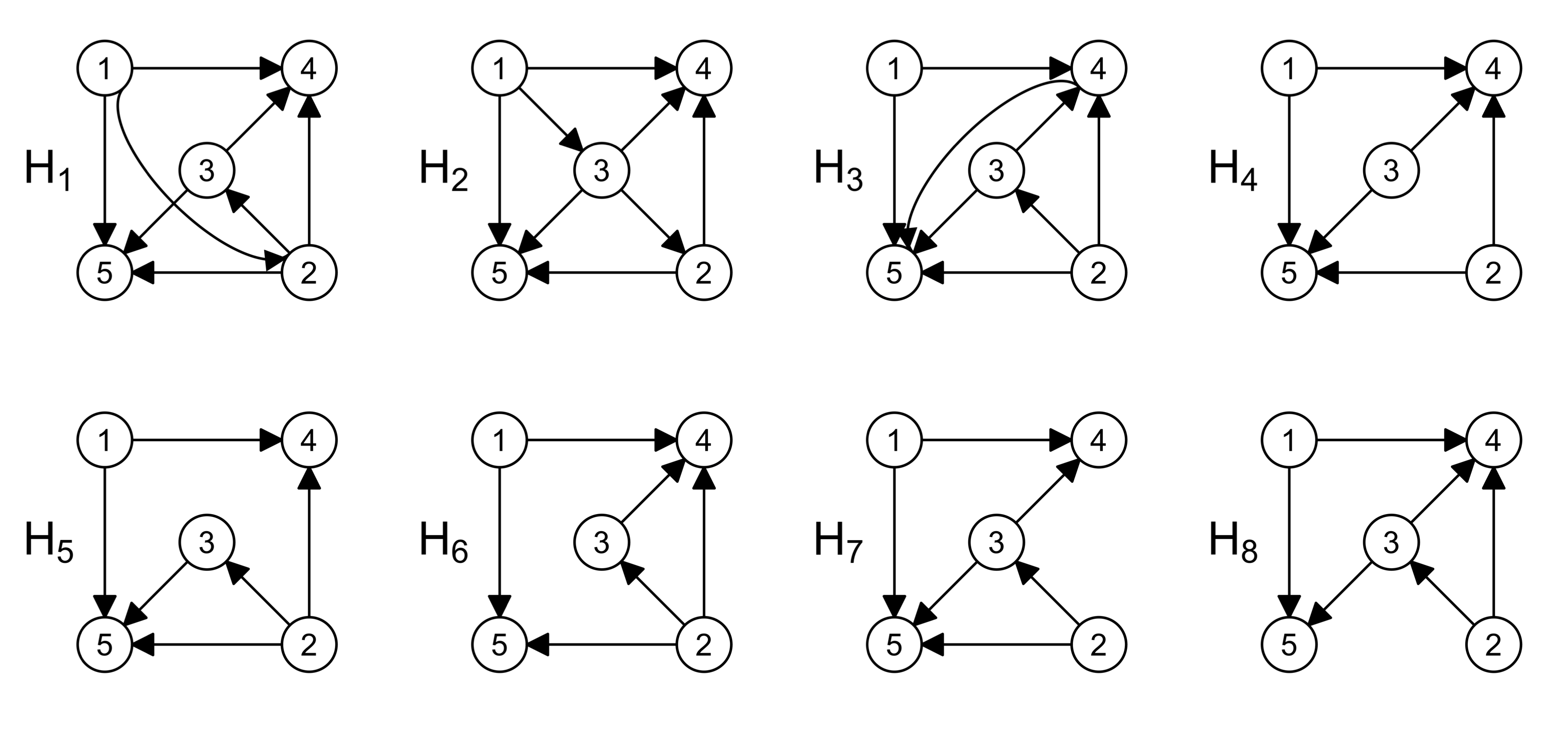}
    \caption{Characterization of $\cNc(\EG{H})$ where $H$ is as given in Figure~\ref{fig5}. 
    For each $\cE \in \cNc(\EG{H})$, we plot a member DAG of $\cE$. 
    The 8 DAGs correspond to the 8 operations listed in Table~\ref{tb2:neigbhor}. 
    $H_1$: insert $1 - 2$; $H_2$: insert $1 - 3$; $H_3$: insert 4-5; $H_4$: delete $2-3$;
    $H_5$: delete $3 \rightarrow 4$;  $H_6$: delete $3 \rightarrow 5$; $H_7$: delete $2 \rightarrow 4$; $H_8$: delete $2 \rightarrow 5$. 
     }
    \label{fig6}
\end{figure}

\begin{figure}[!t]
\begin{center}
\includegraphics[width=0.8\linewidth]{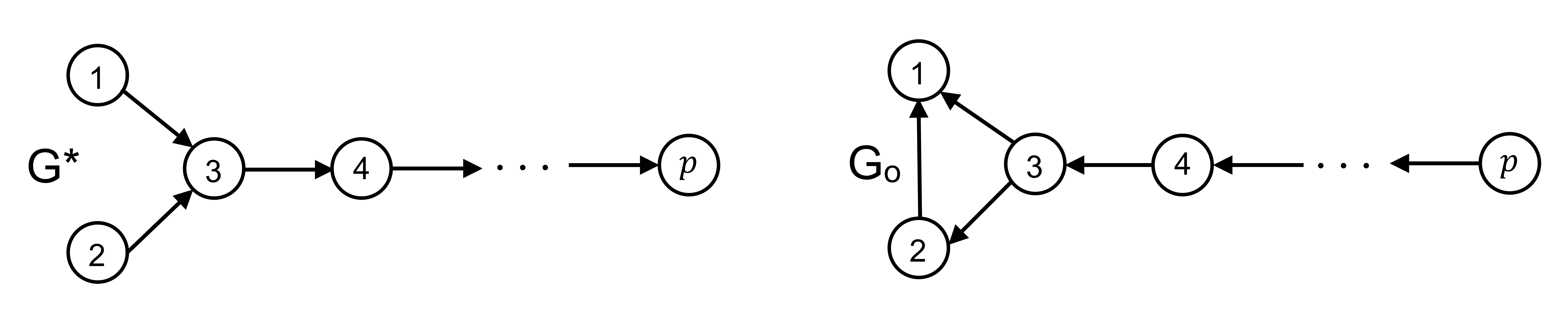} 
\caption{The DAGs considered in Supplement~\ref{supp:slow.structure.mcmc}. $G^*$: the true DAG. $G_0$: a minimal I-map of $G^*$.  }\label{fig:dag}
\end{center}
\end{figure} 

\subsection{An example for comparing \RWGES{} and structure MCMC}\label{supp:slow.structure.mcmc}

Consider the true data-generating DAG $G^*$ and another DAG $G_0$ shown in Figure~\ref{fig:dag} and assume $n = \infty$.  Observe that $G_0$ is a minimal I-map of $G^*$.  
Define $G_1 = G^* \cup \{2 \rightarrow 1\}$, which is Markov equivalent to $G_0$.  
Let  $\cE^* = \EG{G^*}$ and $\cE_0 = \EG{G_0}$. 
For ease of presentation, we assume the following condition on $\score$ holds, which is similar to but stronger than Condition~\ref{cond:local.consist}: for any distinct DAGs $G, G'$ such that $G' = G \cup \{i \rightarrow j\}$,  
\begin{enumerate}[label=(\roman*)]
    \item $\score(G) - \score(G') = \infty$ if $i \indep j \mid \Pa_j(G)$ in $G^*$; 
    \item $\score(G') - \score(G) = \infty$ if $i \notindep j \mid \Pa_j(G)$ in $G^*$. 
\end{enumerate}
If $p$ and the true data-generating model (perfectly Markovian w.r.t. $G$) are fixed and we let $n \rightarrow \infty$, this condition should hold for a large class of scoring criteria~\citep{chickering2002optimal}. We make this assumption so that if  an MH algorithm proposes to add $i \rightarrow j$ to $G$ but $i \indep j \mid \Pa_j(G)$ in $G^*$, the proposal will be rejected with probability one. 
Further,  we assume the search spaces are unrestricted, i.e., $\cpdag$ for \RWGES{} and $\dag$ for structure MCMC, and do not consider swap moves of \RWGES{} (but the same conclusion still holds if swap moves are used). 

\paragraph*{Analysis of \RWGES{}} 
By the definition of minimal I-maps given in Section~\ref{sec:notation}, one can verify that $G_0$ is a minimal I-map of $G^*$ with ordering $(p, p-1, \dots, 2, 1)$. Hence, it is straightforward to check that the only proposal in $\adds(\cE_0) \cup \dels(\cE_0)$ that will be accepted is $\cE^*$. Since the maximum degree of $G_0$ is always $3$ (which does not depend on $p$), analogously to the proof of Lemma~\ref{lm:bound.nb}, one can show that $|\adds(\cE_0) \cup \dels(\cE_0)|  \leq 2^3 p^2 = 8 p^2$.  That is, on average it takes \RWGES{} less than $8p^2$ iterations to propose $\cE^*$, which will always be accepted by our assumption on $\score$.  
 
\paragraph*{Analysis of structure MCMC} 
Structure MCMC is defined on the DAG space which uses single-edge addition, deletion and reversal as the proposal moves. For structure MCMC, to move from $G_0$ to $G^*$, we have to move to $G_1$ first, which requires the sampler to reverse all edges of $G_0$ except the edge $2 \rightarrow 1$. Moreover, these edges have to be reversed from right to left. %for example, the edge $p \rightarrow (p - 1) $ must be reversed first, since reversing any other edge (except $2 \rightarrow 1$) results in a cyclic graph or a DAG that is not an I-map of $G^*$.  Using the standard random walk theory, one can show that on average it takes structure MCMC $O(p^4)$ iterations to move from $G_0$ to $G_1$. A more detailed analysis is given in Supplement~\ref{supp:slow.structure.mcmc}. 

To show this, note that since $G_0$ is a minimal I-map, we cannot add or remove any edge of $G_0$. We can only move from $G_0$ to DAGs in $\revs(G_0)$, the set of all DAGs that can be obtained from $G_0$ by reversing one edge. Reversing $2 \rightarrow 1$ yields a Markov equivalent DAG but has no effect (the resulting DAG is essentially the same as $G_0$ compared to $G^*$), so we will ignore the reversal of this edge henceforth. 
Reversing $3 \rightarrow 1$ results in a cycle and thus is not allowed. Reversing $3 \rightarrow 2$ results in the v-structure $2 \rightarrow 3 \leftarrow 4$, which does not exist in $G^*$; thus,  the resulting DAG is not an I-map of $G^*$ and will be rejected by our assumption on $\score$. Similarly, reversing $4 \rightarrow 3$ will also be rejected since it results in the v-structure $3 \rightarrow 4 \leftarrow 5$ (assuming $p \geq 5$). 
So the only edge we can reverse other than $2 \rightarrow 1$ is the edge $p \rightarrow (p - 1)$. 
Repeating this argument, we see that structure MCMC has to reverse the edges $p \rightarrow (p - 1)$, $(p - 1) \rightarrow (p - 2)$, \dots, $4 \rightarrow 3$ in order; denote the resulting DAG by $G_2$ (see Figure~\ref{fig:dag2}).
Note that after the edge $p \rightarrow (p - 1)$ is reversed, structure MCMC may reverse the edge $(p - 1) \rightarrow p$ again and return to $G_0$. 
%at the DAG $( G_0 \setminus \{ p \rightarrow (p - 1) \}  ) \cup \{ (p - 1) \rightarrow p \}$, 
Hence, moving from $G_0$ to $G_2$ usually takes much more than $p - 3$ successful edge reversals. Indeed, this can be seen as a reflected symmetric random walk on $\{4, 5,  \dots, p\}$, and by a standard result from random walk theory, the expected number of successful edge reversals needed to move from $G_0$ to $G_2$ is $(p - 3)^2$. 
Since in structure MCMC the neighborhood size of each DAG has order $p^2$, we conclude that on average it takes structure MCMC $O(p^4)$ iterations to move from $G_0$ to $G^*$. 
  
\begin{figure}[!htp]
\begin{center}\bigskip 
\includegraphics[width=0.95\linewidth]{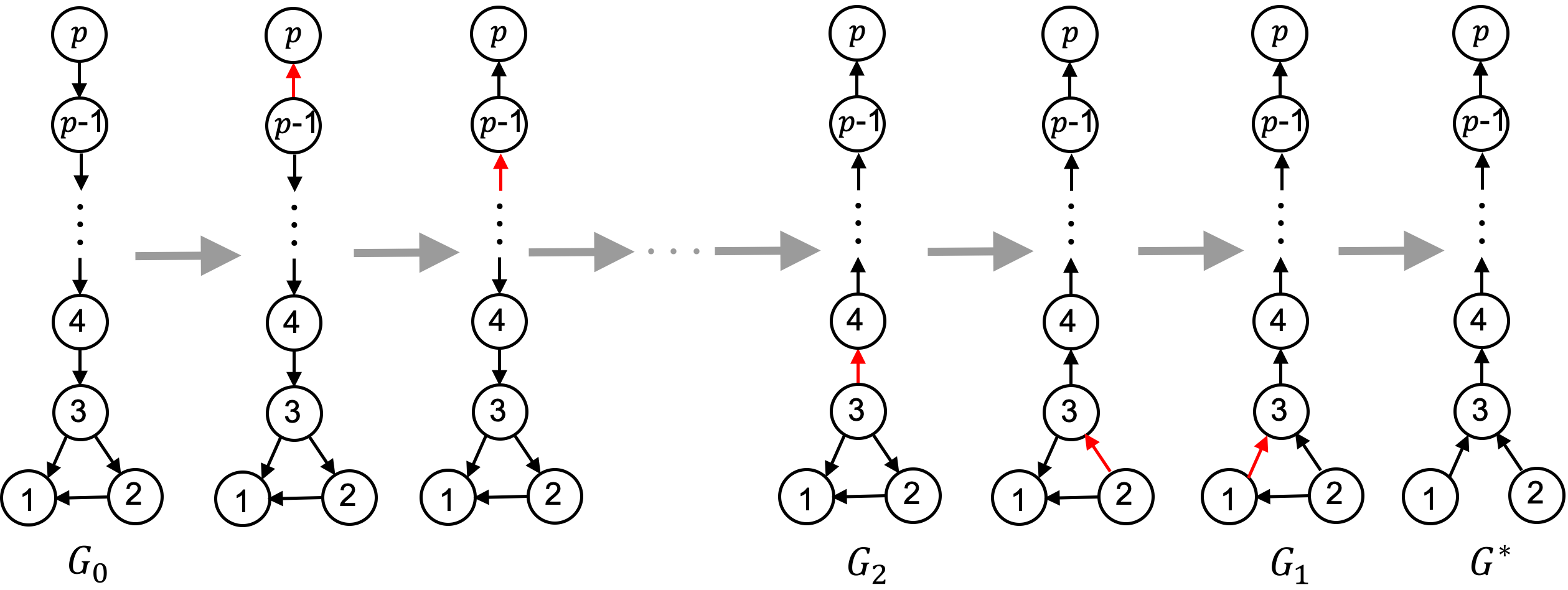} 
\caption{A likely path of structure MCMC from $G_0$ to $G^*$ in the example considered in Supplement~\ref{supp:slow.structure.mcmc}.  
Edges that have just been reversed are shown in red.  }\label{fig:dag2}
\end{center}
\end{figure}

\clearpage 
\newpage
\section{Implementation and simulation studies on \RWGES{} }\label{supp:numerics} 
\setallcounters 
\setcounter{figure}{0}
\setcounter{table}{0}

\subsection{Implementation of \RWGES{} }\label{supp:implementation}
To implement the \RWGES{} sampler, we need to propose new states from $\nbg(\cE)$ for each $\cE$. 
Though $\nbg$ is constructed by applying add-delete-swap moves to all member DAGs in $\cE$, we do not need to enumerate these DAGs to sample from $\nbg(\cE)$. 
Instead, recall the ``Insert'' and ``Delete'' operators of the GES algorithm defined in Supplement~\ref{sec:supp.cpdag}.   They can be used
to generate the sets $\adds(\cE)$ and $\dels(\cE)$, which only involve local modifications of the CPDAG  $\EGG{\cE}$. 
%Let $\cO_{\rm{ins}}(\cE)$ and $\cO_{\rm{del}}(\cE)$ denote the set of ``Insert'' operators and the set of ``Delete'' operators defined for  $\EGG{\cE}$, respectively. 
 Define 
\begin{align*}
\cO^{\rm{ins}}_{i j}(\cE)  =\;& \{ Insert(i, j, S) \colon  S \subseteq \und_j(\EGG{\cE}) \setminus \adj_i(\EGG{\cE}) \}, \\
\cO^{\rm{del}}_{i j}(\cE) =\;& \{ Delete(i, j, S) \colon  S \subseteq \und_j(\EGG{\cE}) \cap \adj_i(\EGG{\cE}) \}. 
\end{align*}
By Definition~\ref{def:ges.insert}, $\cO^{\rm{ins}}_{i j}(\cE)$ is the set of all insertion operators that add the edge $i \rightarrow j$ to $\EGG{\cE}$. By Definition~\ref{def:ges.del}, $\cO^{\rm{del}}_{i j}(\cE)$ is the set of all deletion operators that remove the edge $i \rightarrow j$ or $i - j$ from $\EGG{\cE}$. Further,  as explained in Supplement~\ref{sec:supp.cpdag}, if  $O$ is a valid operator in $\cO^{\rm{ins}}_{i j}(\cE)$ that converts $\cE$ to $\cE'$, then it can be matched with an operator $O' \in \cO^{\rm{del}}_{i j}(\cE')$ that can convert $\cE'$ back to $\cE$. 
Now we describe an efficient implementation of \RWGES{}.  

\begin{alg} \label{alg:rwges}
Let $\cE$ denote the current equivalence class and $\EGG{\cE}$ be its CPDAG. Sample an ordered pair $(i, j)$ uniformly from $[p]^2$ such that $ i \neq j$. 
\begin{enumerate}[label=(\roman*)]
\item If $i, j$ are adjacent in $\EGG{\cE}$, sample  $O$ uniformly from   $\cO^{\rm{del}}_{i j}(\cE)$. 
\item If $i, j$ are not adjacent in $\EGG{\cE}$, sample $O$ uniformly from   $\cO^{\rm{ins}}_{i j}(\cE)$. 
\end{enumerate}
Apply the operator $O$ to $\EGG{\cE}$. If the resulting PDAG has no consistent extension, stay at $\cE$.  
If the resulting PDAG has a consistent extension $G'$, propose moving to $\cE' = \EG{G'}$ and accept it with probability 
\begin{align*}
\text{case (i): } \min \left\{ 1,   \frac{ \post(\cE')}{ \post(\cE) } \frac{    |  \cO^{\rm{del}}_{i j}(\cE)  |  }{ |\cO^{\rm{ins}}_{i j}(\cE') | }   \right\}, \quad 
\text{case (ii): } \min \left\{ 1,   \frac{ \post(\cE')}{ \post(\cE) } \frac{    |  \cO^{\rm{ins}}_{i j}(\cE)  |  }{ |\cO^{\rm{del}}_{i j}(\cE') | }   \right\}. 
\end{align*}
\end{alg}
 
To implement the node degree constraints, for simplicity, we  require that the maximum degree $\leq d$ for some positive integer $d$ (i.e., any node is adjacent to at most $d$ nodes). Then, in case (ii) of Algorithm~\ref{alg:rwges}, if the degree of node $j$ is equal to $d$ but the degree of node $i$ is less than $d$, we replace the addition with a swap proposal. Note that there is no need to check the degree constraint in case (i) of Algorithm~\ref{alg:rwges}.  
Under our degree constraint, it is easy to check that  $|\cO^{\rm{ins}}_{i j}(\cE)| \leq 2^{d - 1}$ and $|\cO^{\rm{del}}_{i j}(\cE)| \leq 2^{d - 1}$. Hence, the probability of proposing any insertion or deletion operator is at least $p^{-2} 2^{1 - d}$, which is greater than or equal to $2 p^{- (2 + t_0)}$ if we assume $d \leq  t_0 \log_2 p$.  
We refer readers to~\citet{chickering2002optimal} for how to convert DAGs to CPDAGs and tricks that can further expedite  CPDAG operations.   
%As discussed in the previous subsection, a key reason why we can prove the rapid mixing of \RWGES{} is that it directly searches the space of equivalence classes, while for DAG MCMC methods traversing large equivalence classes could be problematic.  
%But it should be noted that such advantage of \RWGES{} can be realized in practice because there exist efficient local CPDAG operators.  

If we run the \RWGES{} sampler for $N$ iterations (after burn-in), we get $N$ equivalence classes $\cE^{(1)}, \dots, \cE^{(N)}$. 
For each $\cE^{(k)}$, we convert it to the corresponding CPDAG $H^{(k)} = \EGG{\cE^{(k)}}$, and we check which edges are in $H^{(k)}$; $i \rightarrow j, i \leftarrow j$ and $i - j$ are treated as different edges. 
For each $i < j$, \RWGES{} outputs the frequencies of the edge $i \rightarrow j$, $j \rightarrow i$ and $i - j$ being in $H^{(1)}$, $H^{(2)}$, \dots, $H^{(N)}$; denote them by $\hat{P}(i \rightarrow j), \hat{P}(j \rightarrow i)$ and $\hat{P}(i - j)$, respectively. 

\subsection{Further details about simulation studies}\label{supp:more.sim} 

\paragraph*{First simulation study}
Figure~\ref{fig:true.cpdag} shows the CPDAG representing the equivalence class of the DAG $G^*$ used in the simulation study presented in Section~\ref{sec:first.sim.rapid}.  
The true covariance matrix $\Sigma^*$ has $\lmax(\Sigma^*) = 9.6$ and $\lmin(\Sigma^*) = 0.15$.  
The simulated sample covariance matrix $\hat{\Sigma}_n = n^{-1} \X^\top \X$ has $\lmax(\hSigma_n) = 12.3$ and $\lmin(\hSigma_n) = 0.04$ for $n = 200$ and    $\lmax(\hSigma_n) = 10.7$ and $\lmin(\hSigma_n) = 0.11$ for $n = 800$. 

This simulation study also illustrates a typical way to construct a high-dimensional data set that satisfies the strong beta-min condition.  Recall that this condition is quite restrictive mainly because we need~\eqref{eq:beta.min} to hold for all $\sigma \in \bbS^p$ but $\bbS^p$ grows super-exponentially with $p$. 
However, if the true DAG $G^*$ can be partitioned into disconnected sub-DAGs, then the true covariance matrix $\Sigma^*$ and its inverse are block diagonal (after row/column permutation), and the strong beta-min condition is much easier to satisfy. In particular, if nodes $i, j$ belong to a maximal connected sub-DAG with $k$ nodes,   the mapping $\sigma \mapsto (B^*_{\sigma})_{ij}$ can take at most $2^{k-1}$ distinct values (since the parent set of node $j$ has at most $2^{k-1}$ possibilities).  
For the CPDAG shown in Figure~\ref{fig:true.cpdag}, there is one  maximal connected subDAG with 10 nodes, and the other connected subDAGs have at most $3$ nodes. 
A more general and simpler way to construct examples that satisfy the assumptions of Theorem~\ref{th:main.rapid} is given below. 
  
\begin{example}\label{supp:high.dim.example}
Let $q$ be a fixed positive integer and $\Sigma_0$  be a fixed $q \times q$ positive definite covariance matrix with $\lmin(\Sigma_0) = 2 \vmin$ and $\lmax(\Sigma_0) = \vmax / 2$, where $\vmin, \vmax$ are  nonzero universal constants. 
Let $p = m q$ where $m = \lfloor \exp(n^\xi ) \rfloor $  for some $\xi \in (0, 1)$. 
Let the true covariance matrix   $\Sigma^* = \diag(\Sigma_0, \dots, \Sigma_0)$ (i.e., a block diagonal matrix). 
Then, we have $\lmin(\Sigma^*) = 2 \vmin$ and $\lmax (\Sigma^*) = \vmax / 2$, which satisfies the restricted eigenvalue condition~\ref{A:eigen}. 
Further, the block diagonal structure of $\Sigma^*$ implies that $d^* \leq q$ and 
$$ \beta_{\rm{min}} =   \left\{  | (B^*_\sigma)_{ij}|  \colon  (B^*_\sigma)_{ij}  \neq 0, \, i, j \in [p], \, \sigma \in \bbS^p \right\}$$ is a universal nonzero constant (since it can be determined using $\Sigma_0$.) 
Choose  $\din = \nu_0 q$ where $\nu_0 = 4 \vmax^2 (\vmax - \vmin)^2 / \vmin^4 + 1$  and $\dout = 2 \nu_0 q^2$ so that $d^* \din + 1\leq \dout$ and Assumptions~\ref{A:np} and~\ref{A:size} hold. 
Since $\din, \dout$ are fixed constants, we can assume that $\din + \dout \leq \log_2 p$ for sufficiently large $n$. 
For the prior parameters, we  choose $\alpha = 1/2$ and $c_2 = 6 \din + 9 + t$ with $t > 4$ so that Assumption~\ref{A:prior} holds. 
Finally,  the strong beta-min condition is satisfied for sufficiently large $n$ since $\beta_{\rm{min}}$ is a fixed constant but the lower bound in~\eqref{eq:beta.min}  goes to zero. Note that the true underlying DAG model $G^*$ can have at most $p (q - 1) /2$ edges. 
\end{example}

\begin{figure}[!htp]
\begin{center}
\includegraphics[width=0.5\linewidth]{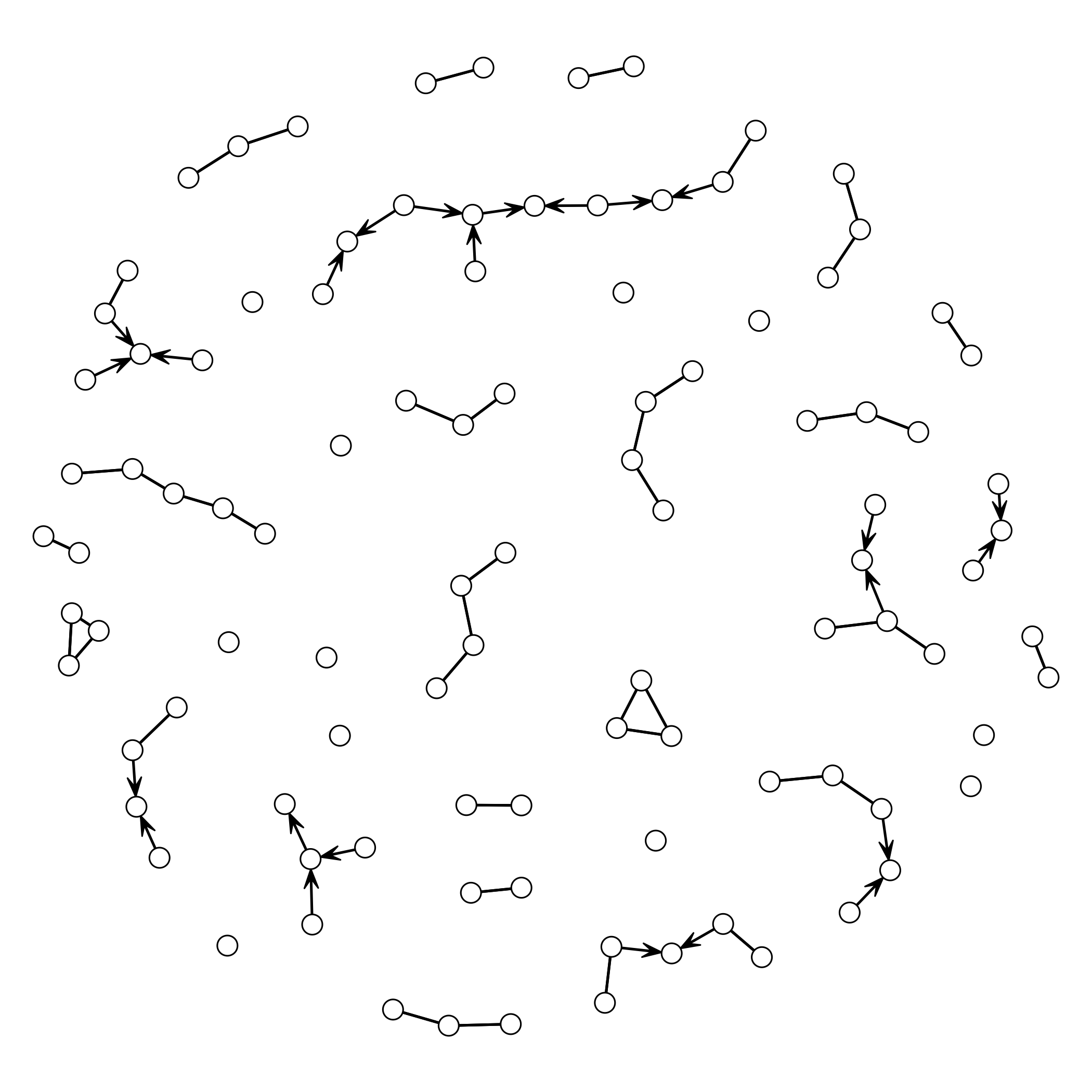} 
\caption{The true CPDAG used to generate the data set used in the first simulation study. }\label{fig:true.cpdag}
\end{center}
\end{figure}

\paragraph*{Second simulation study} 
For the simulation study presented in Section~\ref{sec:second.sim.high.dim},  the median extreme eigenvalues of the true covariance matrix and sample covariance matrix are reported in Table~\ref{table:high.eigen}. 

The TPR and FPR statistics reported in Tables~\ref{table:high_dim} and~\ref{table:c2} are calculated as follows. 
Let $H^* = \EGG{G^*}$ be the CPDAG of the true equivalence class.  
Recall that we say $i, j$ are adjacent in  $H^*$ if they are connected by either a directed or undirected edge. 
Define 
\begin{align*}
\mathrm{TPs} = \;&  \sum_{i < j} \ind\{ i, j \text{ are adjacent in } H^* \} \left\{  \hat{P}(i \rightarrow j) + \hat{P}(j \rightarrow i) + \hat{P}(i - j) \right\}, \\
\mathrm{TP} = \;& \sum_{i < j} \left[ \ind\{ i \rightarrow j \in H^* \}    \hat{P}(i \rightarrow j) + 
\ind\{ j \rightarrow i \in H^* \}  \hat{P}(j \rightarrow i) + 
\ind\{ i - j \in H^* \}  \hat{P}(i - j)   \right], \\ 
\mathrm{FP} =\;& \sum_{i < j} \ind\{ i, j \text{ are not adjacent in } H^* \} \left\{  \hat{P}(i \rightarrow j) + \hat{P}(j \rightarrow i) + \hat{P}(i - j) \right\}.  
\end{align*}
Then,  ``TPR (skeleton)'' is defined as $\mathrm{TPs}/ | G^* |$, where $|G^*|$ is also the number of edges in $H^*$. 
TPR is defined as $\mathrm{TP}/ | G^* |$, and FPR is defined as $\mathrm{FP} / \{  p(p-1)/2 -  |G^*| \}$.

\begin{table}[!htp]
    \centering
    \begin{tabular}{|cc|cccc|}
    \hline  
    p  &  n & $\lmax(\Sigma^*)$ & $\lmin(\Sigma^*)$ & $\lmax(\hSigma_n)$ & $\lmin(\hSigma_n)$ \\
    \hline  
7 & 60 & 7.7 (2.9--15) & 0.17 (0.12--0.3) & 6.8 (3.6--19) & 0.14 (0.09--0.28) \\
14 & 90 & 13 (5.6--35) & 0.15 (0.11--0.23) & 14 (6.8--43) & 0.11 (0.07--0.2) \\
28 & 120 & 30 (16--102) & 0.10 (0.08--0.13) & 33 (16--113) & 0.06 (0.03--0.09) \\
56 & 150 & 75 (43--188) & 0.09 (0.07--0.11) & 78 (41--173) & 0.04 (0.04--0.05) \\
112 & 180 & 149 (73--348) & 0.08 (0.06--0.09) & 152 (75--353) & 0.02 (0.01--0.02) \\
224 & 210 & 252 (146--631) & 0.07 (0.06--0.08) & 239 (148--633) & 0   \\
448 & 240 & 411 (251--844) & 0.06 (0.06--0.07) & 436 (264--814) & 0 \\
  \hline 
  \end{tabular}
   \caption{  Extreme eigenvalues of $\Sigma^*$ and $\hat{\Sigma}_n = n^{-1} \X^\top \X$ 
   for the simulation study presented in Section~\ref{sec:second.sim.high.dim}. Each cell gives the median value out of the 20 replicates with $0.05$ and $0.95$ quantiles  given in parentheses. 
   }
    \label{table:high.eigen}
\end{table}

\paragraph*{Third simulation study} 
For the simulation study conducted in Section~\ref{sec:third.sim.c2}, results for \RWGES{} without swap moves are shown in Table~\ref{table:c2.noswap}. Comparing Table~\ref{table:c2} with Table~\ref{table:c2.noswap}, one can see that using swap moves does improve the performance of \RWGES{}, especially when $d$ is small.

\begin{table}[!htp]
    \centering
    \begin{tabular}{|c|ccc|ccc|}
    \hline 
     & \multicolumn{3}{|c|}{$c_2 = 1.3$} & \multicolumn{3}{|c|}{$c_2 = 1.1 + 0.1 d$} \\
    \hline 
    d  & TPR (skeleton) & TPR & FPR &   TPR (skeleton) & TPR & FPR \\
    \hline 
   4 &   0.484 (0.01) & 0.271 (0.01) & 0.1 (0.004)   & 0.495 (0.01) & 0.299 (0.02) & 0.0961 (0.004)\\
   5 &   0.569 (0.01) & 0.304 (0.02) & 0.118 (0.004)   & 0.575 (0.02) & 0.338 (0.02) & 0.108 (0.005)\\
   6 &  0.649 (0.01) & 0.37 (0.02) & 0.123 (0.005)   & 0.647 (0.01) & 0.385 (0.02) & 0.107 (0.005)\\
  7 &   0.696 (0.01) & 0.404 (0.02) & 0.125 (0.006)   & 0.675 (0.01) & 0.398 (0.02) & 0.104 (0.006)\\
   8 &   0.707 (0.01) & 0.403 (0.02) & 0.136 (0.006)   & 0.691 (0.01) & 0.419 (0.02) & 0.0981 (0.006)\\
   9 &   0.718 (0.01) & 0.407 (0.02) & 0.14 (0.007)  & 0.692 (0.01) & 0.433 (0.021) & 0.0962 (0.006)\\
  \hline 
  \end{tabular}
   \caption{ Simulation study shown in Table~\ref{table:c2} without the swap proposal. 
   %Simulation study with $p = 20$, $n = 100$ and expected node degree $D = 4$. Results are averaged over 50 data sets. % and we run \RWGES{} for each data set for $40,000$ iterations.
   }
    \label{table:c2.noswap}
\end{table}

\clearpage 
\newpage 

\section{Discussion on the strong beta-min or faithfulness condition}\label{sec:discussion}
\setallcounters 
%\subsection{When the strong beta-min or faithfulness condition fails}\label{subsec:discussion}

In our theoretical results for structure learning, we assume that the true distribution of $\sX$, denoted by $\mu$, is faithful w.r.t. some DAG $G^*$, largely because we want to use the ``Chickering sequence'' argument from the consistency proof for GES~\citep{chickering2002optimal}.  
While for PC algorithms, it is known that the faithfulness assumption can be relaxed~\citep{zhang2008detection, adj.faith}, there seems no existing result on the consistency of GES beyond faithfulness; we can only find some related discussion in~\citet{zhang2017weakening}. 
The reason is that, without faithfulness, even if $n \rightarrow \infty$, GES and \RWGES{} can get trapped at some ``sub-optimal'' minimal I-map of $\mu$; that is, there may exist $\sigma, \tau \in \bbS^p$ such that $|G^\mu_\sigma| > |G^\mu_\tau|$ but any $\cE' \in \dels(\EG{ G^\mu_\sigma } )$ is not an I-map of $\mu$, where $G^\mu_\sigma$ denotes the minimal I-map of $\mu$ with ordering $\sigma$. An example is given in the proof of Theorem 2.4 in~\citet{raskutti2018learning} (though it was constructed for some other purpose).  %%% the adjacent faithfulness example in the paper ThreeFacesOfFaithful also works. 
From a Bayesian perspective,  faithfulness is probably not a restrictive assumption: if the true distribution of $\sX$ is generated randomly according to the SEM~\eqref{eq:sem} with $G = G^*$ and $\pi_0(B, \Omega \mid G)$ being continuous, then the distribution of $\sX$ is almost surely faithful to $G^*$; see Lemma~\ref{lm:chol.perfect} in the supplement.

In high-dimensional settings, much more of a concern is the strong beta-min condition or the strong faithfulness assumption, which  can often fail to hold in  reality~\citep{uhler2013geometry}. 
Here we describe a more flexible approach  which may help overcome this limitation.  Let $\bmin^2$ be the lower bound given in~\eqref{eq:beta.min} (the constant in the bound can be adjusted if necessary). 
For each $\sigma \in \bbS^p$,  define a DAG $\tilde{G}(\sigma)$ such that 
\begin{equation}\label{eq:def.tilde.G}
    i \rightarrow j \in \tilde{G}(\sigma) \text{ if and only if } (B^*_\sigma)_{ij}^2 \geq \bmin^2. 
\end{equation}
As long as the other nonzero entries of $B^*_\sigma$ are sufficiently small (so that their summed effects for each node have the same order as the noise), $\tilde{G}(\sigma)$ should coincide with the DAG $\hat{G}(\sigma)$ defined in~\eqref{eq:def.hatG}, and we can prove consistency and rapid mixing results for the DAG selection problem with ordering $\sigma$ by treating $\tilde{G}(\sigma)$ as the ``true'' model~\citep[cf.][]{yang2016computational}. 
All we need is just to replace $G^*_\sigma$ with $\tilde{G}(\sigma)$  in the construction of canonical paths. 
In some cases, we can prove the rapid mixing of \RWGES{} using this method. 
One extreme example is when all coefficients in the generating SEM  are sufficiently close to zero and then \RWGES{} always quickly moves to the null model. 
Below is a more interesting example. 

 %%% this example may not be good. Uhler et al.'s definition of strong faithfulness still applies here. 
\begin{example}\label{ex:unfaithful}
Let $p = 3$ and consider the following SEM for the random vector $\sX$,   
$$\sX_1 = Z_1, \quad \sX_2 = (1 + \epsilon) \sX_1 + Z_2, \quad \sX_3 =  \sX_1 - \sX_2 + Z_3,$$ 
where $Z_1, Z_2, Z_3$ are standard normal random variables. Denote the distribution of $\sX$ by $\mu$. If $\epsilon = 0$, one can easily show $1 \indep 3$ and $\mu$ is faithful to the DAG $\tilde{G}$ given by $1 \rightarrow 2 \leftarrow 3$. 
Now let $\epsilon$ be nonzero but very small. Then, $\mu$ is faithful to the complete DAG, but given a moderate sample size, the effective ``true'' DAG is still $\tilde{G}$ since the evidence for $1 \notindep 3$ is too weak to be detected (the posterior mass will not concentrate on the equivalence class of complete DAGs either). 
However, our path method can be used to prove the rapid mixing of \RWGES{}, $\tilde{G}(\sigma)$ defined in~\eqref{eq:def.tilde.G} just becomes the minimal I-map of $\tilde{G}$ in this case. 
\end{example}
 
 %%% approximately strong faithful 
%Assuming the strong beta-min condition or strong faithfulness simplifies the analysis in that we can use Chickering algorithm to show that the search can not be trapped at $\EG{\tilde{G}_\sigma}$. 
 
Nevertheless, in general, we cannot say much about the local posterior landscape at $\EG{\tilde{G}(\sigma)}$, if $\tilde{G}(\sigma)$ is not a minimal I-map of $G^*$.  
It is likely that for some $\sigma \in \bbS^p$,  $\EG{\tilde{G}(\sigma)}$ becomes a sub-optimal local mode. 
%in which case one may want to use methods such as parallel tempering (i.e., run multiple \RWGES{} samplers at different temperatures) to efficiently sample from the posterior.    
But the simulation study in Section~\ref{sec:second.sim.high.dim} suggests that our theory  still holds ``approximately'' in the sense that the chain can quickly find models that are not too different from the true one (i.e., the chain mixes rapidly in most parts of the space). 
In addition to parallel tempering, another strategy for overcoming multimodality is to choose a larger neighborhood than the one used in \RWGES{} so that when the strong beta-min condition fails, we can 
%For other search methods which consider a smaller neighborhood set than $\nbg(\cdot)$, the  multimodality of the posterior distribution can only be more severe.  
modify the construction of the canonical path from $\tilde{G}(\sigma)$ to some ``best'' DAG $\tilde{G}$. 
In this regard, the DAG MCMC samplers proposed in~\citet{grzegorczyk2008improving} and~\citet{su2016improving} may be very useful. They are equipped with proposal moves that are much more complex than single-edge modifications, which enable the samplers to jump between DAGs that encode very different CI relations.   
Similar ideas might be used to improve other sampling methods, including those defined on the space of equivalence classes. 
  
\bibliographystyle{plainnat} % Style BST file (imsart-number.bst or imsart-nameyear.bst)
\bibliography{reference}

\end{document}